\documentclass[acmtog,authorversion,nonacm,table]{acmart}

\usepackage{booktabs} 

\setcitestyle{nosort,square} 

\usepackage{picinpar,moresize,xfrac,graphpap,dcolumn,wrapfig,graphicx}

\microtypesetup{stretch=30,shrink=30}
\DeclareGraphicsExtensions{.png,.pdf,.jpg}
\usepackage{subcaption}
\usepackage{stackengine}

\usepackage{tikz,pgfplots}
\usepackage{tikz-cd}
\usepackage{pgf}
\usepgfplotslibrary{colormaps}
\usetikzlibrary{pgfplots.colormaps}
\usetikzlibrary{arrows,automata,positioning,backgrounds}
\usepackage{rotating}
\pgfplotsset{compat=newest}
\pgfplotsset{plot coordinates/math parser=false}
\newlength\figureheight
\newlength\figurewidth

\usepackage{soul,array,calc,url,ragged2e,graphpap}
\usepackage{booktabs} 
\usepackage{tabularx}
\usepackage{colortbl}
\usepackage{multirow}
\usepackage{hyphenat}
\usepackage{transparent}
\usepackage{enumerate}
\usepackage{mathrsfs}
\usepackage{mdframed}
\usepackage{acro}

\usepackage[figure]{hypcap}
\usepackage{mathtools}
\mathtoolsset{centercolon}
\usepackage{nicefrac}
\usepackage{units}
\usepackage[normalem]{ulem}
\usepackage{cancel}
\usepackage{pbox}
\usepackage{multicol}
\usepackage{cases}
\usepackage[all]{xy}

\usepackage{adjustbox}
\usepackage{wrapfig}
\AfterEndEnvironment{wrapfigure}{\setlength{\intextsep}{0mm}}

\usepackage{cleveref}
\crefmultiformat{subequation}%
{\edef\crefstripprefixinfo{#1}(#2#1#3}%
{,#2\crefstripprefix{\crefstripprefixinfo}{#1}#3)}%
{,#2\crefstripprefix{\crefstripprefixinfo}{#1}#3}%
{,#2\crefstripprefix{\crefstripprefixinfo}{#1}#3)}

\makeatletter
\DeclareFontFamily{OMX}{MnSymbolE}{}
\DeclareSymbolFont{MnLargeSymbols}{OMX}{MnSymbolE}{m}{n}
\SetSymbolFont{MnLargeSymbols}{bold}{OMX}{MnSymbolE}{b}{n}
\DeclareFontShape{OMX}{MnSymbolE}{m}{n}{
    <-6>  MnSymbolE5
   <6-7>  MnSymbolE6
   <7-8>  MnSymbolE7
   <8-9>  MnSymbolE8
   <9-10> MnSymbolE9
  <10-12> MnSymbolE10
  <12->   MnSymbolE12
}{}
\DeclareFontShape{OMX}{MnSymbolE}{b}{n}{
    <-6>  MnSymbolE-Bold5
   <6-7>  MnSymbolE-Bold6
   <7-8>  MnSymbolE-Bold7
   <8-9>  MnSymbolE-Bold8
   <9-10> MnSymbolE-Bold9
  <10-12> MnSymbolE-Bold10
  <12->   MnSymbolE-Bold12
}{}
\let\llangle\@undefined
\let\rrangle\@undefined
\DeclareMathDelimiter{\llangle}{\mathopen}%
                     {MnLargeSymbols}{'164}{MnLargeSymbols}{'164}
\DeclareMathDelimiter{\rrangle}{\mathclose}%
                     {MnLargeSymbols}{'171}{MnLargeSymbols}{'171}
\makeatother

\usepackage{fancyvrb,listings}
\usepackage{algorithm}
\usepackage{algorithmicx}
\usepackage[noend]{algpseudocode}

\algrenewcommand\alglinenumber[1]{\footnotesize #1:}
\makeatletter  
 \renewcommand{\ALG@name}{\small Algorithm} 
\makeatother 
\algdef{SE}[DOWHILE]{Do}{doWhile}{\algorithmicdo}[1]{\algorithmicwhile\ #1}%

\newtheoremstyle{mine}{3pt}{3pt}{\itshape}{}{\bfseries}{.}{.5em}{}
\theoremstyle{mine}
\newtheorem{theorem}{Theorem}[section]

\newtheorem{lemma}[theorem]{Lemma}
\newtheorem{proposition}[theorem]{Proposition}
\newtheorem{corollary}[theorem]{Corollary}
\newtheorem{definition}[theorem]{Definition}
\newtheorem{example}[theorem]{Example}

\newtheorem{remark}[theorem]{Remark}

\newcommand{\figref}[1]{\textup{Fig.~\ref{#1}}}
\newcommand{\tabref}[1]{\textup{Table~\ref{#1}}}
\newcommand{\teqref}[1]{\textup{Eq.~(\ref{#1})}}

\newcommand{\secref}[1]{\textup{Section~\ref{#1}}}
\newcommand{\appref}[1]{\textup{Appendix~\ref{#1}}}
\renewcommand{\algref}[1]{\textup{Alg.~\ref{#1}}}
\newcommand{\thmref}[1]{\textup{Theorem~\ref{#1}}}
\newcommand{\lemref}[1]{\textup{Lemma~\ref{#1}}}

\newcommand{\corref}[1]{\textup{Corollary~\ref{#1}}}

\def\cf{\emph{cf.}}
\def\ie{\emph{i.e.}}

\def\NN{\mathbb{N}}

\def\RR{\mathbb{R}}

\def\ZZ{\mathbb{Z}}

\def\cC{\mathcal{C}}

\def\cI{\mathcal{I}}

\def\cM{\mathcal{M}}

\def\cR{\mathcal{R}}

\def\sfA{\mathsf{A}}

\def\sfE{\mathsf{E}}

\def\sfH{\mathsf{H}}

\def\sfJ{\mathsf{J}}

\def\sfL{\mathsf{L}}
\def\sfM{\mathsf{M}}

\def\sfV{\mathsf{V}}

\DeclareMathAlphabet{\mathpzc}{OT1}{pzc}{m}{it}
\def\pzcA{\mathpzc{A}}
\def\pzcB{\mathpzc{B}}

\def\pzcL{\mathpzc{L}}

\def\pzcV{\mathpzc{V}}

\def\fg{\mathfrak{g}}
\def\fz{\mathfrak{z}}

\usepackage{bbm}
\DeclareSymbolFont{bbold}{U}{bbold}{m}{n}
\DeclareSymbolFontAlphabet{\mathbbold}{bbold}

\def\det{\operatorname{det}}

\def\im {\operatorname{im}}
\def\ker{\operatorname{ker}}

\def\id{\operatorname{id}}
\newcommand{\Id}{\mathrm{Id}}

\DeclareMathOperator*{\argmin}{argmin}

\DeclareMathOperator{\SO}{SO}
\DeclareMathOperator{\SE}{SE}

\newcommand{\se}{\mathfrak{se}} 

\DeclarePairedDelimiterX\braket[2]{\langle}{\rangle}{#1\,\delimsize\vert\,\mathopen{}#2}

\definecolor{b1}{rgb}{0.158099,0.313781,0.636957}
\definecolor{b2}{rgb}{0.525367,0.691857,0.998936}
\definecolor{g1}{rgb}{0.256000,0.640000,0.576000}
\definecolor{g2}{rgb}{0.559573,0.800781,0.760580}
\definecolor{p1}{rgb}{0.416000,0.192000,0.640000}
\definecolor{p2}{rgb}{0.680880,0.561880,0.799881}
\definecolor{r1}{rgb}{0.800000,0.200000,0.000000}
\definecolor{r2}{rgb}{1.000000,0.702595,0.603461}
\definecolor{k1}{gray}{0}
\definecolor{k2}{gray}{0.7}

\def\OmegaMW{\Omega^{\scriptscriptstyle MW}}
\def\ThetaMW{\Theta^{\scriptscriptstyle MW}}
\def\OmegaMWD{\Omega^{\scriptscriptstyle\sf MW}}

\setcopyright{rightsretained}
\acmJournal{TOG}
\acmYear{2026} \acmVolume{0} \acmNumber{0} \acmArticle{0} \acmMonth{0}\acmDOI{0}
\makeatletter
\let\@authorsaddresses\@empty
\makeatother
\begin{document}
\title{Elastic Curves via Geometric Mechanics}

\author{Oliver Gross}
\authornote{These authors contributed equally.}
\author{Rohit Jammula}
\authornotemark[1]
\author{Albert Chern}
 \affiliation{%
   \institution{University of California San Diego}
   \streetaddress{9500 Gilman Dr, MC 0404}
   \city{La Jolla}
   \state{CA}
   \postcode{92093}
   \country{USA}
 }
\email{rjammula@ucsd.edu, ogross@ucsd.edu, alchern@ucsd.edu}

\newcommand{\AC}[1]{{\footnotesize\color[rgb]{0.2,0.4,0.8}\textbf{AC:} #1}}
\newcommand{\RJ}[1]{{\footnotesize\color[RGB]{17,0,255}\textbf{RJ} #1}}
\newcommand{\OG}[1]{{\footnotesize\color[RGB]{219, 48, 122}\textbf{OG:} #1}}

\begin{abstract}
Elastic curves are the mathematical shapes of thin elastic rods in equilibrium, with deep connections to mechanics, geometry, and computer graphics. Traditionally described as stationary points of bending energy under length and torsion constraints, their rich theory admits many equivalent characterizations. We develop a new one from the viewpoint of geometric mechanics.
Our main contribution relies on a lesser-known \emph{isoperimetric characterization}: a curve is elastic if and only if it is a critical point of the length functional under fixed area and volume vectors. We show that these constraints transform naturally under orientation-preserving rigid body motions, identifying them as momentum variables for these symmetries.
This structure suggests a new discrete theory. We show that the low-order integral quantities length, area, and volume vectors are all naturally defined for polygonal curves, leaving the same transformation laws  exactly satisfied. The resulting definition of discrete elastic curves in terms of the isoperimetric characterization restricted to discrete polygonal curves is variational, structure-preserving, and requires no auxiliary discretizations of curvature or material frames.
Finally, the same structure carries the \emph{Marsden--Weinstein} form, a canonical (pre-)symplectic structure on the space of curves, to polygonal curves. This yields novel approaches to Hamiltonian dynamics on discrete space curves, including tangent, vortex-filament, and modified Korteweg--de Vries flows.
\end{abstract}


\begin{teaserfigure}
    \includegraphics[width=\textwidth]{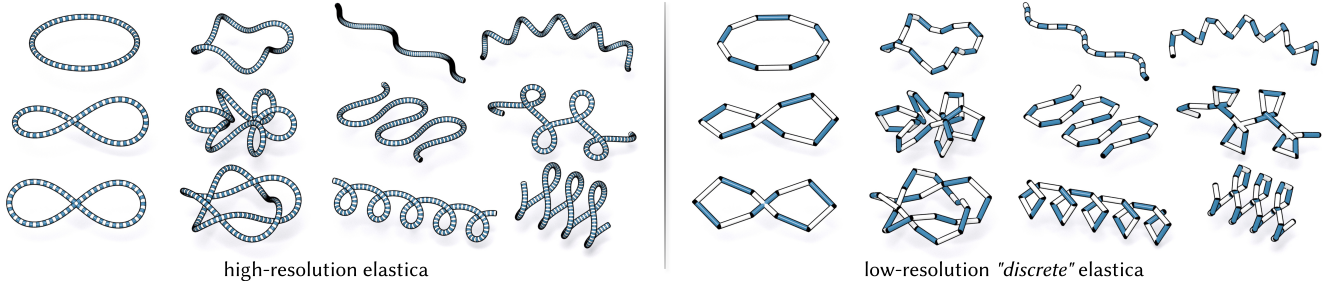}
    \caption{We compute elastic curves from an \emph{isoperimetric characterization} of smooth elastica that carries over to polygonal curves without requiring auxiliary discretizations of curvature or material frames. The resulting lower-order variational problem can be optimized directly over polygon vertices, is better conditioned, and produces consistent curves across resolutions, providing a canonical notion of \emph{discrete elastica} that retains connections to dynamical systems of the smooth theory.
    }
    \label{fig:Teaser}
\end{teaserfigure}

\maketitle

\section{Introduction}
\emph{Elastic curves}, also known as \emph{elastica}, are the shapes taken by thin elastic wires in static equilibrium.  The study of elastic curves dates back to the early days of the calculus of variations, beginning with the classical work of Bernoulli and Euler~\cite{Euler:1744:MIL, Levien:2008:EMH}.  They continue to play central roles in modeling thin elastic materials, such as in plant growth, DNA mechanics, hair simulation, elastic and magnetic knots, deployable structures and fabrication-aware design~\cite{Goriely:1998:SHR,Benham:2005:DNA,
Bertails:2006:SPH,Ricca:2018:GES,Hafner:2021:DSP,Hafner:2023:DSK,
Panetta:2019:XSH,Vidulis:2023:CEM,Dandy:2024:TCE}.

While they are traditionally characterized as stationary points of the bending energy 
\[
\tfrac{1}{2}\int \kappa^2\,ds,
\]
where \(\kappa\) is the curvature, and \(ds\) is the arc-length measure of the
curve, possibly under constraints such as fixed length and total torsion, there
are unusually many equivalent characterizations of the same curves. Kirchhoff's
analogy interprets elastic space curves as spinning-top dynamics of their
tangent indicatrices~\cite{Kirchhoff:1859:UGB}, which can be expressed as a
pendulum equation~\cite{Pinkall:2024:DGF}. They also arise as solitons of the
localized induction equation, an integrable curve flow related by the Hasimoto
transform to the nonlinear Schr\"odinger equation~\cite{Hasimoto:1971:MVF,
Hasimoto:1972:SVF,Langer:1984:TSC,LangerSinger:1996:LAK,
Singer:2008:LEC,Pinkall:2024:DGF}. Elastic curves also provide a classical bridge to surface theory, where cylinders over free elastica and their M\"obius images give Willmore surfaces, with related constructions yielding constrained Willmore surfaces~\cite{Pinkall:1985:HTS,Bohle:2008:CWS,Heller2014SpaceForms,Pinkall:2024:DGF}. Notably, a number of these results
generalize to elastic curves in arbitrary dimensions~\cite{Tjaden:1991:EEK, Pinkall:2024:DGF}, space forms with curvature~\cite{Langer:1984:TSC}, or non-constant bending stiffness, as in rods of
variable thickness~\cite{Hafner:2021:DSP,Hafner:2023:DSK,Gross:2025:ECV}.

While these perspectives coincide in the continuous theory, this coincidence is no longer automatic when curves are replaced by polygons, which have well-defined edge lengths and turning angles, but no canonical curvature. From the viewpoint of \emph{Discrete Differential Geometry (DDG)}, one defines discrete quantities and identities that remain meaningful and can be evaluated on polygonal curves~\cite{Bobenko:2008:DDG,Wardetzky:2007:DQC}. For elastic polygons, existing discretizations do so by choosing one of the equivalent continuous characterizations as the basis for the definition.

This choice leads to two broad DDG approaches to elastica. The first takes the local variational characterization as its starting point, defining discrete curvature, frame, and strain quantities directly on polygons. A particularly influential model was proposed by \citet{Bergou:2008:DER}: discrete elastic rods discretize the local framed-rod picture through adapted frames, parallel transport, holonomy, bending, and torsion, thereby defining a discrete elastic energy for bendable rods. Related framed-rod and Cosserat models extend this viewpoint to stretch, shear, dynamics, contact, and strand simulation~\cite{Pai:2002:SIS,Kugelstadt:2016:POB,Gazzola:2018:FIP,Bertails:2006:SPH,DerouetJourdan:2010:SID,Casati:2013:SSC,DerouetJourdan:2013:IDH,Shi:2023:LCM,Daviet:2023:IHS,Hsu:2025:SCR}, with applications in hair simulation, inverse design, and fabrication~\cite{Panetta:2019:XSH,Pillwein:2020:EGG,Pillwein:2021:GDE,Ren:2022:UME,Becker:2023:CSH,Vidulis:2023:CEM,Dandy:2024:TCE,Hafner:2021:DSP,Hafner:2023:DSK}.

\begin{figure}[h]
    \centering
    \includegraphics[width = \columnwidth]{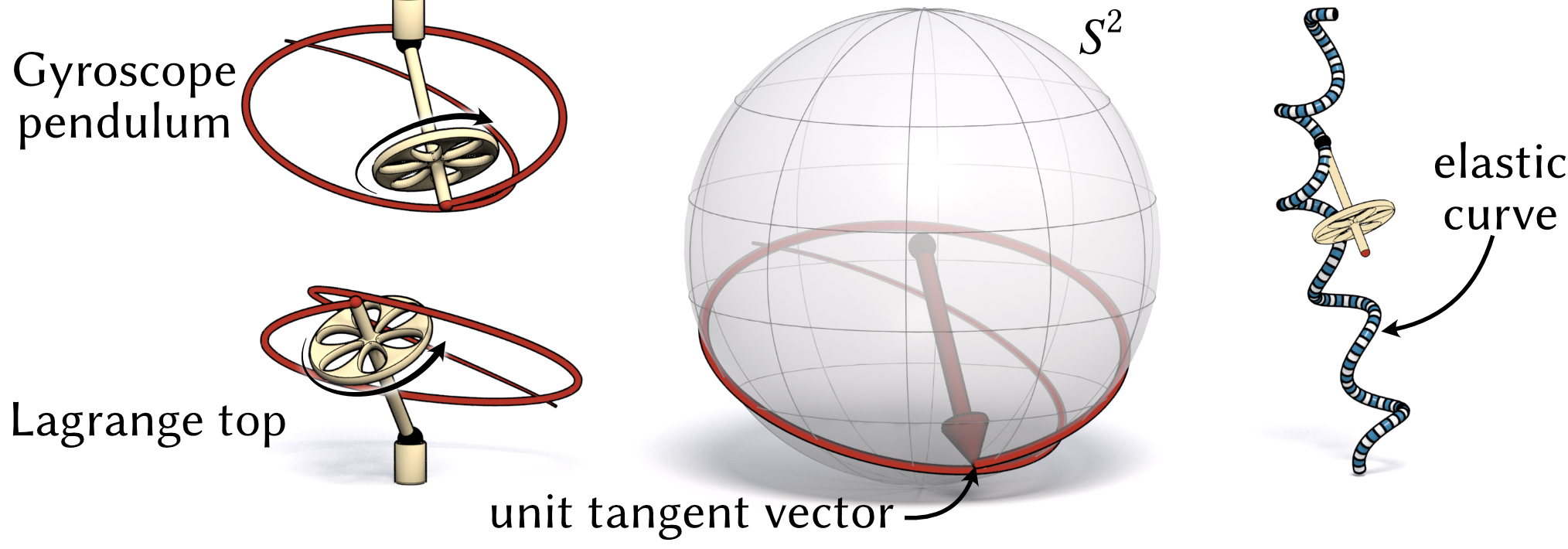}
    \caption{\emph{Kirchhoff's analogy} identifies the tangent indicatrix of an elastic space curve with the motion of a Lagrange top (equivalently, a symmetric gyroscope under gravity), whose reduced dynamics are governed by a pendulum equation.}
    \label{fig:ElasticFromPendulum}
\end{figure}
A different line preserves the integrable-system characterizations instead. The discrete integrable geometry of curves and lattices was developed in this direction by~\citet{Doliwa:1999:GDC} and provides semi-discrete analogues of curve flows related to the nonlinear Schr\"odinger and localized-induction hierarchies. Building on this viewpoint, \citet{Bobenko:2000:DTL} introduced an integrable variational discretization of elastic curves through a discrete Lagrange-top formulation (\figref{fig:ElasticFromPendulum}). \citet{Pinkall:2007:DSR} later constructed a doubly discrete analogue of the smoke-ring flow. Hoffmann and collaborators developed related constructions through discrete Hasimoto flows, B\"acklund transformations, and skew parallelogram nets, characterizing discrete elastic and constrained elastic curves as invariant curves of integrable discrete curve evolutions~\cite{Hoffmann:2008:DHS,Hoffmann:2025:SPN, Hoffmann:2026:DCT}. 
These constructions preserve the integrable-system character of elastica, but define discrete elastica only implicitly through the particular equations or flows they discretize. The resulting notion is therefore inherited from a chosen model of the smooth dynamics rather than posed intrinsically on polygon space, and it does not yield a geometric problem that can be evaluated and optimized directly on polygon vertices.
 
In this paper, we propose a new formulation for discrete elastic curves that takes a third route. Rather than improving a local curvature, frame, or strain model, and rather than discretizing a particular integrable system, we start from a lesser-known and only recently discovered variational characterization of smooth elastica~\cite{Chern:2020:CHF}: 
\begin{quote}
	\emph{Elastic curves are critical points of length under constraints on their \emph{area} and \emph{volume vectors}.}
\end{quote}
Here, components of area and volume vectors are respectively the projected areas along coordinate axes and the swept volumes rotating about coordinate axes.

Smooth elastica therefore admit an isoperimetric characterization that lowers the order of the variational problem. Whereas bending energy depends on second derivatives and leads to a fourth-order Euler--Lagrange equation, the isoperimetric formulation uses only the first-order quantities length, area, and volume and yields second-order stationary conditions. Unlike bending energy, these quantities restrict canonically to polygons without choosing discrete curvature, frames, or strains. This gives a new notion of discrete elastica with better conditioning and consistent behavior under refinement that retains the geometric structure of the smooth theory.

We show that the area and volume vectors have a kinematic meaning in geometric mechanics: for closed curves, they are the \emph{linear} and \emph{angular momenta} associated with translational and rotational symmetries~\cite{Marsden:1983:COV,Shashikanth:2003:LVR}. 
Building on this viewpoint, we place the isoperimetric characterization of elastica on a systematic footing, and generalize it to quasi-periodic curves, whose geometry repeats up to a prescribed rigid motion, called the \emph{monodromy}. For each monodromy, we identify the compatible area-volume data as the \emph{Noether charges} associated with the residual rigid-motion symmetries. The same first-principles derivation carries over to polygonal curves and yields their discrete Noether charges. Building on this general discrete isoperimetric theory, we develop a numerical method for computing discrete elastica.

As a further consequence, our geometric-mechanics formulation
provides a \emph{pre-symplectic structure} for discrete polygonal curves. In the continuous setting, this structure is given by a closed \(2\)-form known as the \emph{Marsden--Weinstein form}. Its kernel consists of reparametrizations, and it therefore descends to a genuine symplectic form after quotienting by reparametrizations. The Marsden--Weinstein form underlies various Hamiltonian flows, including the aforementioned integrable dynamical system of \emph{vortex-filament} motion and the higher-order \emph{modified Korteweg--de Vries (mKdV)} flows in the same hierarchy, while reparametrizations appear as its degenerate directions. 

With the new discrete pre-symplectic structure arising from our discrete theory, we obtain structure-driven discrete analogues of tangent, vortex-filament, and mKdV flows that retain the relevant conservation laws up to a small residual, as demonstrated through comparisons with local and integrable discretizations.

A key benefit of this route is that it does not force us to choose one
classical characterization of elastica at the expense of the others. The
isoperimetric formulation gives the variational definition, the momentum
interpretation keeps the rigid-body-mechanics meaning, and the
Marsden--Weinstein form retains the connection to Hamiltonian curve
flows and the localized-induction hierarchy.

Remaining closely tied to classical results, this new discrete theory opens up numerous opportunities by enabling unexplored methods for analysis and simulation of related concepts. Specifically, we make the following contributions:
\begin{itemize}
    \item We give a systematic geometric-mechanics treatment of quasi-periodic elastica with monodromy \(g\in\SE(3)\), identify the compatible rigid-motion symmetries, and show that the corresponding area-volume components are their Noether momenta.

    \item We provide a structure-preserving discretization for
    quasi-periodic polygonal curves: discrete area and volume, together with
    the Marsden--Weinstein form restricted to polygonal curves, satisfy the
    same transformation laws exactly.

    \item We define discrete elastic curves as critical points of
    length at fixed momentum on the space of discrete polygonal curves, yielding a lower-order, curvature-free variational characterization.

    \item We use the discrete Marsden--Weinstein form to define Hamiltonian flows on polygonal curves, yielding structure-driven analogues of tangent, vortex-filament, and mKdV flows.
\end{itemize}

\paragraph{Outline} The remainder of the paper is organized as follows: \secref{sec:SpaceCurves} reviews the classical variational theory of elastic curves together with equivalent characterizations, including the isoperimetric formulation that motivates our approach. In \secref{sec:GeomMech}, we develop the
geometric-mechanics framework and identify the area and volume data as the momentum map comprising the Noether charges associated with rigid-motion symmetries. In
\secref{sec:Discretization}, we transfer these structures to polygonal curves
and introduce our main discrete formulation, defining discrete elastica as
critical points of length at fixed discrete momentum. In \secref{sec:NumericalOptimization}, we present a simple numerical method to compute discrete elastica with prescribed monodromy and momenta. In
\secref{sec:HamiltonianFlowsOnDiscreteSpaceCurves}, we use the discrete Marsden--Weinstein
structure to construct tangent, vortex-filament, and modified
Korteweg--de Vries flows. Finally, in \secref{sec:ResultsValidation}, we evaluate the discrete formulation numerically and validate its convergence, conservation properties, and agreement with the continuous theory.

\section{Elastic Curves}
\label{sec:SpaceCurves}
Let $\cM$ denote the space of smooth curves $\gamma\colon I \to \RR^3$,
where $I$ is either a compact interval $[0,1]$, the circle $S^1$, or
the real line $\RR$. A smooth \emph{infinitesimal deformation}, or \emph{variation}
(\emph{virtual velocity}), of \(\gamma\) is a vector field \(\mathring\gamma\)
along \(\gamma\). The \emph{tangent space}\footnote{We distinguish between \(\gamma\) as an abstract element of the
function space \(\cM\) and \(\vec\gamma\colon I\to\RR^3\) as the
corresponding vector-valued function on \(I\). Thus derived quantities,
such as \(\vec\gamma'\), \(\vec\gamma''\), and higher derivatives, and
variations, such as \(\dot{\vec\gamma}\) and \(\mathring{\vec\gamma}\),
are also functions on \(I\). For example, \(\dot{\vec\gamma}\colon
I\to\RR^3\) is evaluated at \(t\in I\) as
\(\dot{\vec\gamma}(t)\in\RR^3\).}
\begin{equation*}
    T_\gamma\cM = \{\mathring{\vec\gamma}\colon I\to\RR^3\}
\end{equation*}
consists of all such fields, \ie, smooth maps \(\mathring{\vec\gamma}\colon I\to\RR^3\).
A curve $\gamma\in\cM$ is \emph{arc-length parametrized} if
$|\vec\gamma'|=1$, so that the parameter $t\in I$ measures distance
along $\gamma$ directly. We work throughout with arc-length
parametrized curves and write $(\cdot)'= \frac{d}{ds}$ for the arc-length
derivative.

\subsection{Geometric functionals on \(\cM\)}
\label{sec:GeometricFunctionals} 
To set up the variational definition of elastic curves, we first recall a few basic geometric functionals on \(\cM\). These are quantities that assign a number to each curve, independently of its parametrization and of the chosen
Euclidean frame of \(\RR^3\). Such functionals often appear as terms in elastic potential energy models for materials whose centerline is described by a curve. 

\subsubsection{Length and Bending Energy}
The simplest geometric quantity associated to a curve $\gamma$ is its
\emph{length}
\begin{equation}
    \label{eq:LengthEnergy}
    \pzcL(\gamma) = \int_I ds.
\end{equation}
The next is curvature: the unit tangent $T = \vec\gamma'$ traces a
path (the \emph{tangent indicatrix}) on the unit sphere $S^2\subset\RR^3$, and its rate of change $\vec\gamma'' = T'$ measures
how sharply the curve bends. The {\emph{bending energy}\hfill\text{}}
\begin{wrapfigure}[7]{r}{0.4\columnwidth}
    \centering
    \vspace{-1.4\baselineskip}
    {\hspace{-2.em}\includegraphics[width=0.45\columnwidth]{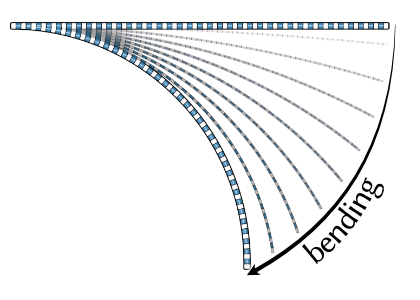}}
\end{wrapfigure}
\begin{equation}
    \label{eq:BendingEnergy}
\displaywidth=\parshapelength\numexpr\prevgraf+2\relax
    \pzcB(\gamma) = \tfrac{1}{2}\int_I |\vec\gamma''|^2\,ds
\end{equation}
integrates its squared norm over the curve and models the elastic energy stored in
a bent inextensible rod~\cite{Antman:2005:NPE,Audoly:2010:EAG}.

We are interested in stationary points of \(\teqref{eq:BendingEnergy}\).
Without further constraints these are free elastica. If length is imposed as
a constraint, the Euler--Lagrange equation acquires a Lagrange multiplier
term corresponding to tension~\cite{Pinkall:2024:DGF}.

\subsubsection{Holonomy and total torsion}
Both length and bending energy are \emph{geometric invariants} of the curve, \ie, they are unchanged under reparametrization and under changes of frame of \(\RR^3\). For space curves, there is another geometric invariant measuring ``twisting.'' As one moves along the curve, normal directions can be parallel transported, and comparing the initial and final normal directions gives a net rotation angle. This angle, the \emph{holonomy of the normal bundle}, equivalently \emph{total torsion} modulo \(2\pi\), provides the torsional constraint in the classical theory of space elastica~\cite{Bishop:1975:MFC}.

\begin{wrapfigure}[9]{r}{0.35\columnwidth}
  {\hspace{-2.em}
    \includegraphics[width=0.4\columnwidth]{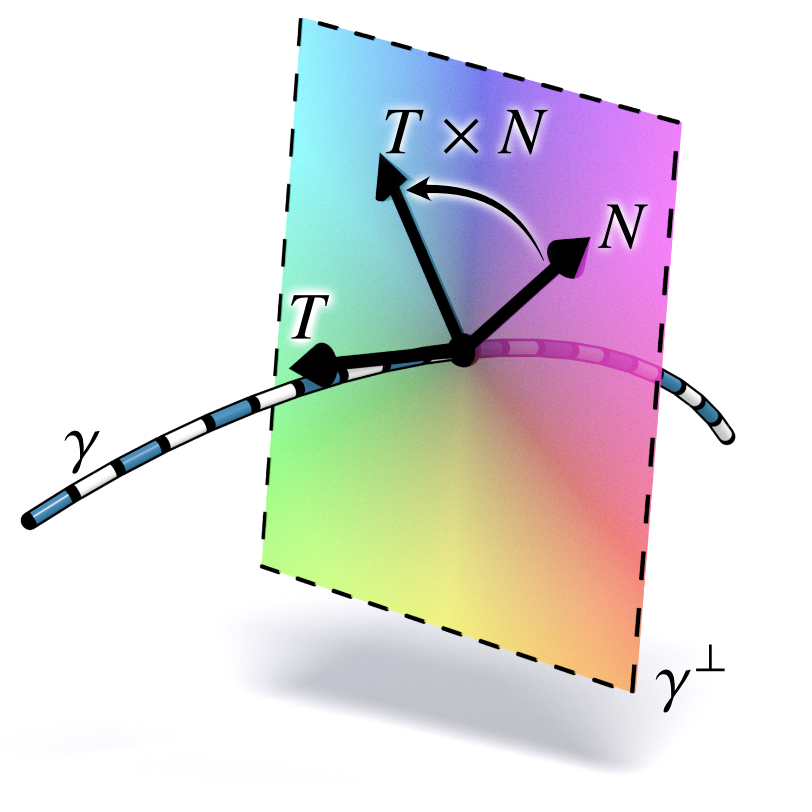}
  }  
\end{wrapfigure}
On each normal plane, the map \(T\times\cdot\) acts as rotation by
\(90^\circ\). Thus it plays the role of multiplication by \(i\), and defines a \emph{complex structure} on the normal bundle. This complex structure is preserved by parallel transport~\cite[Sec.~3.2]{Chern:2020:CHF}. 

Given a pair of reference unit normal vectors \(W=(W_0,W_1)\) at \(\vec\gamma(0)\) and \(\vec\gamma(1)\) the \emph{holonomy} is the unique angle
\(\Phi_W\in\mathbb{R}/2\pi\mathbb{Z}\) such that
\begin{equation}
    \label{eq:HolonomyDefiningEquation}
    Z_1 = e^{-i\Phi_W}W_1 = \cos(\Phi_W)\,W_1 - \sin(\Phi_W)\,T_1\times W_1,
\end{equation}
where \(Z_1\) is the parallel transport of \(W_0\) along \(\gamma\). In the closed case \(I=S^1\) one takes \(W_0=W_1\). 
For open curves, \(\Phi_W\) depends on the choice of endpoint normals
\(W_0\) and \(W_1\). Changing these normals changes only the reference angle
between the initial and final normal planes. The variational gradient of
\(\Phi_W\) is independent of this choice~\cite[Thm.~5.3]{Pinkall:2024:DGF}.

\begin{remark}
Much of the literature uses \emph{total torsion}
\begin{equation*}
    \mathpzc{T}(\gamma,N) = \int_I \langle N', T\times N\rangle\, ds
\end{equation*}
of a framed curve \((\gamma, N)\) in place of holonomy. While their variational gradients coincide~\cite{Pinkall:2024:DGF}, the total torsion is intrinsic to the curve only modulo \(2\pi\), which is precisely what the holonomy \(\Phi_W\in\mathbb{R}/2\pi\mathbb{Z}\) captures.
\end{remark}

\subsection{Characterizations of elastic curves}
\label{sec:ElasticCurves}
We are now in a position to state the central definition of this section. A curve
\(\gamma \in \cM\) is called \emph{elastic} if it is a critical point of the bending
energy \(\pzcB(\gamma)\) (\teqref{eq:BendingEnergy}) under the constraints of fixed
length \(\pzcL_0\) (\teqref{eq:LengthEnergy}) and holonomy \(\Phi_0\)
(\teqref{eq:HolonomyDefiningEquation}). The \(L^2\)-gradients of the bending energy, length, and holonomy of an arc-length parametrized curve $\gamma\in\mathcal{M}$ under perturbations with compact support in the interior of \(I\) are 
\begin{align*}
\nabla\pzcB
&=
\vec{\gamma}''''+\tfrac{3}{2}|\vec{\gamma}''|^2\vec{\gamma}''
-\langle\vec{\gamma}'''',\vec{\gamma}'\rangle\vec{\gamma}',
\\
\nabla\pzcL&=-\vec{\gamma}'',
\\
\nabla\Phi_W&=-\vec{\gamma}'\times\vec{\gamma}'''.
\end{align*}
Consequently, stationarity of \(\pzcB+\lambda_1\pzcL+\lambda_2\Phi_W\) yields the Euler--Lagrange equation
\begin{equation}
    \label{eq:ElasticELEqClassic}
    \vec \gamma'''' + \tfrac{3}{2}\langle\vec \gamma'',\vec \gamma''\rangle\vec \gamma''
    - \langle\vec \gamma'''',\vec \gamma'\rangle\vec \gamma'
    - \lambda_1\vec \gamma''
    - \lambda_2\,\vec \gamma'\times\vec \gamma'''
    = 0,
\end{equation}
for Lagrange multipliers \(\lambda_1,\lambda_2\in\mathbb{R}\).

\subsubsection{Equivalent characterizations and physical interpretations}
\label{sec:EquivalentCharacterizations}
Elastic curves admit several equivalent
characterizations~\cite{Pinkall:2024:DGF} connecting them to mathematical physics and integrable systems, and, as we show in \secref{sec:MWForm}, to symplectic geometry.

Classical examples include \emph{Kirchhoff's analogy} (\figref{fig:ElasticFromPendulum}), which identifies elastic space curves with the motion of a spinning top~\cite{Kirchhoff:1859:UGB, Bobenko:2000:DTL}. In this description, the tangent indicatrix carries the reduced dynamics and satisfies a pendulum equation.

Another classical characterization is closely tied to vortex-filament dynamics.
The elastic Euler--Lagrange~\teqref{eq:ElasticELEqClassic} can be written as
\begin{equation*}
    \tfrac{d}{ds}
    (
        \vec\gamma'''
        +
        \tfrac32|\vec\gamma''|^2\vec\gamma'
        -
        \lambda_1\vec\gamma'
        -
        \lambda_2\vec\gamma'\times\vec\gamma''
    )
    =
    0 .
\end{equation*}
Hence there is a constant vector \(\vec c_1\in\RR^3\) such that
\begin{equation*}
    \vec\gamma'''
    +
    \tfrac32|\vec\gamma''|^2\vec\gamma'
    -
    \lambda_1\vec\gamma'
    -
    \lambda_2\vec\gamma'\times\vec\gamma''
    =
    -\vec c_1 .
\end{equation*}
Crossing this equation with \(\vec\gamma'\) from the left gives
\begin{equation*}
    \tfrac{d}{ds}
    (
        \vec\gamma'\times\vec\gamma''
        +
        \lambda_2\vec\gamma'
    )
    =
    \vec c_1\times\vec\gamma' .
\end{equation*}
Integrating once more shows that, for some constant vector
\(\vec c_2\in\RR^3\) and some \(\lambda_3\in\RR\),
\begin{equation}
    \label{eq:SRFtoElasticRelation}
    \vec\gamma'\times\vec\gamma''
    =
    \vec c_1\times\vec\gamma+\vec c_2+\lambda_3\,\vec\gamma' .
\end{equation}
Conversely, differentiating this relation and reversing the two steps recovers the elastic Euler--Lagrange equation, so that the two characterizations~\teqref{eq:ElasticELEqClassic} and~\teqref{eq:SRFtoElasticRelation} are indeed equivalent. The left-hand side is the velocity field of the \emph{localized induction equation}
\begin{equation}
    \label{eq:LocalizedInductionEquation}
    \tfrac{\partial}{\partial t}\vec \gamma
    = \vec \gamma'\times\vec \gamma'',
\end{equation}
which governs the evolution of thin vortex filaments and is also known as the
\emph{vortex-filament flow}\footnote{This flow is also known as smoke-ring
flow, Hasimoto flow, Heisenberg flow, or binormal curvature flow.}. 
This flow belongs to a family of mutually commuting curve
evolutions---the \emph{localized-induction hierarchy}---meaning that applying two flows from the family in either order gives the same result. Under the \emph{Hasimoto transform}, this family corresponds to the \emph{nonlinear Schr\"odinger (NLS) hierarchy}. This integrable structure supplies many related curve flows with shared conservation laws~\cite{Chern:2020:CHF}.

Thus, \(\teqref{eq:SRFtoElasticRelation}\) says
that elastic curves are precisely those curves for which the localized
induction flow reduces to an infinitesimal rigid motion (\figref{fig:RigidMotionToElastica}), up to the tangential
repara\-metrization term. 
This viewpoint goes back to~\citet{DaRios:1906:SMD} and~\citet{LeviCivita:1908:SAE}, was rediscovered
repeatedly~\cite{Ricca:1991:RDR}, and was later related by~\citet{Hasimoto:1971:MVF} to the nonlinear Schrödinger equation. It has also
motivated discrete theories of elastic curves based on discrete analogues of
the vortex-filament flow~\cite{Pinkall:2007:DSR, Hoffmann:2008:DHS}. In these
theories, discrete elastica are polygons that evolve by rigid body motions
under the chosen discrete flow. Since the continuous velocity is the
curvature-binormal vector \(\vec\gamma'\times\vec\gamma''\), this approach
requires a choice of how curvature, normal directions, and binormal motion
are represented on a polygon. Our isoperimetric formulation avoids this
choice by using only length, area, and volume.
The doubly discrete
Hasimoto flow, for example, provides a discrete analogue of the smoke-ring
flow, although its variational interpretation is not yet understood.

\begin{figure}[t]
    \centering
    \includegraphics[width=\columnwidth]{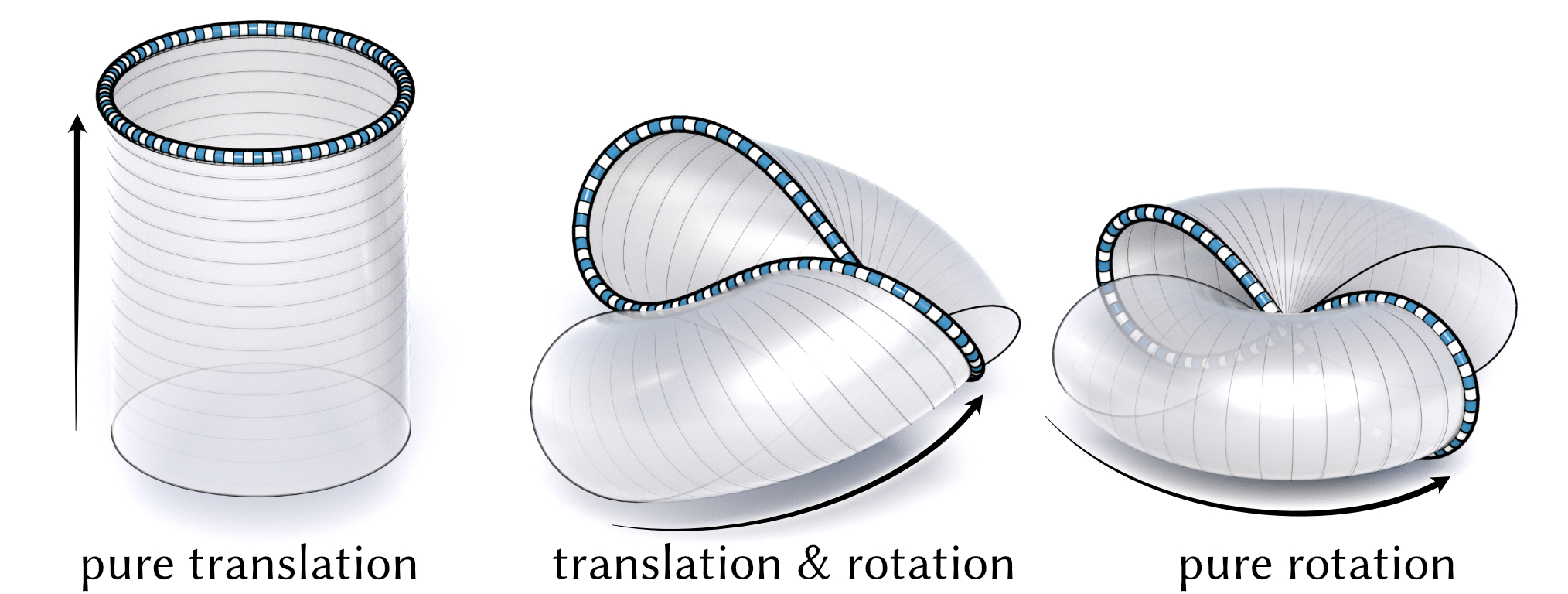}
    \caption{Under the localized-induction equation, an elastic curve evolves by an orientation-preserving rigid body motion, up to tangential reparametrization. This characterizes elastica as invariant curves of the flow.}
    \label{fig:RigidMotionToElastica}
\end{figure}

In contrast to these classical correspondences with spinning tops, pendulums,
vortex filaments, and integrable flows, the present work builds on a more
recent variational characterization of elastica. First noted by~\citet{Chern:2020:CHF} as a consequence of the recursion formula for the
localized-induction hierarchy, this characterization identifies elastica as critical points of the length functional under area and volume constraints. 
It is well suited to our purposes, since  
length, area and volume are all canonically defined directly for polygonal curves. In contrast to previous approaches, this ``isoperimetric characterization'' (see~\secref{sec:IsoperimetricCharacterization}) does not require choosing one of the many discrete analogues of curvature, torsion, or integrable system for polygons.

\subsubsection{Isoperimetric characterization}
\label{sec:IsoperimetricCharacterization}
The classical \emph{isoperimetric problem} asks which closed planar curve encloses the largest area at fixed length. Equivalently, one may ask which curve has the shortest length at fixed enclosed area. In both forms, the answer is the circle. We now pose an analogous problem for space curves.

For closed space curves \(\gamma\colon S^1\to \RR^3\), define the 
\emph{area vector} (\figref{fig:ProjectedAreaClosed})
\begin{equation*}
    \pzcA(\gamma) = \tfrac{1}{2}\oint \vec\gamma\times\vec\gamma'\,ds \in \mathbb{R}^3
\end{equation*}
and the \emph{volume vector} (\figref{fig:RevolutionVolume})
\begin{equation*}
    \pzcV(\gamma) = \tfrac{1}{3}\oint \vec\gamma\times(\vec\gamma\times\vec\gamma')\,ds \in \mathbb{R}^3.
\end{equation*}
The component \(\langle\vec a,\pzcA(\gamma)\rangle\) of the area vector along a unit vector \(\vec a\in S^2\subset\RR^3\) equals the signed area of the projection of \(\gamma\) onto \(\vec a^\perp\). Analogously, the component \(\langle\vec a,\pzcV(\gamma)\rangle\) of the volume vector equals (up to a factor of \(2\pi\)) the signed volume swept by rotating a spanning surface around the axis \(\mathbb{R}\vec a\) (see~\secref{sec:Isoperimetric}).

\citet[Cor. 2]{Chern:2020:CHF} show that a curve \(\gamma\in\cM\) is elastic if and only if it is a critical point of
\(\pzcL\) at fixed \(\pzcA_0\) and \(\pzcV_0\). In other words, elastic curves are precisely critical points of the length functional for prescribed area and volume vectors. While at first glance it is not obvious why the area and volume vectors should be
the right constraints for a variational problem on space curves, the present
paper provides an answer based on the following observation: Under a rigid
motion \(h=(Q,\vec c)\in\mathrm{SE}(3)\), they transform as
\begin{align*}
    \pzcA(h\circ\gamma) &= Q\pzcA(\gamma),\\
    \pzcV(h\circ\gamma) &= Q\pzcV(\gamma) + \vec c\times Q\pzcA(\gamma),
\end{align*}
which is precisely the coadjoint action of \(\mathrm{SE}(3)\) (\secref{sec:SpecialEuclideanGroup} \teqref{eq:CoadjointAction}). That is, \(\pzcA\) and \(\pzcV\) respectively transform like linear momentum and angular momentum in 3D kinematics. They constraints are therefore the Noether charges of the rigid-motion symmetry, \ie, the conserved quantities the geometry itself singles
out. We develop this connection in more generality in the following Sections~\ref{sec:Isoperimetric} and~\ref{sec:GeomMech}.

\subsection{General Theory}
\label{sec:Isoperimetric}
The isoperimetric characterization and the transformation laws for the area and volume vectors \(\pzcA\) and \(\pzcV\) extend beyond closed curves. To include open curves, we regard a curve segment as one fundamental domain of a \emph{quasi-periodic curve}, by which we mean a curve that repeats up to an orientation-preserving rigid motion \(g\in\mathrm{SE}(3)\).

More precisely, let \(\RR\) have its fixed orientation, and let \(\tau\colon\RR\to\RR\) denote a nontrivial translation in the positive direction. We denote by \(\cM_g\) the space of smooth immersions \(\gamma\colon \RR\to\RR^3\) satisfying
\begin{equation*}
    \tau^*\gamma=g\circ\gamma,
\end{equation*}
or equivalently, \(\gamma(\tau(t))=g(\gamma(t))\). Here, \(g\) is the
\emph{monodromy}, and any interval \(I=[t_0,\tau(t_0)]\) for \(t_0\in\RR\) is a \emph{fundamental domain}. The case \(g=\id\) recovers closed curves, while general \(g\) describes open curves together with a prescribed rigid motion relating consecutive copies.

\begin{figure}[b]
    \centering
    \includegraphics[width = \columnwidth]{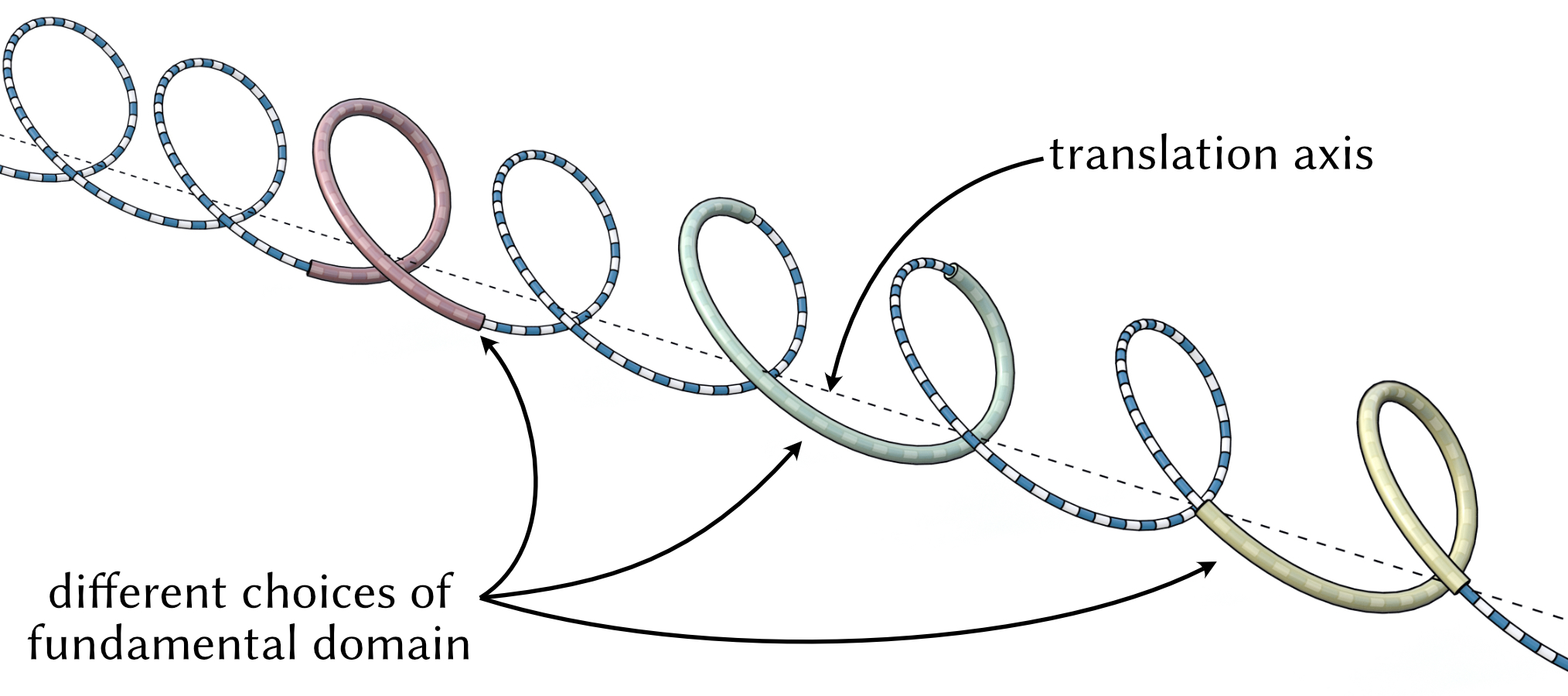}
    \caption{A quasi-periodic curve with translational monodromy is represented by one fundamental domain, whose copies differ by a fixed translation.}
    \label{fig:QuasiPeriodicFundamentalDomainsPureTranslation}
\end{figure}
In this framework, the area and volume vectors admit canonical extensions, and their well-defined components depend only on the monodromy. Thus closed and open elastica can be treated within a single framework.

\begin{remark}
\label{rem:ScrewAxis}
By a suitable choice of coordinate origin, the monodromy \(g\) always takes
the form of a screw motion (see~\appref{app:NormalForm}), that is, a
rotation about and translation along an axis \(\mathbb{R}\vec s_g\) for some
unit vector \(\vec s_g\in S^2\subset\mathbb{R}^3\), which we call the
\emph{screw axis} of \(g\). 
\end{remark}

\begin{figure}[t]
    \centering
    \includegraphics[width = \columnwidth]{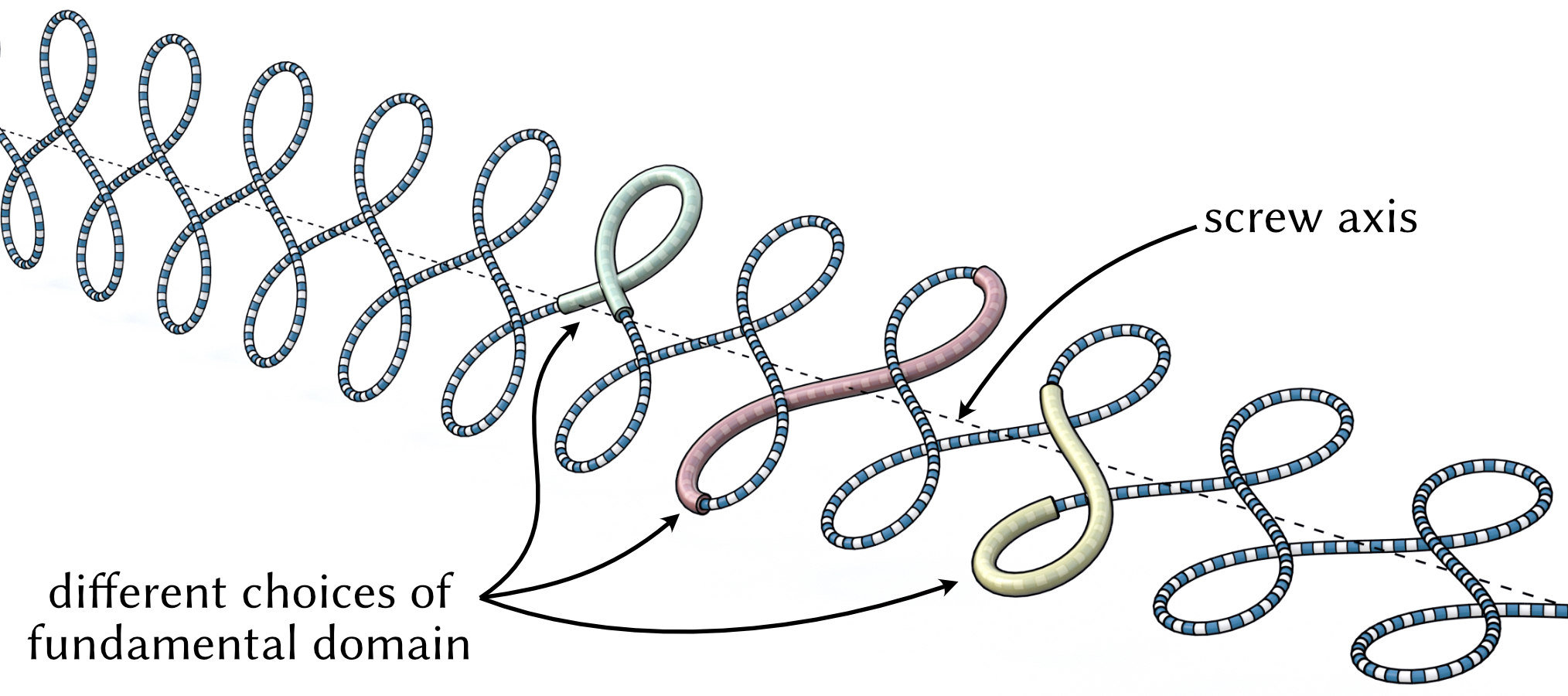}
    \caption{A quasi-periodic curve with screw monodromy repeats by a rotation and translation along the same axis, so that one fundamental domain represents the full infinite curve.}
    \label{fig:QuasiPeriodicFundamentalDomains}
\end{figure}
\subsubsection{Monodromy cases and compatible rigid motions}
\label{sec:MonodromyCases}
For elements \(g\in\SE(3)\) write \(g=(R,\vec b)\), where \(R\in\SO(3)\) is a rotation and \(\vec b\in \RR^3\), and distinguish three cases. If \(g=\id\), the
curve is closed (see~\figref{fig:ProjectedAreaClosed}). If \(g=(\Id,\vec b)\) with
\(\vec b\neq 0\), consecutive copies differ by a pure translation
(see~\figref{fig:QuasiPeriodicFundamentalDomainsPureTranslation}). If \(g=(R,\vec b)\) with
\(R\neq \Id\), then, after choosing the origin on its axis, \(g\) is a screw
motion: a rotation about an axis together with a translation along the same
axis (see~\figref{fig:QuasiPeriodicFundamentalDomains}). This distinction determines which
components of the area and volume vectors are independent of the chosen
fundamental domain.

\begin{figure}[h]
    \centering
\includegraphics[width=\columnwidth]{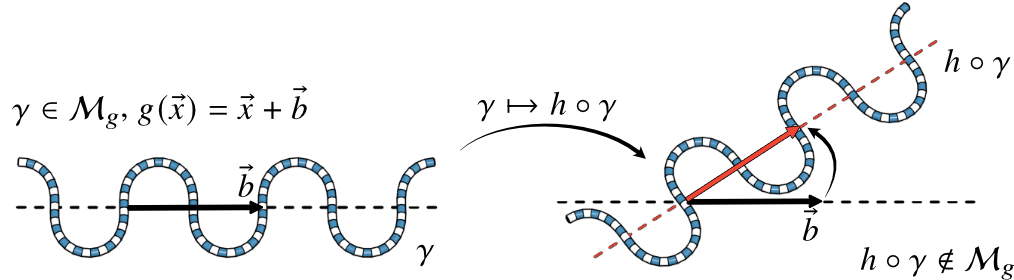}
    \caption{A rigid motion \(h\) generally changes the monodromy of a quasi-periodic curve from \(g\) to \(hgh^{-1}\). In the translational example shown, \(h\) rotates the translation direction \(\vec b\), so \(h\circ\gamma\notin\cM_g\). Thus only rigid motions in \(Z(g)\) preserve \(\cM_g\).}
    \label{fig:MonodromyNotPreserved}
\end{figure}
A rigid motion need not preserve the space \(\cM_g\) (\figref{fig:MonodromyNotPreserved}). Indeed, if
\(h\in\SE(3)\) and \(\gamma\in\cM_g\), then
\begin{equation*}
    (h\circ\gamma)(\tau(t))
    =
    h(g(\gamma(t)))
    =
    (hgh^{-1})((h\circ\gamma)(t)).
\end{equation*}
Thus \(h\circ\gamma\in\cM_{hgh^{-1}}\). The rigid motions preserving
\(\cM_g\) are therefore precisely those commuting with \(g\). The corresponding subgroup of \(\SE(3)\) is called the \emph{centralizer} of \(g\in\SE(3)\) and denoted by
\begin{equation*}
    Z(g) \coloneqq \{h\in\SE(3)\mid hg=gh\}.
\end{equation*}

We distinguish three cases for \(g=(R,\vec b)\in\SE(3)\):
\begin{enumerate}[(i)]
    \item If \(g=\id\), then \(Z(g)=\SE(3)\).
    \item If \(R=\Id\) and \(\vec b\neq0\), then
    \begin{equation*}
        Z(g)
        =
        \{(Q,\vec c)\in\SE(3)\mid Q\vec b=\vec b,\ \vec c\in\RR^3\}.
    \end{equation*}
    so \(Z(g)\) consists of rotations about the direction \(\vec s_g=\vec b/|\vec b|\) together with arbitrary translations.
    \item If \(R\neq \Id\) and \(g\) is not a pure half-turn\footnote{Pure half-turn monodromy has an additional disconnected centralizer component characterized by \(Q\vec s_g=-\vec s_g\). We exclude this case throughout. Since \(g^2=\id\), two consecutive fundamental domains form a closed curve, so it may instead be treated within the closed theory with an imposed half-turn symmetry.}, choose the origin on its screw axis \(\RR\vec s_g\) and write \(\vec b=\ell\vec s_g\). Then
    \begin{equation*}
        Z(g)
        =
        \{(Q,\vec c)\in\SE(3)\mid Q\vec s_g=\vec s_g,\ 
        \vec c\parallel\vec s_g\}.
    \end{equation*}
\end{enumerate}

\subsubsection{Area vector of a space curve}
How does one define an area-type quantity for a space curve when there is
no canonical spanning surface whose area can be measured? A natural answer
is to look at projections to \(2D\) planes. Fix
\(\vec a\in S^2\subset\RR^3\) and project \(\gamma\) orthogonally onto
\(\vec a^\perp\). 
Then, for a closed curve \(\gamma\in\cM_{\id}\), the
signed area enclosed by the projection of \(\gamma\) to the plane
\(\vec a^\perp\) is given by
\begin{align*}
    \pzcA_{\vec a}(\gamma)
    &\coloneqq \tfrac{1}{2}\oint
    \langle\vec a,\,
    (\vec\gamma - \langle\vec\gamma,\vec a\rangle\vec a)
    \times
    (\vec\gamma' - \langle\vec\gamma',\vec a\rangle\vec a)\rangle\,ds \\
    &= \tfrac{1}{2}\oint
    \langle\vec a,\,\vec\gamma\times\vec\gamma'\rangle\,ds \\
    &= \Bigl\langle\vec a,\,
    \tfrac{1}{2}\oint \vec\gamma\times\vec\gamma'\,ds\Bigr\rangle,
\end{align*}
where the area form on \(\vec a^\perp\) is given by
\(\langle\vec a,\cdot\,\times\cdot\rangle\). This motivates the definition
of the \emph{area vector}
\begin{equation}
    \label{eq:AreaVector}
    \tfrac{1}{2}\oint \vec\gamma\times\vec\gamma'\,ds
    \in \RR^3,
\end{equation}
whose component \(\langle\vec a,\tfrac{1}{2}\oint \vec\gamma\times\vec\gamma'\,ds\rangle\) is the curve's signed
projected area along \(\vec a\) (\figref{fig:ProjectedAreaClosed}). 

For a non-closed quasi-periodic curve, changing the fundamental
domain produces a boundary term, so the line integral in
\teqref{eq:AreaVector} is not well defined by itself. 
\lemref{thm:AreaFundamentalDomainIndependent} states that the relevant components of the \emph{completed} area vector 
\begin{equation}
    \label{eq:AreaVectorCompleted}
    \pzcA(\gamma)
    \coloneqq
    \tfrac{1}{2}\int_I \vec\gamma\times\vec\gamma'\,ds
    \, -\tfrac{1}{2}\vec b\times\vec\gamma(t_0)
    \in \RR^3,
\end{equation}
become independent of the chosen fundamental domain (\figref{fig:ProjectedAreas}).\\
 
 \begin{figure}[h]
    \centering
    \includegraphics[width = \columnwidth]{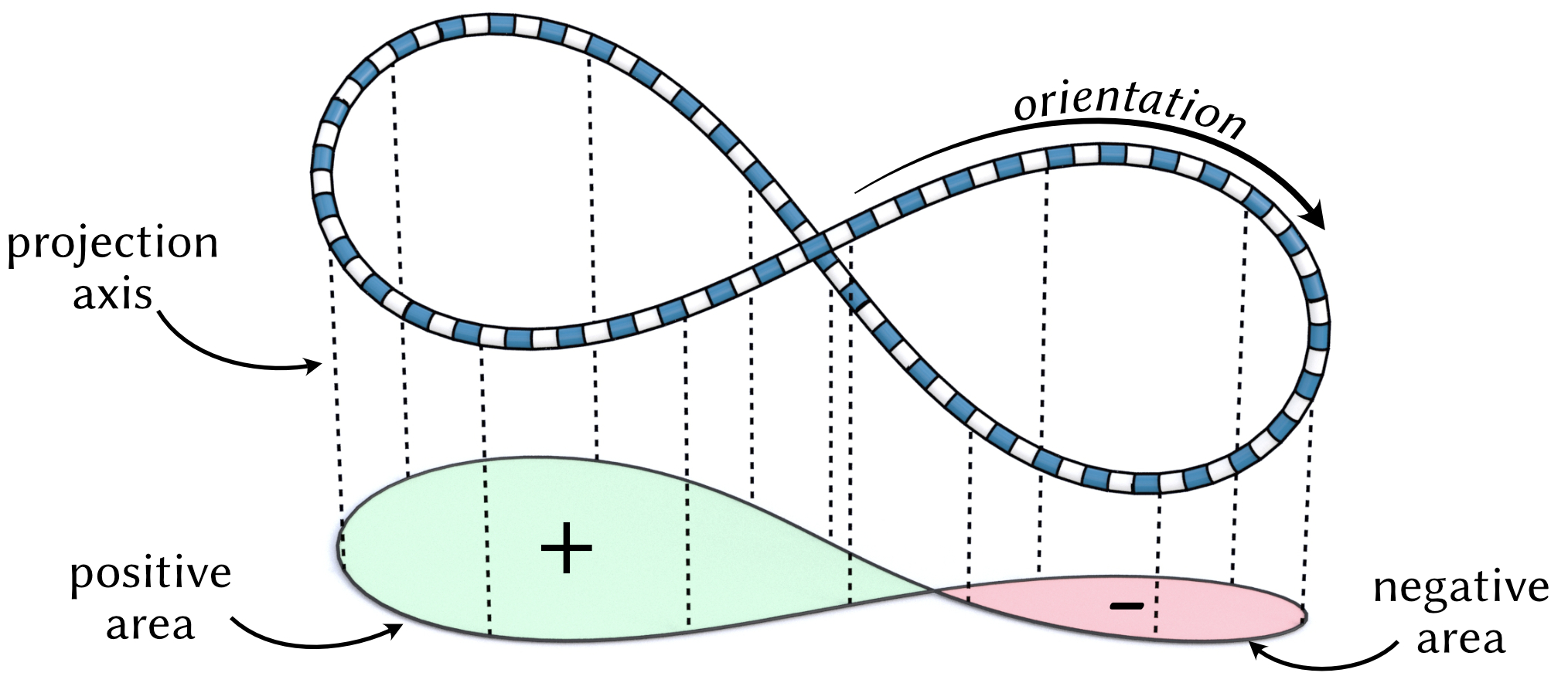}
    \caption{The component of the area vector in the projection direction records the signed area enclosed by the projected curve, with sign determined by orientation.}
    \label{fig:ProjectedAreaClosed}
\end{figure}

\begin{lemma}
\label{thm:AreaFundamentalDomainIndependent}
Let \(\gamma\in\cM_g\) have monodromy \(g=(R,\vec b)\), and let \(\pzcA(\gamma)\) be defined by \teqref{eq:AreaVectorCompleted} on a fundamental domain \(I=[t_0,\tau(t_0)]\). Then the following area data are independent of the choice of \(I\):
\begin{enumerate}[(i)]
    \item If \(g=\id\), the full vector \(\pzcA(\gamma)\).

    \item If \(R=\Id\) and \(\vec b\neq0\), the full vector \(\pzcA(\gamma)\).

    \item If \(R\neq \Id\), after choosing the origin on the screw axis \(\vec s_g\) of \(g\), the projection
    \(\langle\vec s_g,\pzcA(\gamma)\rangle\) onto the screw axis.
\end{enumerate}
\end{lemma}
\begin{proof}
    See~\appref{sec:AreaTransformation}.
\end{proof}

\begin{figure}[t]
    \centering
    \includegraphics[width = \columnwidth]{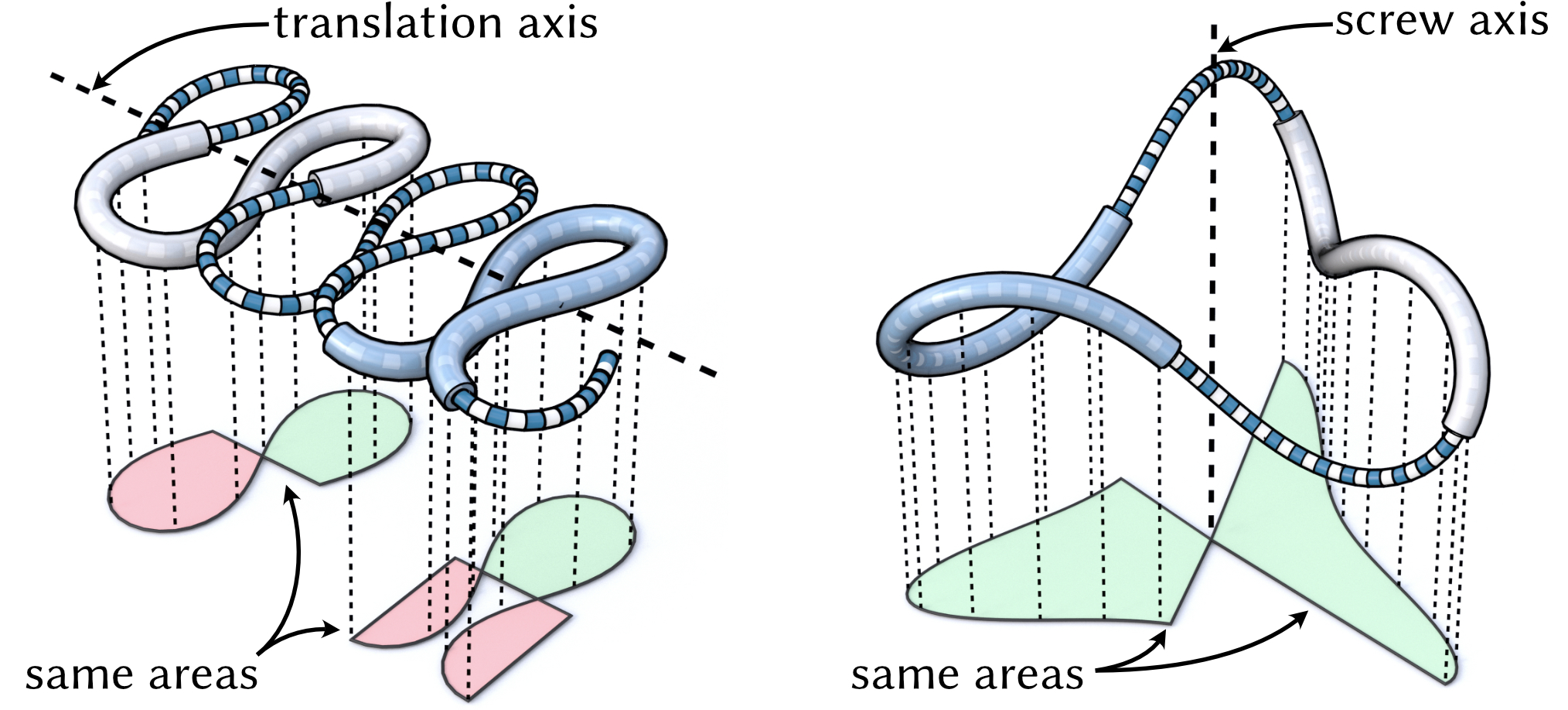}
    \caption{The intrinsic components of the area vector depend on the monodromy. For pure translations the full area vector is well-defined, while for screw monodromy only the axial component is independent of the chosen fundamental domain.}
    \label{fig:ProjectedAreas}
\end{figure}

\subsubsection{Volume vector of a space curve}
\label{sec:VolumeVector}
To define the analogous volume-type quantity, again start with a closed
curve. Fix \(\vec a\in S^2\subset\RR^3\) and consider the rotational vector
field \(X_{\vec a}(\vec r)=\vec a\times\vec r\). Since
\(\operatorname{div}X_{\vec a}=0\), the flux
\begin{equation}
    \label{eq:FluxFormula}
    \Upsilon_{\vec a}(\gamma)
    \coloneqq
    \iint_{\Sigma}\langle X_{\vec a}(\vec r),N\rangle\,dS
\end{equation}
through any surface \(\Sigma\) with boundary \(\gamma\) and surface unit
normal \(N\) is independent of the choice of \(\Sigma\). Evaluating this
flux on the cone
\(\Sigma=\{u\,\vec\gamma(s)\mid u\in[0,1],\,s\in I\}\) gives
\begin{align*}
    \Upsilon_{\vec a}(\gamma)
    &= \int_0^1\oint
    \langle\vec a\times(u\,\vec\gamma),
    \partial_u \vec r\times\partial_s \vec r\rangle\,ds\,du \\
    &= \int_0^1u^2\,du\oint
    \langle\vec a\times\vec\gamma,
    \vec\gamma\times\vec\gamma'\rangle\,ds \\
    &= \tfrac{1}{3}\oint
    \langle\vec a\times\vec\gamma,
    \vec\gamma\times\vec\gamma'\rangle\,ds \\
    &= \Bigl\langle\vec a,\,
    \tfrac{1}{3}\oint
    \vec\gamma\times(\vec\gamma\times\vec\gamma')\,ds\Bigr\rangle.
\end{align*}
This motivates the definition of the \emph{volume vector} (\figref{fig:RevolutionVolume})
\begin{equation}
    \label{eq:VolumeVector}
    \tfrac{1}{3}\oint
    \vec\gamma\times(\vec\gamma\times\vec\gamma')\,ds
    \in\RR^3.
\end{equation}
\begin{figure}[h]
    \centering
    \includegraphics[width = \columnwidth]{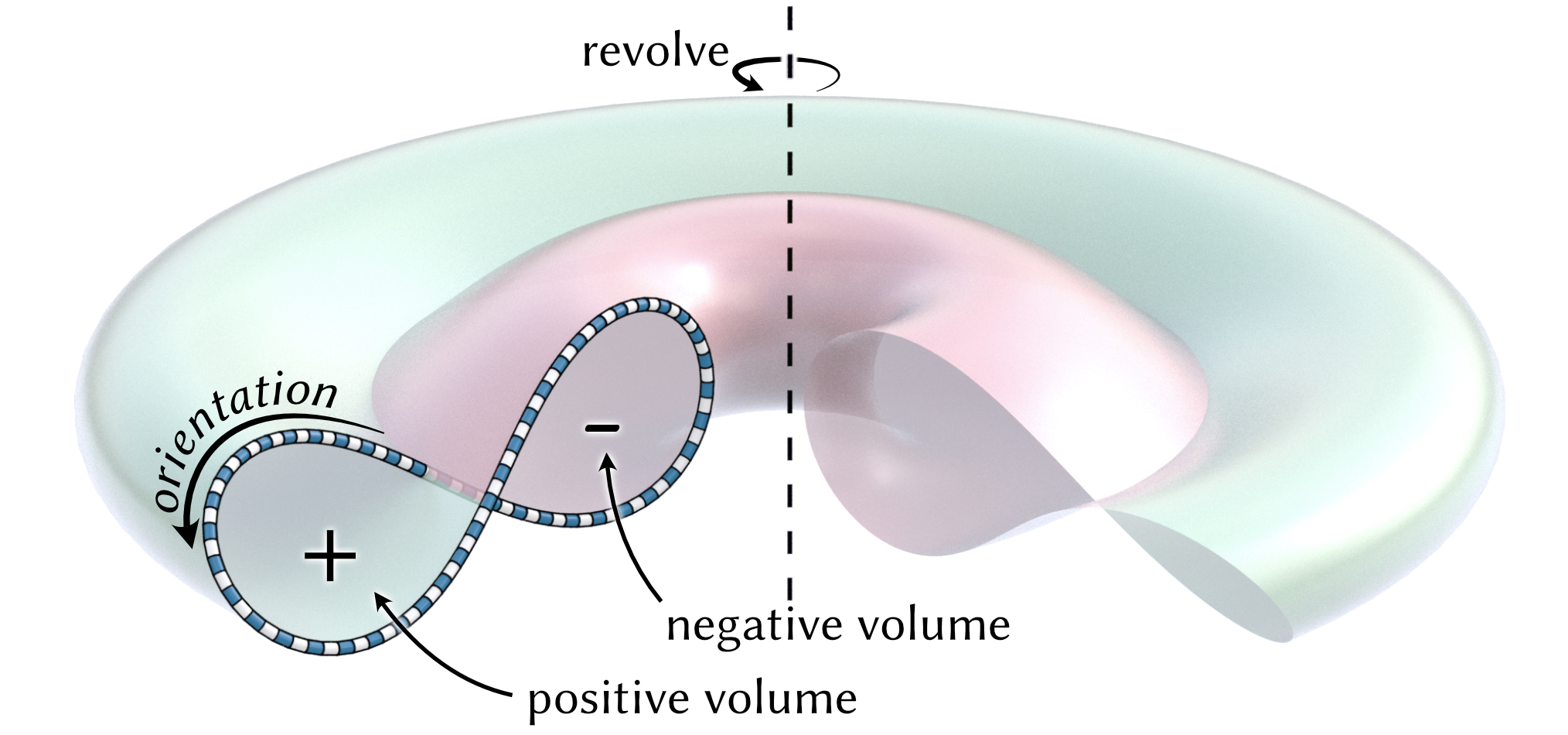}
    \caption{The corresponding component of the volume vector measures the signed volume swept by rotating a spanning surface around the chosen axis.}
    \label{fig:RevolutionVolume}
\end{figure}
For a closed curve, the component \(\langle\vec a,\pzcV(\gamma)\rangle\) is the flux of the rotational vector field \(X_{\vec a}(\vec r)=\vec a\times\vec r\), which is proportional to the signed volume swept out by rotating a spanning surface around this axis. 

Analogously, the volume vector for a non-closed quasi-periodic curve requires additional care: the cone has an additional boundary, so the line integral in \teqref{eq:VolumeVector} depends on the chosen fundamental domain. \lemref{thm:VolumeFundamentalDomainIndependent} states that the relevant components of the \emph{completed} volume vector
\begin{equation}
    \label{eq:VolumeVectorCompleted}
    \pzcV(\gamma)
    \coloneqq
    \tfrac{1}{3}\int_I
    \vec\gamma\times(\vec\gamma\times\vec\gamma')\,ds
    \, -\tfrac{1}{6}\vec\gamma(t_0)\times\bigl(\vec b\times\vec\gamma(t_0)\bigr)
    \in\RR^3,
\end{equation}
become independent of the chosen fundamental domain \([t_0,\tau(t_0)]\).

\begin{lemma}
\label{thm:VolumeFundamentalDomainIndependent}
Let \(\gamma\in\cM_g\) have monodromy \(g=(R,\vec b)\), and let \(\pzcV(\gamma)\) be defined by \teqref{eq:VolumeVectorCompleted} on a fundamental domain \(I=[t_0,\tau(t_0)]\). Then the following volume data are independent of the choice of \(I\):
\begin{enumerate}[(i)]
    \item If \(g=\id\), the full vector \(\pzcV(\gamma)\).

    \item If \(R=\Id\) and \(\vec b\neq0\), with \(\vec s_g=\vec b/|\vec b|\), the projection
    \(\langle\vec s_g,\pzcV(\gamma)\rangle\) onto the chosen translation axis.

    \item If \(R\neq \Id\), after choosing the origin on the screw axis \(\vec s_g\) of \(g\), the projection
    \(\langle\vec s_g,\pzcV(\gamma)\rangle\) onto the screw axis.
    \end{enumerate}
\end{lemma}
\begin{proof}
    See~\appref{sec:VolumeTransformation}.
\end{proof}

\subsubsection{Isoperimetric characterization of quasi-periodic elastic curves}
We now turn to the isoperimetric characterization of
elastic curves for quasi-periodic curves. The constraints are the
well-defined components of the area and volume vectors, which depend on
the monodromy as established in Lemmas~\ref{thm:AreaFundamentalDomainIndependent}
and~\ref{thm:VolumeFundamentalDomainIndependent}.

\begin{theorem}
\label{thm:Isoperimetric}
Let \(\gamma\in\cM_g\) have monodromy \(g=(R,\vec b)\), and choose the origin on the screw axis whenever \(R\neq\Id\). For any \(\vec\lambda_1,\vec\lambda_2\in\RR^3\), where \(\vec\lambda_1\) is arbitrary if \(g=\id\) and parallel to \(\vec s_g\) otherwise, while \(\vec\lambda_2\) is arbitrary if \(R=\Id\) and parallel to \(\vec s_g\) if \(R\neq\Id\), the following are equivalent:
\begin{enumerate}[(i)]
    \item The functional \(\pzcL+\langle\pzcV\mid\vec\lambda_1\rangle+\langle\pzcA\mid\vec\lambda_2\rangle\) is stationary at \(\gamma\) with respect to every variation \(\mathring\gamma\in T_\gamma\cM_g\), \ie, \(\mathring{\vec\gamma}(\tau(t))=R\mathring{\vec\gamma}(t)\).

    \item The corresponding local first variation on the underlying infinite curve vanishes for every compactly supported variation \(\mathring{\vec\gamma}\in C_0^\infty(\RR,\RR^3)\).
\end{enumerate}
Moreover, \(\gamma\) is elastic if and only if there exists a pair \((\vec\lambda_1,\vec\lambda_2)\) satisfying the conditions above for which either, and hence both, of these equivalent conditions holds.
\end{theorem}
\begin{proof}
    See~\appref{app:IsoperimetricProof}.
\end{proof}
 
\thmref{thm:Isoperimetric} ensures that we do not need to distinguish between variations on the space of curves \(\cM\) under variations with compact support and on the space of quasi-periodic curves \(\cM_g\) under variations compatible with the quasi-periodicity.
 
\begin{corollary}[Isoperimetric characterization of elastic curves]
\label{thm:IsoperimetricCharacterization}
Let \(\gamma\in\cM_g\) have monodromy \(g=(R,\vec b)\in\mathrm{SE}(3)\), and choose the origin on the screw axis of \(g\) whenever \(R\neq\Id\). Then \(\gamma\) is elastic if and only if it is a critical point of \(\pzcL\) at fixed well-defined components of the area and volume vectors.
\end{corollary}
\begin{proof}
Constrained criticality is equivalent to stationarity of \(\pzcL\) plus a linear combination of the well-defined area--volume components, so the claim follows from \thmref{thm:Isoperimetric}.
\end{proof}

While~\corref{thm:IsoperimetricCharacterization} was first observed by~\citet[Cor. 2]{Chern:2020:CHF}, we provide a proof of \thmref{thm:Isoperimetric} (and thus~\corref{thm:IsoperimetricCharacterization}) for the quasi-periodic case and using only elementary techniques from the calculus of variations.

\section{Geometric Mechanics of Elastic Curves}
\label{sec:GeomMech}

The isoperimetric characterization of \secref{sec:Isoperimetric} has a deeper
geometric explanation: the area and volume vectors are not arbitrary integral
constraints, but the \emph{momentum map} of a natural symmetry group action on the
space of curves. Making this precise requires a brief review of the relevant \emph{geometric mechanics}. For a more in-depth overview, we refer to~\cite{Marsden:1999:IMS}. We develop the theory directly for quasi-periodic curves
\(\gamma\in\cM_g\), with the classical closed curve case \(g=\id\) a special case.

\subsection{Symmetries of \(\cM_g\)}
\label{sec:LieGroups}
A \emph{Lie group} is a group that is simultaneously a smooth manifold, with
group multiplication and inversion smooth maps. The tangent space at the
identity of any Lie group \(G\) is its Lie algebra \(\mathfrak{g}=T_{\mathrm{id}}G\),
which encodes the infinitesimal structure of the group. When a Lie group acts
on a configuration space and the Lagrangian is invariant under this action,
the group constitutes a symmetry of the variational problem, and Noether's
theorem associates a conserved quantity to each such symmetry.

\subsubsection{Adjoint and Coadjoint actions}
Every Lie group \(G\) acts on its Lie algebra \(\mathfrak{g}\) by the
\emph{adjoint action}
\begin{equation*}
    \mathrm{Ad}_g\colon\mathfrak{g}\to\mathfrak{g}, \qquad
    \mathrm{Ad}_g\,\xi = \left.\tfrac{d}{dt}\right|_{t=0} g\,h(t)\,g^{-1},
\end{equation*}
where \(t\mapsto h(t)\in G\) is any curve with \(h(0)=\mathrm{id}\) and \(h'(0)=\xi\). It describes how infinitesimal rigid body motions transform under conjugation by a finite
group element. Dualizing this action gives the \emph{coadjoint action}
\begin{equation*}
    \mathrm{Ad}^*_g\colon\mathfrak{g}^*\to\mathfrak{g}^*, \qquad g\in G,
\end{equation*}
defined by \(\langle\mathrm{Ad}^*_g\eta\mid x\rangle =
\langle\eta\mid\mathrm{Ad}_{g^{-1}}x\rangle\) for \(x\in\mathfrak{g}\) and
\(\eta\in\mathfrak{g}^*\), where \(\langle\cdot\mid\cdot\rangle\colon
\mathfrak{g}^*\times\mathfrak{g}\to\mathbb{R}\) is the natural pairing.
Intuitively, \(\mathrm{Ad}^*_g\) describes how \emph{momenta}---physical quantities dual to
infinitesimal rigid body motions---transform under a finite rigid motion \(g\). 

\subsubsection{The special Euclidean group}
\label{sec:SpecialEuclideanGroup}
For our purposes, the \emph{special Euclidean group} \(\mathrm{SE}(3)\) of
orientation-preserving rigid motions of \(\mathbb{R}^3\) is the primary
example. It is a six-dimensional Lie group whose Lie algebra
\(\mathfrak{se}(3)\) consists of infinitesimal rigid body motions, identified
with pairs \((\vec \omega,\vec v)\in\mathbb{R}^3\times\mathbb{R}^3\), where \(\vec \omega\in\RR^3\) is an infinitesimal rotation and \(\vec v\in\RR^3\) an infinitesimal translation. Elements \((\vec l,\vec p)\in\RR^3\times\RR^3\cong\mathfrak{se}(3)^*\) of the dual space canonically pair with elements \((\vec \omega,\vec v)\in\mathbb{R}^3\times\mathbb{R}^3\cong \mathfrak{se}(3)\) by
\begin{equation*}
    \langle(\vec l,\vec p)\mid(\vec \omega,\vec v)\rangle
    = \langle \vec l\mid\vec \omega\rangle+\langle \vec p\mid \vec v\rangle.
\end{equation*}
The coadjoint action of \(\mathrm{SE}(3)\) on its dual Lie algebra
\(\mathfrak{se}(3)^*\) is given by
\begin{equation}
\label{eq:CoadjointAction}
    \mathrm{Ad}^*_{(R,\vec b)}(\vec l,\vec p) = (R\vec l+\vec b\times R\vec p,\,R\vec p),
\end{equation}
where \((R,\vec b)\in\mathrm{SE}(3)\) and \((\vec l,\vec p)\in\mathfrak{se}(3)^*\). 

The Lie algebra of the centralizer subgroups \(Z(g)\) of elements \(g\in\SE(3)\) as in \secref{sec:MonodromyCases} is denoted by \(\fz(g)\) and given by the corresponding subspaces
\begin{enumerate}[(i)]
	\item If \(g=\mathrm{id}\), then \(\fz(g)=\se(3)\). 
	\item If \(R=\Id, \vec b\neq 0\),
	\begin{equation*}
        \fz(g)
        =
        \{(\vec \omega,\vec v)\mid
        \vec \omega,\vec v\in\RR^3, \  \vec\omega\parallel \vec s_g\}.
        \end{equation*}
    \item If \(g=(R,\vec b)\) with \(R\neq \Id\),
    \begin{equation*}
        \fz(g)
        =
        \{(\vec \omega,\vec v)\mid
        \vec \omega,\vec v\in\RR^3, \  \vec\omega, \vec v\parallel \vec s_g\}.
        \end{equation*}
\end{enumerate}

\subsubsection{Covariance of area and volume under \(\mathrm{SE}(3)\)}
Although the components identified above are independent of the chosen
fundamental domain, they are not invariant under arbitrary rigid motions of the curve in the ambient space. Instead, area and volume transform covariantly with the placement of the curve and with the chosen screw-axis gauge. 
\begin{theorem}
\label{thm:AreaVolumeTransformation}
Let \(\gamma\in\cM_g\) with monodromy \(g=(R,\vec b)\in\SE(3)\), and let \(h=(Q,\vec c)\in Z(g)\) be a symmetry of \(\cM_g\). Then:
\begin{enumerate}[(i)]
    \item \label{thm:AreaVolumeTransformationCase1} If \(g=\mathrm{id}\), any \(h\in\mathrm{SE}(3)\) is allowed and
    \begin{align*}
        \label{eq:TransformationLaw}
        \pzcA(h\circ\gamma) &= Q\pzcA(\gamma), \\
        \pzcV(h\circ\gamma) &= Q\pzcV(\gamma) + \vec c\times Q\pzcA(\gamma).
    \end{align*}
    
    \item \label{thm:AreaVolumeTransformationCase2} If \(R=\Id\) and \(\vec b\neq 0\), then
    \(\vec s_g=\vec b/|\vec b|\), and \(h=(Q,\vec c)\in Z(g)\) has \(Q\vec s_g=\vec s_g\) with \(\vec c\in\RR^3\) arbitrary. Writing \(\vec c^\perp=\vec c-\langle\vec s_g,\vec c\rangle\vec s_g\),
    \begin{align*}
        \pzcA(h\circ\gamma)
        &=
        Q\pzcA(\gamma)+\vec c\times\vec b, \\
        \left\langle \vec s_g,\pzcV(h\circ\gamma)\right\rangle
        &=
        \left\langle \vec s_g,Q\pzcV(\gamma)+\vec c\times Q\pzcA(\gamma)\right\rangle
        -\tfrac12|\vec b|\,|\vec c^\perp|^2.
    \end{align*}
    When the translation preserves the axis \(\vec s_g\) (\ie, \(\vec c\parallel\vec s_g\)) the two additional terms vanish.

    \item If \(R\neq \Id\), \(h\) must be a screw motion about the axis 
    \(\mathbb{R}\vec s_g\), and the axial projections transform as
    \begin{align*}
        \langle\vec s_g,\pzcA(h\circ\gamma)\rangle &=
        \langle\vec s_g, Q\pzcA(\gamma)\rangle, \\
        \langle\vec s_g,\pzcV(h\circ\gamma)\rangle &=
        \langle\vec s_g, Q\pzcV(\gamma) + \vec c\times Q\pzcA(\gamma)\rangle.
    \end{align*}
\end{enumerate}
\end{theorem}
\begin{proof}
    See~\appref{sec:AreaVolumeTransformationProof}.
\end{proof}

These transformation laws are
the geometric heart of the paper and drive both the continuous and discrete
theories that follow. In~\secref{sec:MomentumMap} we identify the restriction of
\((\pzcV,\pzcA)\) to the compatible infinitesimal rigid motions as
\emph{momenta}, \ie, elements of \(\mathfrak z(g)^*\) transforming under
\(Z(g)\), which explains why fixing them is the natural constraint for elastic
curves.

\subsection{\(\cM_g\) as a (pre-)symplectic manifold}
\label{sec:MWForm}
In geometric mechanics, a \emph{pre-symplectic structure} on a manifold is a closed 2-form. It turns out that the space \(\cM_g\) naturally carries such a structure. The \emph{Marsden--Weinstein form} is given by 
\begin{equation}
\label{eq:MWForm}
    %
    \OmegaMW_\gamma(\dot\gamma,\mathring\gamma)
    =
    \int_I
    \det\bigl(\dot{\vec\gamma},\mathring{\vec\gamma},d\vec\gamma\bigr),
\end{equation}
for all \(\dot\gamma,\mathring\gamma\in T_\gamma\cM_g\)~\cite{Marsden:1983:COV}. 
\begin{theorem}
\label{thm:MWSymplectic}
    The Marsden--Weinstein form \(\OmegaMW\) is a pre-symplectic form on \(\cM_g\), \ie, a closed \(2\)-form. Moreover, its kernel consists precisely of the reparametrization directions. 
\end{theorem}
\begin{proof}
    See~\appref{sec:MWLWellDefinedProof}.
\end{proof}

\begin{wrapfigure}[10]{r}{0.425\columnwidth}
    \vspace{-.8em}
  {\hspace{-2.em}
    \includegraphics[width=0.475\columnwidth]{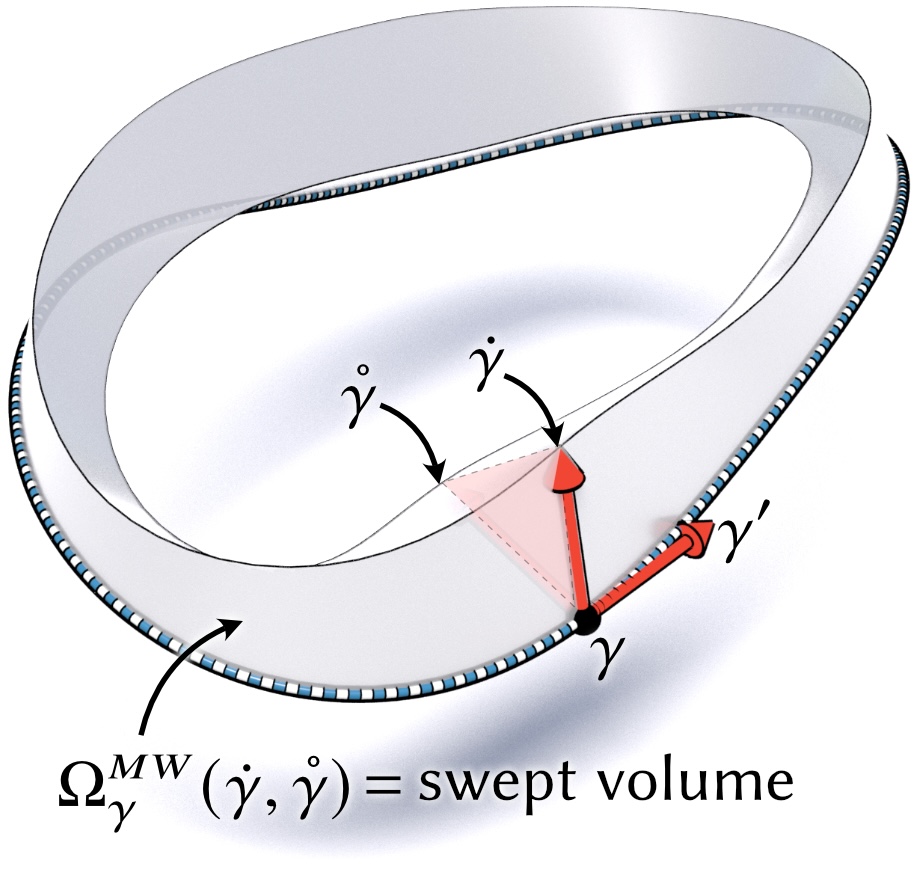}
  }  
\end{wrapfigure}
Geometrically, the Marsden--Weinstein form measures the signed volume swept along \(\gamma\) by the infinitesimal oriented area element spanned by \(\dot\gamma\) and \(\mathring\gamma\).

This geometric interpretation also explains the degeneracy. A reparametrization direction is collinear with the tangent \(\gamma'=d\gamma(\tfrac{\partial}{\partial s})\), the
direction along which the prism is swept, so the prism collapses and has no volume. Thus, after quotienting by reparametrizations, the Marsden--Weinstein form descends to a closed, non-degenerate \(2\)-form, \ie, a genuine
\emph{symplectic form}. 

 The action of the centralizer \(Z(g)\) preserves the Marsden--Weinstein form on \(\cM_g\), \ie, this construction is compatible with the symmetries of \(\cM_g\).
\begin{lemma}
\label{thm:CentralizerSymplectomorphism}
    The centralizer \(Z(g)\) preserves the pre-symplectic form
    \(\OmegaMW\) on \(\cM_g\). Consequently, the induced action of \(Z(g)\)
    on the quotient of \(\cM_g\) by reparametrizations is symplectic.
\end{lemma}
\begin{proof}
    See~\appref{sec:CentralizerSymplectomorphismProof}.
\end{proof}

\subsection{The Momentum Map and Noether charges}
\label{sec:MomentumMap}
Whenever a Lie group \(G\) acts on a symplectic manifold \((M,\Omega)\)
by symplectomorphisms, one seeks a \emph{momentum map}, \ie, a map
\begin{equation*}
    \mu\colon M\to\fg^*
\end{equation*}
satisfying, for every \(\xi\in\fg\),
\begin{equation}
\label{eq:MomentumMapDefiningEq}
    \iota_{X_\xi}\Omega = d\langle\mu\,\vert\, \xi\rangle ,
\end{equation}
where \(X_\xi\) is the vector field on \(M\) generated by \(\xi\). Thus each component \(\langle\mu\,\vert\,\xi\rangle\) records the conserved quantity associated with the infinitesimal symmetry \(\xi\), and the momentum map packages these \emph{Noether charges} into a single \(\fg^*\)-valued function. The same definition applies to the pre-symplectic form on \(\cM_g\), before passing to the quotient by reparametrizations. 

The purpose of this section is to identify the area-volume data of \thmref{thm:Isoperimetric} as momenta. Through this identification, the constraints in the isoperimetric problem are no longer external choices: they are the Noether charges of the rigid-motion symmetries that remain compatible with the monodromy.

\subsubsection{The momentum map}
The relevant symmetry group of \(\cM_g\) is the centralizer \(Z(g)\),
whose Lie algebra \(\fz(g)\subset\se(3)\) consists of the infinitesimal
rigid motions compatible with the monodromy. Hence the momentum map
records only the components of the area-volume pair visible to
\(\fz(g)\).

For \((\vec\omega,\vec v)\in\fz(g)\), the infinitesimal vector field on
\(\cM_g\) is
\begin{equation*}
    X_{(\vec\omega,\vec v)}(\gamma)=\vec \omega\times\vec \gamma+\vec v.
\end{equation*}
By \lemref{thm:CentralizerSymplectomorphism}, this action preserves
\(\OmegaMW\), and its momentum map is as follows~\cite{Marsden:1999:IMS}. 

\begin{theorem}
\label{thm:MomentumMap}
    Let \(g\in\SE(3)\), and let
    \(\iota_\fz\colon\mathfrak z(g)\hookrightarrow\se(3)\) denote the
    inclusion of the Lie algebra of its centralizer. The momentum map for
    the \(Z(g)\)-action on \((\cM_g,\OmegaMW)\) is
    \begin{equation*}
        \mu_g\colon \cM_g\to\mathfrak z(g)^*,\ \gamma\mapsto
        \iota_\fz^*
        \begin{pmatrix}
            \pzcV(\gamma) \\
            \pzcA(\gamma)
        \end{pmatrix}.
    \end{equation*}
\end{theorem}
\begin{proof}
    See~\appref{sec:MomentumMapProof}.
\end{proof}
\thmref{thm:MomentumMap} states that the momentum map on \(\cM_g\) is obtained from the full area-volume pair by restricting the covector \((\pzcV,\pzcA)\in\se(3)^*\) to \(\fz(g)\subset\se(3)\).
Thus the remaining Noether charges are:
\begin{enumerate}[(i)]
    \item For \(g=\id\),
    \begin{equation*}
        \mu_g(\gamma)
        =
        \begin{pmatrix}
            \pzcV(\gamma) \\
            \pzcA(\gamma)
        \end{pmatrix},
    \end{equation*}
    which recovers the known angular and linear momenta of the Marsden--Weinstein form on the space of closed space curves~(\cf~\cite{Shashikanth:2003:LVR}).

    \item For pure translational monodromy, \(R=\Id\) and \(\vec b\neq 0\),
    \begin{equation*}
        \mu_g(\gamma)
        =
        \begin{pmatrix}
            \langle \nicefrac{\vec b}{|\vec b|}, \pzcV(\gamma)\rangle \\
            \pzcA(\gamma)
        \end{pmatrix}.
    \end{equation*}

    \item For screw monodromy, \(R\neq \Id\),
    \begin{equation*}
        \mu_g(\gamma)
        =
        \begin{pmatrix}
            \langle \vec s_g, \pzcV(\gamma)\rangle \\
            \langle \vec s_g, \pzcA(\gamma)\rangle
        \end{pmatrix},
    \end{equation*}
    where \(\vec s_g\in S^2\) is chosen such that \(\RR\vec s_g\) is the screw axis.
\end{enumerate}

With the momentum map in hand, the otherwise unexpected area--volume
constraints cease to be mysterious. On the fixed-monodromy space
\(\cM_g\), the length \(\pzcL\) is invariant under the symmetry group \(Z(g)\), and \(\mu_g\) records the corresponding
Noether charges. Fixing the area--volume components selected by the
monodromy is therefore equivalent to fixing
\(\eta\in\fz(g)^*\). The admissible curves form the momentum level set
\begin{equation*}
    \mu_g^{-1}(\eta)
    =
    \{\gamma\in\cM_g\mid \mu_g(\gamma)=\eta\}.
\end{equation*}
Equivalently, the constrained variational problem is to find critical
points of
\begin{equation*}
    \pzcL|_{\mu_g^{-1}(\eta)}
    \colon
    \mu_g^{-1}(\eta)\to\RR.
\end{equation*}
Thus the monodromy selects the constraints by selecting the
rigid-motion symmetries that survive.
\section{Discretization}
\label{sec:Discretization}
We now transfer the continuous construction to discrete polygonal curves.
Computationally, this is the central simplification: once length,
area, and volume are defined on a polygon, a discrete elastica is
obtained by constrained optimization over its vertex positions. No
curvature model or curve-flow PDE must be discretized or integrated.
The guiding structural requirement is that the discrete area and volume vectors obey the same transformation laws as their continuous counterparts.

\subsection{Discrete quasi-periodic curves}
Fix \(3\leq N\in\NN\) and let
\(I=\{1,\ldots,N\}\subset\ZZ\) be the fundamental index interval. A
discrete curve is a map
\(
    \gamma\colon\ZZ\to\RR^3,\ i\mapsto\vec\gamma_i.
\)
It is \emph{quasi-periodic} with monodromy
\(g=(R,\vec b)\in\SE(3)\) if for all \(i\in\ZZ\)
\begin{equation*}
    \vec\gamma_{i+N}
    =
    g(\vec\gamma_i)
    =
    R\vec\gamma_i+\vec b.
\end{equation*}
Such a curve is determined by its values on \(I\), namely its \(N\) vertices \((\vec\gamma_1,\ldots,\vec\gamma_N)\) of the 
fun\-da\-men\-tal domain. Closed polygons correspond to
\(g=\id\). We write
\begin{equation*}
    \sfM_g
    =
    \{\gamma\colon\ZZ\to\RR^3
    \mid
    \vec\gamma_{i+N}=g(\vec\gamma_i)\ \forall\,i\in\ZZ\}
\end{equation*}
for the space of discrete quasi-periodic curves\footnote{For the differentiable and Hamiltonian constructions below, we restrict
to polygons with \(\vec\gamma_{i+1}\neq\vec\gamma_i\) for every \(i\), retaining the notation \(\sfM_g\) for this open subset.} with monodromy \(g\).

\paragraph{Piecewise-linear realization}
We denote by \(\cI\) the map sending each \(\gamma\in\sfM_g\) to its
continuous piecewise-linear quasi-periodic interpolant. On the \(i\)-th
edge,
\begin{equation*}
    \vec\gamma_i(s)
    =
    (1-s)\vec\gamma_i+s\,\vec\gamma_{i+1},
    \qquad
    s\in[0,1].
\end{equation*}
Writing \(\vec e_i=\vec\gamma_{i+1}-\vec\gamma_i\), we have
\(d\vec\gamma=\vec e_i\,ds\) on this edge. Vector fields along \(\gamma\), such as variations, are interpolated in
the same way. For \(X\in T_\gamma\sfM_g\), set
\begin{equation*}
    X_i(s)
    \coloneqq
    (1-s)X_i+s\,X_{i+1},
    \qquad
    s\in[0,1].
\end{equation*}
These interpolants are continuous and piecewise smooth, with possible loss of smoothness only at the vertices. Evaluating the integral formulas for \(\pzcL\), \(\pzcA\), \(\pzcV\), and \(\OmegaMW\) edgewise gives well-defined extensions to continuous piecewise-smooth quasi-periodic curves, for which we use the same notation.

\subsection{Discrete functionals}
The length of a discrete polygonal curve is given by 
\begin{equation}
\label{eq:DiscreteLengthEnergy}
    \sfL(\gamma)
    =
    \sum_{i=1}^N|\vec \gamma_{i+1}-\vec \gamma_i|.
\end{equation}

\begin{figure}[t]
    \centering
    \includegraphics[width = \columnwidth]{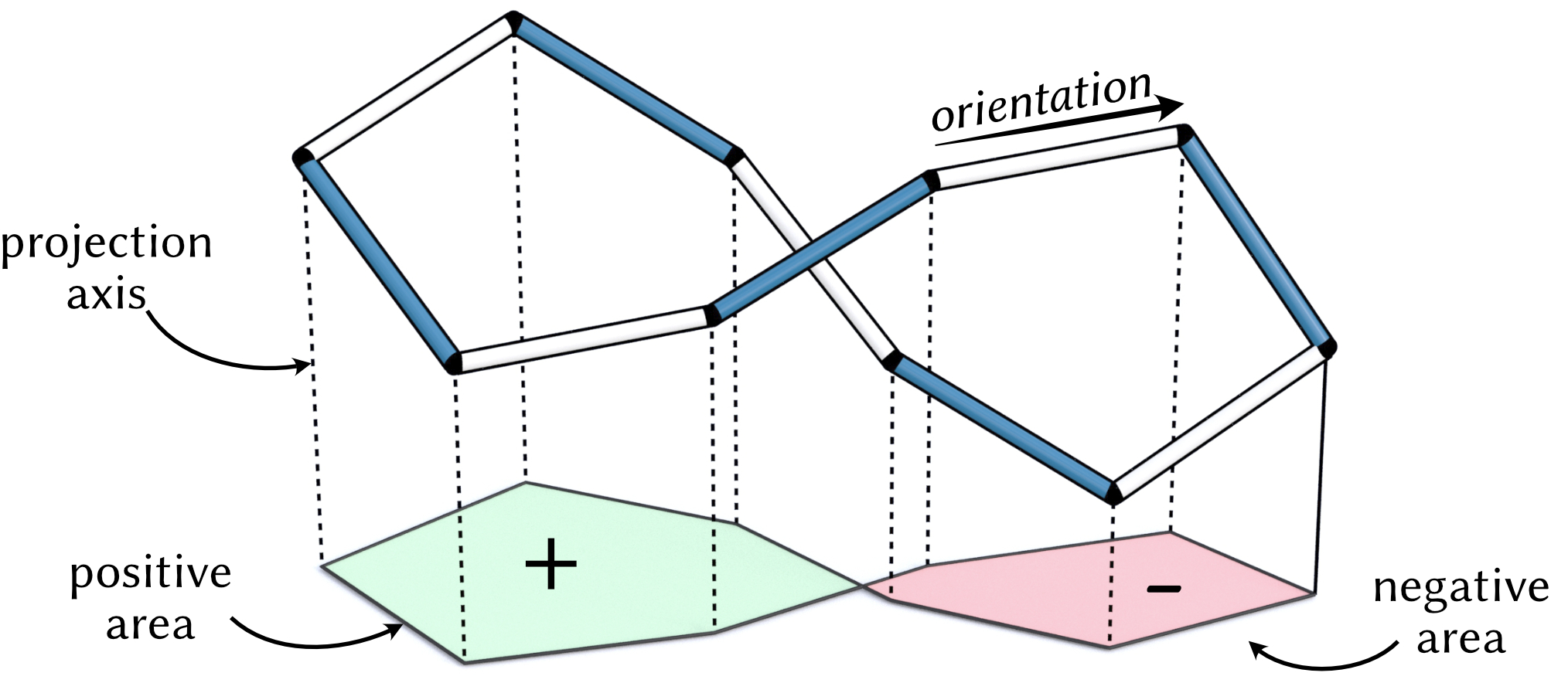}
    \caption{Evaluating the projected-area formula on the piecewise-linear interpolant gives a discrete area vector that satisfies the same rigid-motion transformation law as in the smooth theory.}
    \label{fig:ProjectedAreaDiscrete}
\end{figure}
The discrete area and volume vectors are obtained by evaluating the
continuous formulas on the piecewise-linear interpolant (Figs.~\ref{fig:ProjectedAreaDiscrete}~and~\ref{fig:RevolutionVolumeDiscrete}). On edge \(i\),
the area contribution is
\begin{equation*}
    \tfrac{1}{2}\int_0^1
    \vec\gamma_i(s)\times \vec e_i\,ds
    =
    \tfrac{1}{2}\vec \gamma_i\times\vec \gamma_{i+1}.
\end{equation*}
Summing over the fundamental domain gives 
\begin{equation}
\label{eq:DiscreteArea}
    \sfA(\gamma)
    =
    \tfrac{1}{2}\sum_{i=1}^N
    \vec \gamma_i\times\vec \gamma_{i+1}
    -\tfrac12\vec b\times\vec\gamma_1.
\end{equation}

Similarly, exact evaluation of the continuous volume integral on each
linear edge gives
\begin{equation}
\label{eq:DiscreteVolume}
    \sfV(\gamma)
    =
    \tfrac{1}{3}
    \sum_{i=1}^N
    \frac{\vec \gamma_i+\vec \gamma_{i+1}}{2}
    \times
    (\vec \gamma_i\times\vec \gamma_{i+1})
    -\tfrac16\vec\gamma_1\times(\vec b\times\vec\gamma_1).
\end{equation}

By construction, \(\sfA=\cI^*\pzcA\) and
\(\sfV=\cI^*\pzcV\), where the continuous vectors are understood in the edgewise sense above. Their well-definedness and transformation laws therefore transfer verbatim from the continuous theory (Lemmas~\ref{thm:AreaFundamentalDomainIndependent} and~\ref{thm:VolumeFundamentalDomainIndependent} and \thmref{thm:AreaVolumeTransformation}), giving the following discrete counterpart.

\begin{corollary}
\label{cor:DiscreteTransformation}
The monodromy-compatible components of \(\sfA\) and \(\sfV\) are independent of the chosen discrete fundamental domain and transform exactly by the laws of \thmref{thm:AreaVolumeTransformation}. In particular, for closed polygons the full pair \((\sfV,\sfA)\) transforms by the coadjoint action of \(\SE(3)\), whereas for non-closed polygons the monodromy selects the same components as in the smooth theory.
\end{corollary}

These constraints on discrete curves enter the momentum map below. Restricting \((\sfV,\sfA)\) to \(\fz(g)\) gives a covector in \(\fz(g)^*\).
\begin{figure}[b]
    \centering
    \includegraphics[width = \columnwidth]{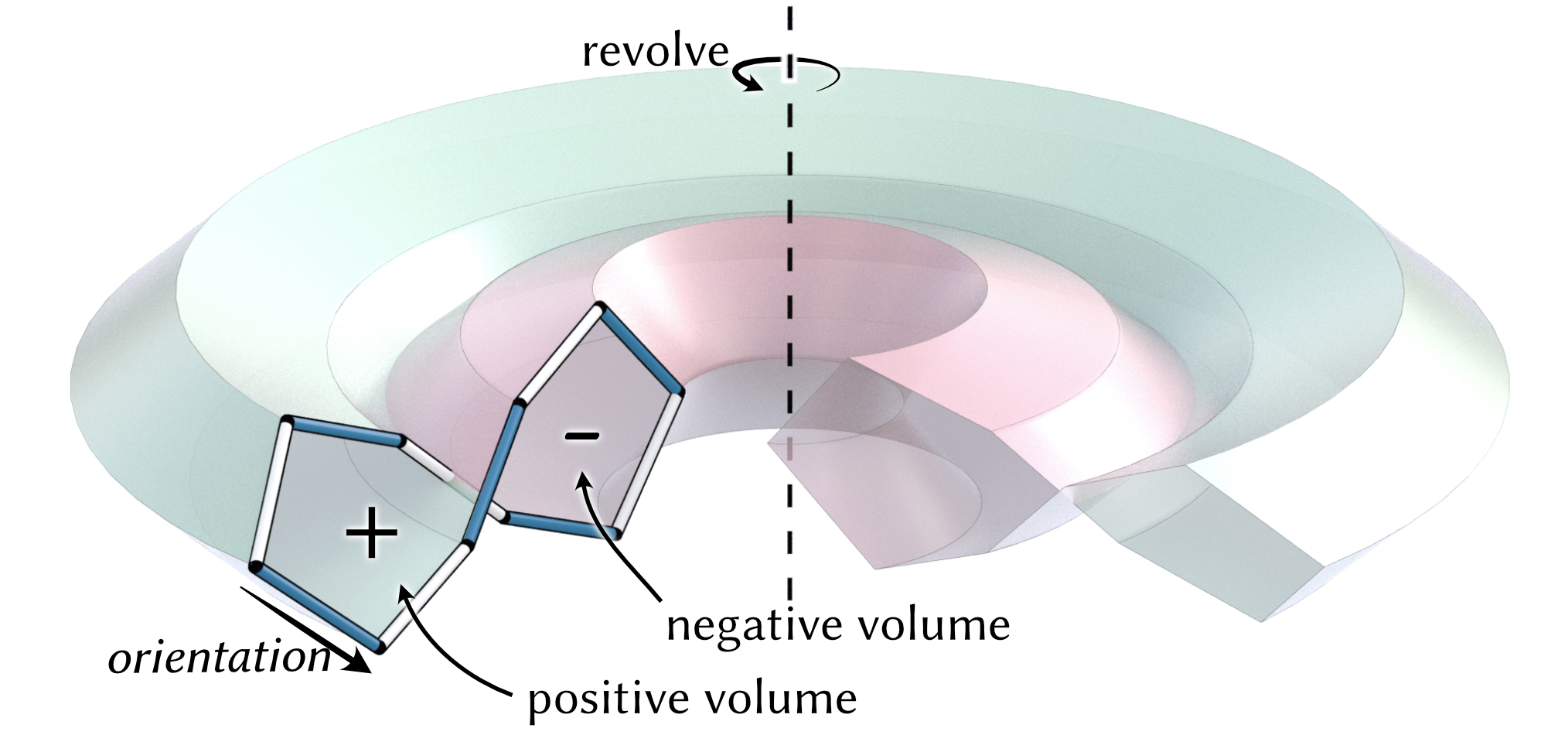}
    \caption{Exact integration along polygon edges gives a discrete volume vector whose transformation law agrees with the coadjoint action of rigid motions.}
    \label{fig:RevolutionVolumeDiscrete}
\end{figure}

\subsection{The discrete Marsden--Weinstein form}
\label{sec:DiscreteMW}
The Marsden--Weinstein form also has a direct discrete counterpart. Using the edgewise extension described above, we define \(\OmegaMWD\coloneqq\cI^*\OmegaMW\), where the differential of \(\cI\)
interpolates variations in the same piecewise-linear basis. The proof of
\thmref{thm:MWSymplectic} applies edgewise, with the boundary terms at
interior vertices canceling, so this extension remains closed. Since the
exterior derivative commutes with pullback,
\begin{equation*}
    d\OmegaMWD
    =
    d(\cI^*\OmegaMW)
    =
    \cI^*(d\OmegaMW)
    =
    0,
\end{equation*}
where we use that \(\OmegaMW\) is closed by \(\thmref{thm:MWSymplectic}\). Moreover, the finite-sum formula makes bilinearity and skew-symmetry immediate, so that we conclude:

\begin{proposition}
\label{thm:DiscreteMWClosed}
    The discrete Marsden--Weinstein form \(\OmegaMWD\) is a pre-symplectic
    \(2\)-form on \(\sfM_g\).
\end{proposition}
\begin{wrapfigure}[5]{r}{0.4\columnwidth}
    \vspace{-2.em}
  {\hspace{-3.em}
    \includegraphics[width=0.5\columnwidth]{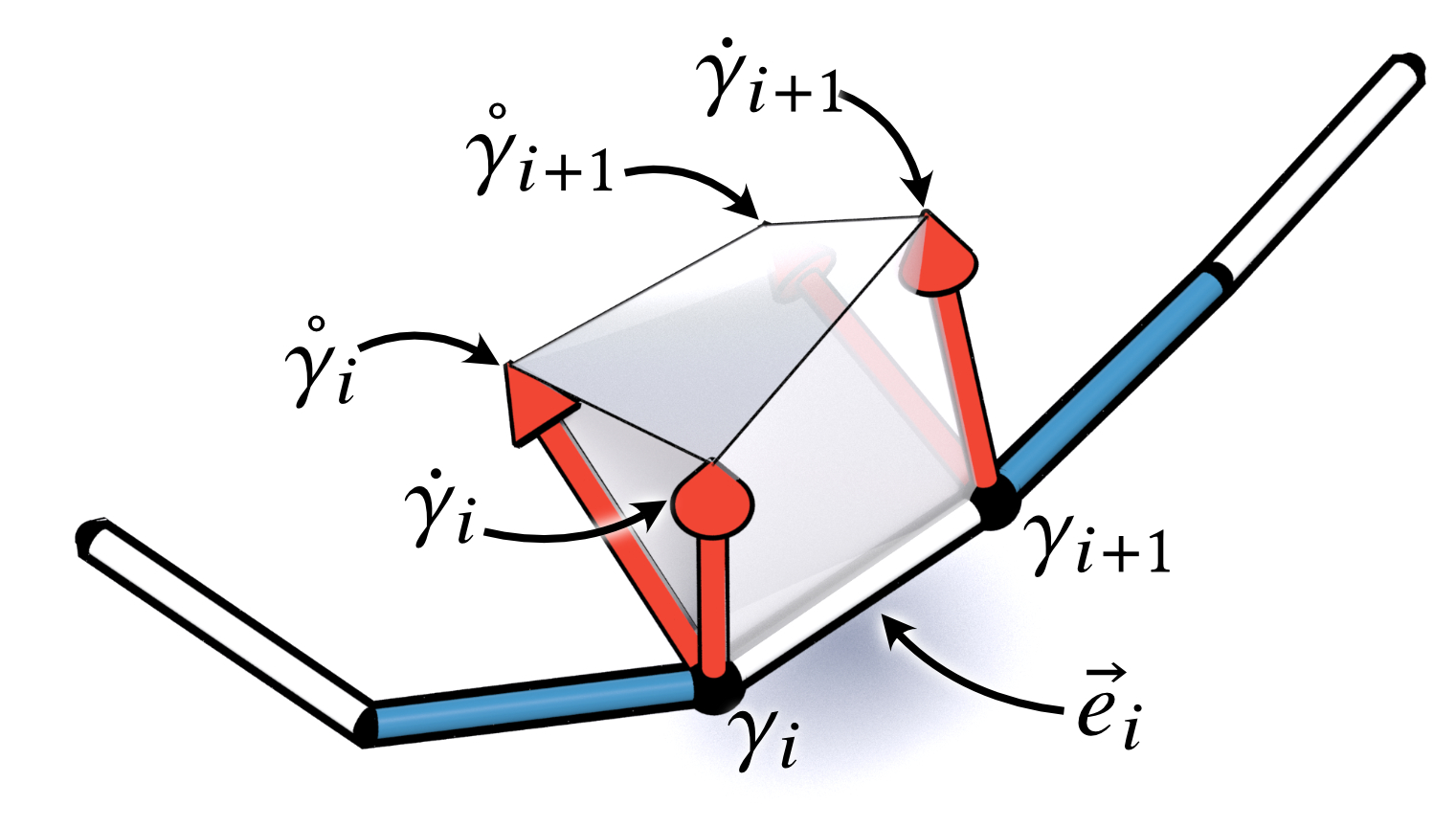}
  }  
\end{wrapfigure}
Geometrically, along each edge the linearly interpolated variations \(\dot\gamma\) and \(\mathring\gamma\) span an infinitesimal oriented area element. Its sweep along the edge determines an infinitesimal signed volume, and~\teqref{eq:DiscreteMWForm} sums these edgewise contributions (see~inset). 

\paragraph{Matrix representation.}
For computations, we assemble an explicit matrix representation \(\Omega\) of \(\OmegaMWD\). For variations \(\dot\gamma,\mathring\gamma\in T_\gamma\sfM_g\), piecewise-linear interpolation gives
\begin{align}
\label{eq:DiscreteMWForm}
\OmegaMWD_\gamma(\dot\gamma,\mathring\gamma)
&=
\sum_{i=1}^N
\int_0^1
\det
\left(
\dot{\vec\gamma}_i(s),
\mathring{\vec\gamma}_i(s),
\vec e_i
\right)\,ds \notag\\
&=
\sum_{i=1}^N
\Bigl[
\tfrac13\det(\dot{\vec\gamma}_i,\mathring{\vec\gamma}_i,\vec e_i)
+
\tfrac16\det(\dot{\vec\gamma}_{i+1},\mathring{\vec\gamma}_i,\vec e_i)
\\
&\qquad\qquad
+
\tfrac16\det(\dot{\vec\gamma}_i,\mathring{\vec\gamma}_{i+1},\vec e_i)
+
\tfrac13\det(\dot{\vec\gamma}_{i+1},\mathring{\vec\gamma}_{i+1},\vec e_i)
\Bigr].\notag
\end{align}
Thus, \(\Omega\) can be assembled directly from the vertex stencil without further integration.

Writing \([\vec e]\colon\RR^3\to\RR^3\) for the skew map
\([\vec e]\,\vec x=\vec e\times\vec x\), and using the convention
\(\OmegaMWD_\gamma(\dot\gamma,\mathring\gamma)
=\mathring{\gamma}^\top\Omega\dot{\gamma}\),
the \(3N\times3N\) matrix \(\Omega\) is block tridiagonal with
wrap-around blocks. Away from the boundary, its three nonzero blocks
in row \(i\) are
\begin{equation*}
    \Omega_{i,i-1}=\tfrac16[\vec e_{i-1}],
    \qquad
    \Omega_{i,i}=\tfrac13[\vec e_{i-1}+\vec e_i],
    \qquad
    \Omega_{i,i+1}=\tfrac16[\vec e_i].
\end{equation*}
The diagonal block collects the two edges incident to vertex \(i\),
while each off-diagonal block comes from the shared edge. At the
boundary, the same local stencil is evaluated on the padded vertices,
using cyclic copies for closed curves and monodromy copies for
quasi-periodic curves, as described in
\secref{sec:BoundaryConditions}. Since
\teqref{eq:DiscreteMWForm} is alternating in its two variations, the
assembled matrix \(\Omega\) is skew-symmetric.

\subsection{The discrete momentum map}
The centralizer \(Z(g)\) acts on \(\sfM_g\) vertexwise, which by construction preserves \(\OmegaMWD\). The momentum map of this action mirrors its continuous counterpart in~\thmref{thm:MomentumMap}.

\begin{theorem}
\label{thm:DiscreteMomentumMap}
    Let \(g\in\SE(3)\), and let
    \(\iota_\fz\colon\fz(g)\hookrightarrow\se(3)\) denote the inclusion
    of the Lie algebra of its centralizer. The discrete momentum map for
    the \(Z(g)\)-action on \((\sfM_g,\OmegaMWD)\) is
    \begin{equation*}
        \mu_g^{\mathrm{disc}}\colon \sfM_g\to\fz(g)^*,\ \gamma\mapsto
        \iota_\fz^*
        \begin{pmatrix}
            \sfV(\gamma) \\
            \sfA(\gamma)
        \end{pmatrix}.
    \end{equation*}
\end{theorem}
\begin{proof}
    See~\appref{sec:DiscreteMomentumMapProof}.
\end{proof}
Unpacking the restriction to \(\fz(g)\) yields the discrete analogues of the three cases in the continuous theory:
\begin{enumerate}[(i)]
    \item For \(g=\id\),
    \begin{equation*}
        \mu_g^{\mathrm{disc}}(\gamma)
        =
        \begin{pmatrix}
            \sfV(\gamma) \\
            \sfA(\gamma)
        \end{pmatrix}.
    \end{equation*}

    \item For pure translational monodromy, \(R=\Id\) and \(\vec b\neq 0\),
    \begin{equation*}
        \mu_g^{\mathrm{disc}}(\gamma)
        =
        \begin{pmatrix}
            \langle \nicefrac{\vec b}{|\vec b|},\sfV(\gamma)\rangle \\
            \sfA(\gamma)
        \end{pmatrix}.
    \end{equation*}

    \item For screw monodromy, \(R\neq \Id\),
    \begin{equation*}
        \mu_g^{\mathrm{disc}}(\gamma)
        =
        \begin{pmatrix}
            \langle \vec s_g,\sfV(\gamma)\rangle \\
            \langle \vec s_g,\sfA(\gamma)\rangle
        \end{pmatrix},
    \end{equation*}
    where \(\vec s_g\in S^2\) is chosen such that \(\RR\vec s_g\) is the screw axis.
\end{enumerate}

The discrete area and volume components are therefore the Noether charges of the Hamiltonian \(Z(g)\)-action on the pre-symplectic manifold \((\sfM_g,\OmegaMWD)\).

\subsection{Discrete Elastic Curves}
\label{sec:DiscreteElasticCurves}
Finally, building on the discrete structural properties established above, we define our notion of discrete elastic curves, motivated by the continuous isoperimetric~\thmref{thm:Isoperimetric}. 

In the smooth theory, \(\pzcL\), \(\pzcA\), and \(\pzcV\) are invariant
under reparametrization. An \(N\)-vertex polygon has no analogous
continuous freedom, since resampling generally changes its geometry.
We therefore fix the discrete arc-length gauge by requiring equal edge
lengths and denote the resulting space of quasi-periodic polygons with
monodromy \(g\in\SE(3)\) by
\begin{equation*}
\sfM_g^{\mathrm{arc}}
\coloneqq
\{\gamma\in\sfM_g \mid |\vec e_i|=|\vec e_j|,\ \forall i,j\in I\}.
\end{equation*}

\begin{definition}
    A curve \(\gamma\in\sfM_g^{\mathrm{arc}}\) is \emph{discrete elastic} if it is a critical point of the discrete length functional \(\sfL\) for fixed discrete momentum data \(\mu_g^{\mathrm{disc}}(\gamma)\in\fz(g)^*\).
\end{definition}
Therefore, in the computations below, we seek discrete elastica that are constrained local minimizers of the \emph{discrete isoperimetric problem} 
\begin{equation}
\label{eq:DiscreteIsoperimetricProblem}
        \begin{cases}
        \displaystyle \min_{\gamma\in\sfM_g^{\mathrm{arc}}} \ \sfL(\gamma), \\[0.4em]
        \text{s.t.} \quad \mu_g^{\mathrm{disc}}(\gamma)=\mu_0 .
    \end{cases}
\end{equation}
Here \(\mu_g^{\mathrm{disc}}\) is the discrete momentum map of~\thmref{thm:DiscreteMomentumMap}, so that the momentum constraint in~\teqref{eq:DiscreteIsoperimetricProblem} fixes exactly those components of \((\sfV,\sfA)\) that are compatible with the monodromy \(g\). 

We thus obtain an explicit finite-dimensional variational problem posed directly on polygon vertices and whose solutions exhibit consistent geometries across different resolutions (Figs.~\ref{fig:Teaser},~\ref{fig:ResolutionAblationClosed},~and~\ref{fig:SmokeRingFlow_ResolutionAblation}). This formulation only involves the geometric quantities \(\sfL\), \(\sfA\), and \(\sfV\) that can be evaluated directly on discrete polygonal curves and no further auxiliary discretization is required.

\section{Numerical Optimization}
\label{sec:NumericalOptimization}
The discrete isoperimetric
problem~\eqref{eq:DiscreteIsoperimetricProblem} is a low-order,
finite-dimen\-sional constrained problem with the polygon vertices as
variables and can therefore be solved directly using standard
optimization methods.

\subsection{Quasi-periodicity}
\label{sec:BoundaryConditions}
Quasi-periodicity can be imposed in two ways, akin to the two characterizations in~\thmref{thm:Isoperimetric}.
One restricts to quasi-periodic curves in \(\cM_g\) and enforces the monodromy as a boundary condition. The other, which we use, treats the curve as de facto infinite and generates whatever lies outside the fundamental domain from the monodromy. Either way only the \(N\) vertices \(\vec\gamma_1,\ldots,\vec\gamma_N\) of one fundamental domain are optimized, and curves of arbitrary length can be obtained from the monodromy relation at no additional cost.

Concretely, whenever a stencil reaches outside the fundamental domain, the required vertices are generated from the monodromy,
\begin{equation*}
    \vec\gamma_{N+1}=g(\vec\gamma_1),
    \qquad
    \vec\gamma_0=g^{-1}(\vec\gamma_N).
\end{equation*}
Since the functionals used here depend only on adjacent edges, one ghost vertex on each side contains all local information from the infinite
curve. We evaluate the functionals on the padded curve and pull their gradients back through the padding map to the \(N\) fundamental-domain
variables. Thus the monodromy is handled entirely by padding, while the optimizer sees only the fundamental variables.

\begin{figure}
    \centering
    \includegraphics[width=\columnwidth]{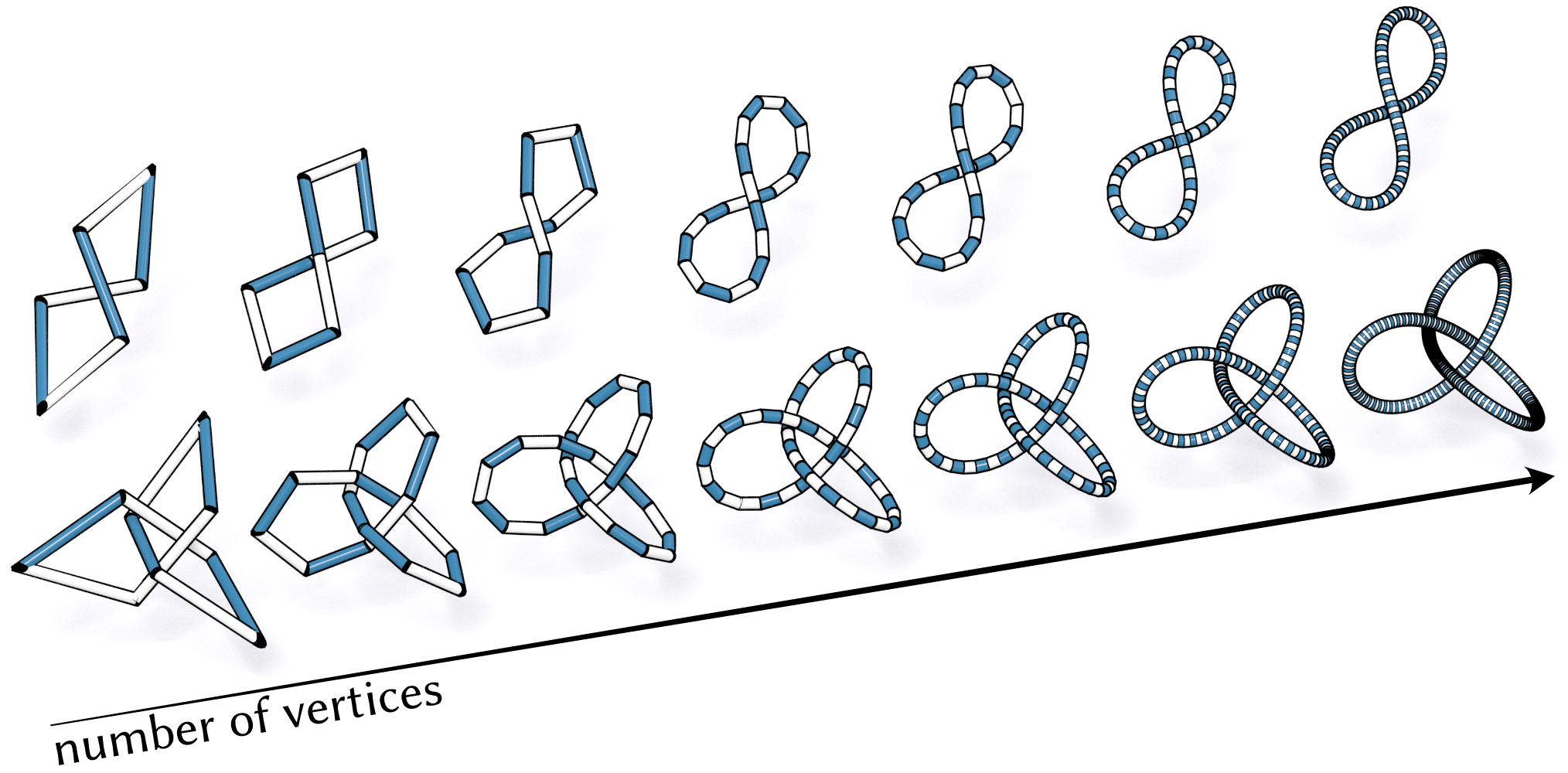}
    \caption{The two rows show polygonal length minimizers at prescribed
area and volume for two closed curves, with the number of vertices
increasing from left to right. Even at the coarsest resolutions, the
solutions faithfully reproduce the distinctive shapes of the
corresponding smooth elastica, while refinement progressively improves
the agreement.}
    \label{fig:ResolutionAblationClosed}
\end{figure}

\subsection{Constraints}
\label{sec:NumericalConstraints}

To identify solutions to \teqref{eq:DiscreteIsoperimetricProblem}, we solve
the constrained optimization problem
\begin{equation}
\label{eq:DiscreteIsoperimetricProblemComputation}
    \begin{cases}
        \displaystyle \min_{\gamma\in\sfM_g} \ \sfL(\gamma), \\[0.4em]
        \text{s.t.} \quad C_g(\gamma)=0 .
    \end{cases}
\end{equation}
Here \(C_g\colon\sfM_g\to\RR^m\) is the constraint map, comprised of momentum and equal-edge constraints. Writing
\(\ell(\gamma)\coloneqq\tfrac{1}{N}\sfL(\gamma)\), 
we set\footnote{The equal-edge residuals sum identically to zero, so only \(N-1\) are independent. We retain all \(N\) to preserve cyclic symmetry. This leaves only an additive indeterminacy in the associated multipliers and does not affect the primal problem.}
\begin{equation}
    \label{eq:DiscreteConstraintMap}
    C_g(\gamma)
    =
    \begin{pmatrix}
        \mu_g^{\mathrm{disc}}(\gamma)-\mu_0 \\
        |\vec e_1|-\ell(\gamma) \\
        \vdots \\
        |\vec e_N|-\ell(\gamma)
    \end{pmatrix}.
\end{equation}
Thus the computed stationary points are discrete arc-length param\-etrized
curves satisfying the constraints on the components of \((\sfV,\sfA)\)
compatible with the monodromy \(g\).

\begin{remark}
The equal-edge constraints fix the discrete arc-length gauge but leave residual symmetries: cyclic relabeling for closed polygons and rigid motions \(h\in Z(g)\) preserving the prescribed momentum data. These leave the geometry unchanged but may appear numerically as null or near-null directions.
\end{remark}

\subsection{Optimization}
\label{sec:AugmentedLagrangianOptimization}

The input is a monodromy \(g\in\SE(3)\), an initial polygon
\(\gamma^0\in\sfM_g\), and prescribed momentum data \(\mu_0\in\fz(g)^*\).
The output is a polygonal curve approximating a stationary point of discrete
length at fixed discrete momentum.

We solve \teqref{eq:DiscreteIsoperimetricProblemComputation} by a
standard augmented Lagrangian method. For multipliers
\(\lambda\in\RR^m\) and penalty parameter \(\rho>0\), define
\begin{equation}
\label{eq:AugmentedLagrangian}
    \mathcal L_\rho(\gamma,\lambda)
    =
    \sfL(\gamma)
    +
    \langle\lambda\,\vert\,C_g(\gamma)\rangle
    +
    \tfrac{\rho}{2}\|C_g(\gamma)\|^2 .
\end{equation}
The resulting iteration is summarized in
\algref{alg:SolveDiscreteIsoperimetricProblem}. Each inner problem is
solved by a quasi-Newton method using the analytic gradient
\begin{equation}
    \label{eq:AugmentedLagrangianGradient}
    \nabla_\gamma\mathcal L_\rho
    =
    \nabla_\gamma\sfL
    +
    dC_g(\gamma)^\top
    \bigl(\lambda+\rho C_g(\gamma)\bigr).
\end{equation}
The low differential order also keeps the reduced Hessian well conditioned
under refinement, which is what makes a quasi-Newton method practical here
(\figref{fig:ConvergenceAndConditioning}).

The length and equal-edge terms have explicit local derivatives, while
\(\sfA\) and \(\sfV\) are polynomial in the vertex positions. Function
and gradient evaluations can therefore be assembled edgewise at linear
cost in \(N\). The same routine applies to every resolution and
monodromy type, with \(g\) entering only through the padded boundary
data and the momentum components selected by \(C_g\).\\

\begin{algorithm}[h]
\caption{\textbf{DiscreteIsoperimetricProblem}
\((g,\gamma^0,\mu_0)\)}
\label{alg:SolveDiscreteIsoperimetricProblem}
\begin{algorithmic}[1]
\Require \(g\in\SE(3)\), \(\gamma^0\in\sfM_g\),
\(\mu_0\in\fz(g)^*\)
\Ensure discrete elastic \(\gamma\in\sfM_g^{\mathrm{arc}}\) with
\(\mu_g^{\mathrm{disc}}(\gamma)=\mu_0\)
\State \(\gamma\gets\gamma^0\)
\State initialize \(\lambda\) and \(\rho>0\)
\While{not converged}
    \State
    \(\widehat\gamma\gets\Call{Pad}{\gamma,g}\)
    \Comment{\secref{sec:BoundaryConditions}}
    \State
    \(C_g\gets\Call{ConstraintMap}{g,\mu_0}\)
    \Comment{\teqref{eq:DiscreteConstraintMap}}
    \State
    \(\widehat\gamma\vert_I
    \gets
    \argmin
    \mathcal L_\rho(\,\cdot\,\vert_I,\lambda)\)
    \Comment{Eqs.~\eqref{eq:AugmentedLagrangian}
    and~\eqref{eq:AugmentedLagrangianGradient}}
    \State
    \(\gamma\gets\Call{DiscardPadding}{\widehat\gamma}\)
    \State
    \(\lambda\gets\lambda+\rho\,C_g(\gamma)\)
    \If{the constraint residual stalls}
        \State increase \(\rho\)
    \EndIf
\EndWhile
\State \Return \(\gamma\)
\end{algorithmic}
\end{algorithm}

\begin{remark}
In practice, a useful warm start is obtained by first minimizing the discrete
Dirichlet energy
\(
    \sfE(\gamma)=\tfrac12\sum_{i=1}^N|\vec e_i|^2
\)
subject only to the momentum constraints. After approximate convergence,
we impose the equal-edge constraints and solve the original length
problem in the discrete arc-length gauge.
\end{remark}

\section{Hamiltonian Flows on Space Curves}
\label{sec:HamiltonianFlowsOnDiscreteSpaceCurves}
The discrete Marsden--Weinstein form of \secref{sec:DiscreteMW} not only
identifies the isoperimetric constraints as momentum data but also
provides, as in the smooth theory, the phase-space structure underlying
Hamiltonian evolutions of curves. In this section, we apply this construction directly to polygonal curves, focusing on discrete analogues of tangential reparametrization, vortex-filament motion, and the modified Korteweg--de Vries flow. While the framework is more
general, these examples are particularly natural because their Hamiltonians can be evaluated directly on polygons.
\begin{figure}[t]
    \centering
    \includegraphics[width=\columnwidth]
    {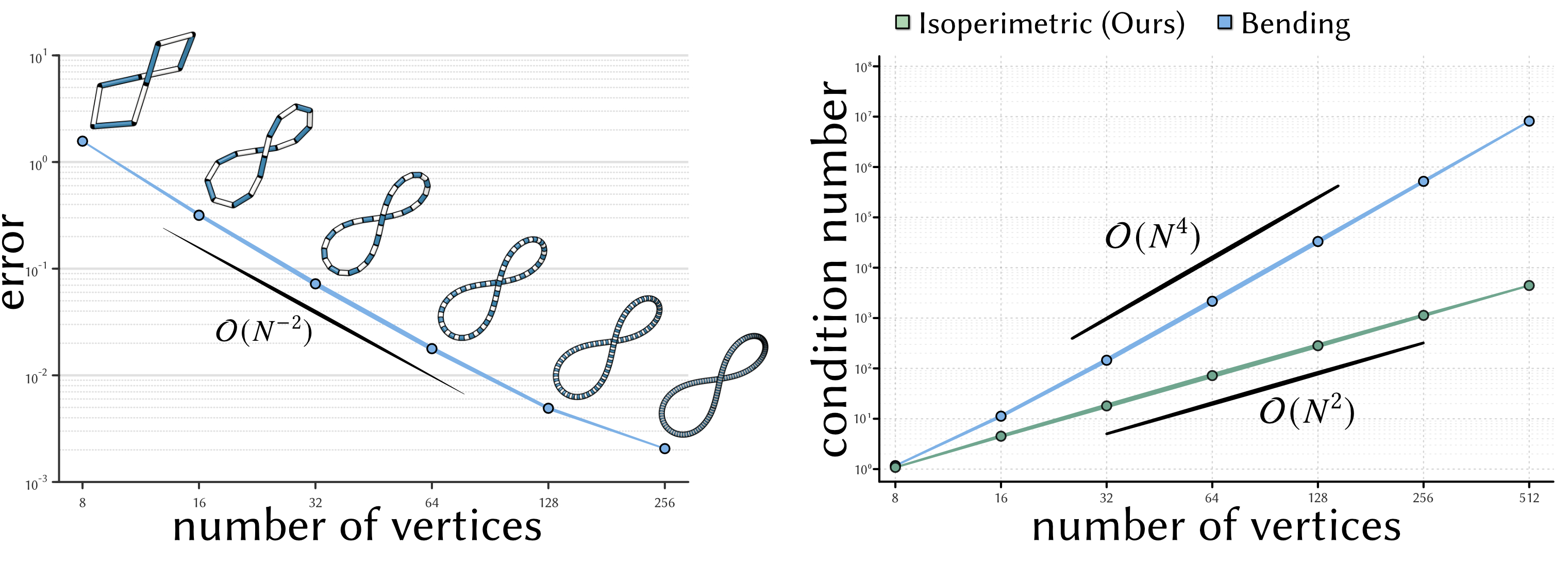}
    \caption{\emph{Left:} Our discrete elastica converge quadratically to the smooth reference solution under refinement. \emph{Right:} The reduced Hessian of our isoperimetric formulation stays far better conditioned, growing like \(N^2\) against \(N^4\) for the bending-energy baseline, as their second- and fourth-order stationarity conditions predict.}
    \label{fig:ConvergenceAndConditioning}
\end{figure}

\subsection{Spectrum of Marsden--Weinstein form}
\label{sec:SpectrumOfMW}
In the continuous setup, the Marsden--Weinstein form is a pre-symplectic closed \(2\)-form whose only degeneracy is given by repara\-metrization of curves. To investigate how the properties of the Marsden--Weinstein form carry over to the discrete setting, we examine the spectrum of its associated linear operator.

\subsubsection{Spectrum of the continuous Marsden--Weinstein form}
Since \(\OmegaMW_\gamma\) is a \(2\)-form rather than an operator, we define its spectrum through its \(L^2\)-Riesz representative \(J_\gamma\). For \(X,Y\in T_\gamma\cM_g\) it is given by
\begin{equation*}
    \OmegaMW_\gamma(X,Y)
    =
    \langle J_\gamma X,Y\rangle_{L^2}.
\end{equation*}
\teqref{eq:MWForm} yields
\begin{equation*}
    J_\gamma X=T\times X,
\end{equation*}
where \(T=\gamma'\), so \(J_\gamma\) is pointwise cross product with the unit tangent.

The kernel of \(J_\gamma\) consists precisely of the tangential fields \(X=fT\), where \(f\colon\RR\to\RR\) is smooth and \(\tau\)-periodic. These fields generate infinitesimal
reparametrizations. On normal
fields,
\begin{equation*}
    J_\gamma^2=-\mathrm{id}.
\end{equation*}
Hence \(J_\gamma\) induces a complex structure on
\(T_\gamma\cM_g/\ker J_\gamma\), naturally identified with the normal
bundle~\cite{Chern:2020:CHF}. This is the local linear model of the
reduced Marsden--Weinstein form.

Together, these identities show that the spectrum of \(J_\gamma\) on
vector fields along \(\gamma\) is
\begin{equation*}
\sigma(J_\gamma)=\{0,\pm i\}.
\end{equation*}
The zero eigenspace corresponds to tangential directions, while the eigenvalues \(\pm i\) are those of the \(90^\circ\)-rotation induced by \(J_\gamma X=T\times X\) on each normal plane.

\begin{table*}[t]
\centering
\small
\begin{tabular}{clll}
\toprule
\(H\) & \(dH_\gamma\) & \(X_H\) & Interpretation \\
\midrule
\(\langle \vec v,\pzcV\rangle\)
&
\(\vec\gamma'\times(\vec v\times\vec\gamma)\)
&
\(X_{\langle \vec v,\pzcV\rangle}=\vec v\times\vec\gamma\)
&
rotation about \(\vec v\)
\\

\(\langle \vec v,\pzcA\rangle\)
&
\(\vec\gamma'\times \vec v\)
&
\(X_{\langle \vec v,\pzcA\rangle}=\vec v\)
&
translation along \(\vec v\)
\\

\(0\)
&
\(0\)
&
\(X_0=\vec\gamma'\)
&
tangent flow
\\

\(\pzcL\)
&
\(-\vec\gamma''\)
&
\(X_\pzcL=\vec\gamma'\times\vec\gamma''\)
&
vortex-filament flow
\\

\(\Phi\)
&
\(-\vec\gamma'\times\vec\gamma'''\)
&
\(X_\Phi=-\vec\gamma'''-\tfrac{3}{2}|\vec\gamma''|^2\vec\gamma'\)
&
mKdV flow
\\

\(\pzcB\)
&
\(\left(\vec\gamma'''+\tfrac{3}{2}|\vec\gamma''|^2\vec\gamma'\right)'\)
&
\(X_\pzcB=-\vec\gamma'\times\vec\gamma''''-\tfrac{3}{2}|\vec\gamma''|^2\vec\gamma'\times\vec\gamma''+\det(\vec\gamma',\vec\gamma'',\vec\gamma''')\vec\gamma'\)
&
bending-energy flow
\\
\bottomrule
\end{tabular}
\caption{Low-order Hamiltonians and the corresponding flows in the localized-induction hierarchy. Here \(\vec v\in\mathbb{R}^3\) is a fixed unit vector. The quantities \(\pzcA\) and \(\pzcV\) are the area and volume vectors, \(\Phi\) is the holonomy angle, and \(\pzcB=\tfrac{1}{2}\int |\vec \gamma''|^2\,ds\) is the bending energy.}
\label{tab:localized-induction-hierarchy}
\end{table*}

\subsubsection{Spectrum of the discrete Marsden--Weinstein form}
\label{sec:SpectrumOfTheDiscreteMarsdenWeinsteinForm}
To compare the discrete form with the smooth operator \(J_\gamma\), let
\(\Omega\) be the matrix assembled from
\teqref{eq:DiscreteMWForm} and equip the vertex fields with the
finite-element \(L^2\) metric represented by the mass matrix
\begin{equation}
\label{eq:FEMassMatrix}
    (\sfM_\gamma X)_i
    =
    \tfrac{|\vec e_{i-1}|}{6}X_{i-1}
    +
    \tfrac{|\vec e_{i-1}|+|\vec e_i|}{3}X_i
    +
    \tfrac{|\vec e_i|}{6}X_{i+1}.
\end{equation}
The corresponding mass representative \(\sfJ_\gamma\) is defined by
\begin{equation*}
    \OmegaMWD_\gamma(X,Y)
    =
    \langle\sfJ_\gamma X,Y\rangle_{\sfM_\gamma}.
\end{equation*}
Thus \(\sfJ_\gamma=\sfM_\gamma^{-1}\Omega\), and its spectrum is
obtained from the generalized eigenvalue problem
\begin{equation*}
    \Omega X=\lambda\sfM_\gamma X,
\end{equation*}
or equivalently from the similar Euclidean skew-symmetric matrix
\begin{equation*}
    \sfM_\gamma^{-1/2}\Omega\sfM_\gamma^{-1/2}.
\end{equation*}
Since \(\sfM_\gamma^{-1/2}\Omega\sfM_\gamma^{-1/2}\) is real and
skew-symmetric, the spectrum of \(\sfJ_\gamma\) is purely imaginary.
\figref{fig:SpectrumJ} shows it under refinement for three curves\footnote{The curves are \(\gamma(t) = (\cos t, \sin t, 0)\),
\(\bigl((2+\cos(3t+\tfrac{3}{10}))\cos 2t, (2+\cos(3t+\tfrac{3}{10}))\sin 2t,
\sin(3t+\tfrac{3}{10})\bigr)\), and
\(\bigl(\cos t, \sin t, \tfrac{7}{20}\sin(2t+\tfrac{7}{10}) +
\tfrac{1}{4}\sin(3t+\tfrac{3}{10})\bigr)\), sampled uniformly in \(t\).}. The
eigenvalues separate into the three groups predicted by the pointwise
model, one near \(0\) for the tangential reparametrization directions and
two approaching \(\operatorname{Im}\lambda = \pm 1\) for the normal planes.\\

\begin{figure}[h]
    \centering
    \includegraphics[width=\columnwidth]
    {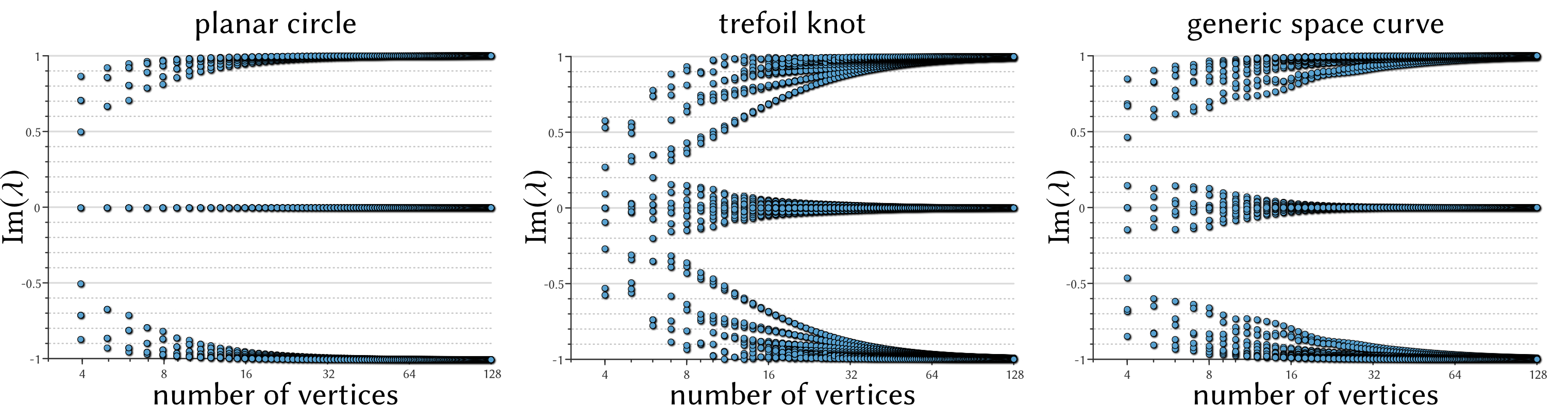}
    \caption{The spectrum of \(\sfJ_\gamma\) is purely imaginary, and for a planar circle, a trefoil knot, and a generic space curve alike it splits under refinement into \(N\) eigenvalues near \(0\) along the discrete reparametrization directions and \(2N\) eigenvalues in two branches approaching \(\pm i\), one for each normal plane. On the planar circle the near-zero group vanishes to machine precision at every resolution, whereas on the two space curves it only contracts under refinement.
}
    \label{fig:SpectrumJ}
\end{figure}

\subsubsection{Reparametrization directions}
\label{sec:ReparametrizationDirections}
The near-zero cluster has a direct geometric interpretation. In the smooth theory, the tangential fields \(X=fT\) are precisely the infinitesimal reparametrizations and form the zero eigenspace of
\(J_\gamma\), while the normal fields give the two branches at \(\pm i\).

An \(N\)-vertex polygon has no canonical tangent at its vertices. Nevertheless, \(\sfJ_\gamma\) has \(N\) small eigenvalues, one scalar degree
of freedom per vertex. We take the corresponding spectral subspace \(\cR_\gamma \subset T_\gamma\cM_g\) as the discrete reparametrization
directions. 

Planar polygons are an exception. All edges lie in a common plane, so \(\Omega\) sends every in-plane field to a multiple of the unit normal \(\vec n\) of said plane. This maps \(2N\) dimensions into \(N\) and leaves \(\ker\sfJ_\gamma = \ker\Omega\) at least \(N\)-dimensional. We observe this for the planar circle in \figref{fig:SpectrumJ}, whose near-zero cluster sits at machine precision at every resolution.

Removing \(\cR_\gamma\) is the discrete analogue of passing to shape space.
In the Hamiltonian constructions below, \(\cR_\gamma\) is removed for the
vortex-filament flow, whereas the tangent flow selects a distinguished
edge-length-preserving direction associated with this near-kernel. The mKdV
flow is obtained from an edge-length-preserving weighted least-squares solve.

\subsection{The localized-induction hierarchy}
We now derive several Hamiltonian curve flows from the
Marsden--Weinstein form. In the smooth setting, these flows belong to the well-studied integrable localized-induction hierarchy~\cite{Chern:2020:CHF,LangerSinger:1996:LAK} (\tabref{tab:localized-induction-hierarchy}), making them natural benchmarks for transferring the Hamiltonian construction to polygonal curves.
\begin{figure}[b]
    \centering
    \includegraphics[width=\columnwidth]
    {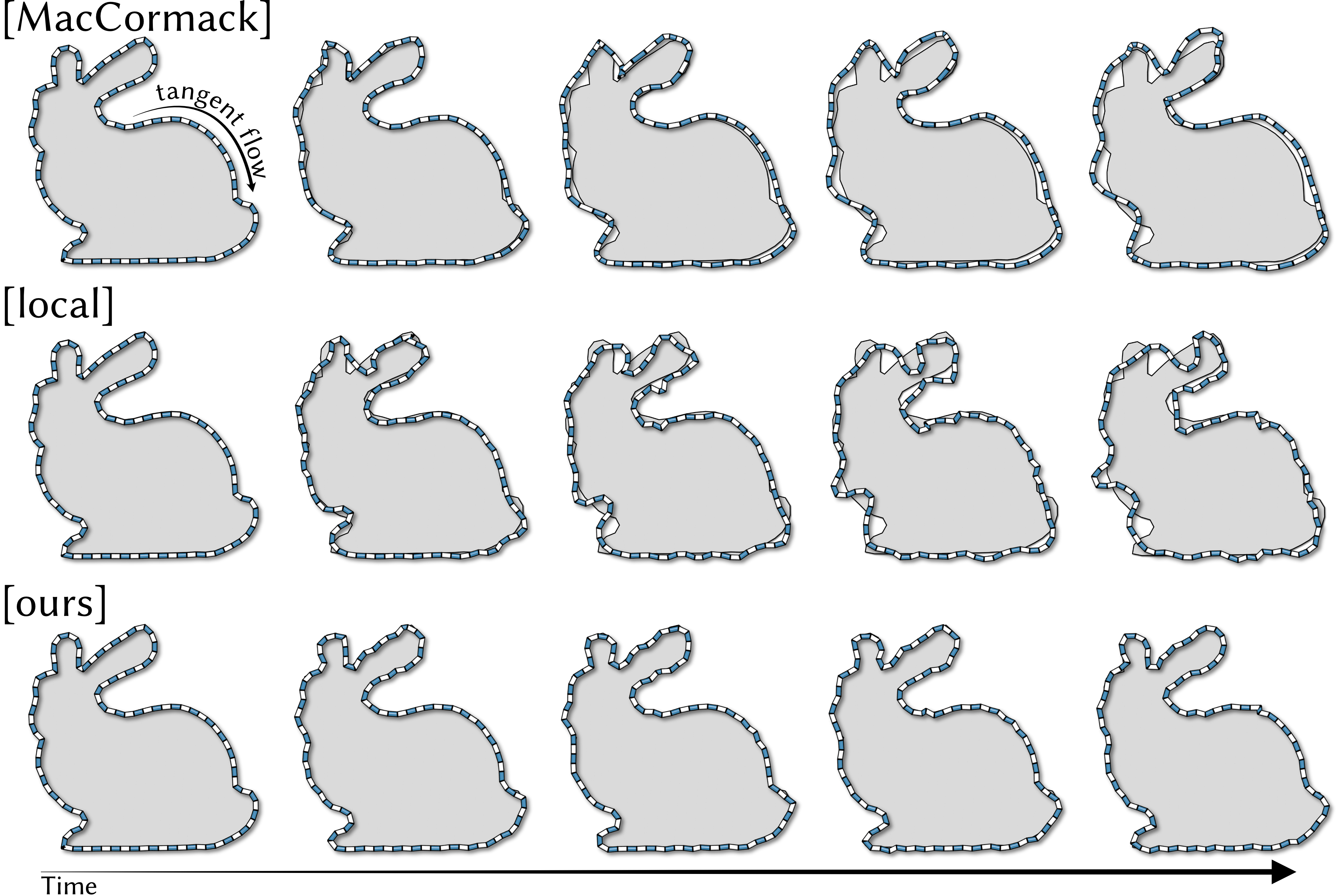}
    \caption{On the planar bunny, MacCormack advection and the local
    M\"obius flow visibly deform the shape over long times, while the
    Hamiltonian tangent flow largely preserves its geometric features.}
    \label{fig:TangentFlow_PlanarBunny_Comparison}
\end{figure}

For a Hamiltonian \(H\), its vector field \(X_H\) is defined by
\begin{equation}
\label{eq:DefHamiltonianFlow}
    \iota_{X_H}\OmegaMW_\gamma=dH_\gamma .
\end{equation}
Because \(\OmegaMW\) is pre-symplectic, \(X_H\) is determined only
modulo \(\ker\OmegaMW_\gamma\), the reparametrization directions. The
low-order cases used below are listed in
\tabref{tab:localized-induction-hierarchy}.

For a constant Hamiltonian, \(dH=0\), so its Hamiltonian vector fields
lie in \(\ker\OmegaMW_\gamma\). The choice
\begin{equation*}
    X_0=\vec\gamma'
\end{equation*}
gives the tangent flow, which only reparametrizes the curve and leaves
its geometry unchanged (Figs.~\ref{fig:TangentFlow_PlanarBunny_Comparison}~and~\ref{fig:TangentFlow_4FoldSymmetry_Comparison}).

The length functional generates the localized-induction, or
vortex-filament, flow (Figs.~\ref{fig:SmokeRingFlowEllipseComparison}
and~\ref{fig:SmokeRingFlow_ResolutionAblation})
\begin{equation*}
    X_\pzcL=\vec\gamma'\times\vec\gamma''.
\end{equation*}

The highest-order Hamiltonian in this hierarchy that remains naturally
defined on polygons is the holonomy \(\Phi\), which generates the
modified Korteweg--de Vries flow (\figref{fig:mKdVSolitonOnSphere})
\begin{equation*}
    X_\Phi
    =
    -\vec\gamma'''
    -\tfrac32|\vec\gamma''|^2\vec\gamma'.
\end{equation*}

\subsection{Hamiltonian flows on discrete curves}
\label{sec:discretizing-the-flows}
Although the smooth tangent, vortex-filament, and mKdV flows involve
tangents, curvature, and higher derivatives, their Hamiltonians admit
natural polygonal counterparts. For a discrete Hamiltonian \(\sfH\),
we therefore define its polygonal vector field by
\begin{equation}
\label{eq:DefHamitlonianFlowDiscrete}
    \OmegaMWD_\gamma X_\sfH=d\sfH_\gamma .
\end{equation}

The discrete Marsden--Weinstein form does not exactly reproduce the reparametrization kernel of its smooth counterpart. 
Its near-kernel \(\cR_\gamma\) makes~\teqref{eq:DefHamitlonianFlowDiscrete} ill-conditioned along its small singular modes. These modes stay separated from the rest of the spectrum under refinement (\figref{fig:SpectrumJ}).
For the tangent flow, we select a distinguished direction associated with this near-kernel. For the vortex-filament flow, we remove \(\cR_\gamma\), while for the mKdV flow, we impose edge-length preservation in a weighted least-squares solve (Sections~\ref{sec:TangentFlow}--\ref{sec:mKdVFlow}). These choices make the corresponding discrete vector fields numerically tractable.

The next member of the hierarchy, generated by the bending energy,
lies outside this intrinsic construction because it requires an
auxiliary curvature discretization.

\subsubsection{Tangent flow}
\label{sec:TangentFlow}
For a constant Hamiltonian, the Hamiltonian equation reduces to
\begin{equation}
\label{eq:HamiltonianTangentFlow}
    \OmegaMWD_\gamma X_0=0 .
\end{equation}
Since this homogeneous equation does not canonically select a nonzero
discrete field, we impose unit norm and infinitesimal edge-length
preservation and minimize its residual. Let \(\cC_\gamma\) denote the
edge-length differential. We define
\begin{equation*}
    X_0
    \coloneqq
    \argmin_{\substack{
        X\in\ker\cC_\gamma\\
        \|X\|_{\sfM_\gamma}=1}}
    \|\OmegaMWD_\gamma X\|_{\sfM_\gamma^{-1}}^2 .
\end{equation*}
Here \(\sfM_\gamma\) represents the \(L^2\) metric from
\teqref{eq:FEMassMatrix}, and \(\sfM_\gamma^{-1}\) represents its dual
metric on covectors. Thus \(X_0\) is the unit edge-length-preserving
field of minimal Marsden--Weinstein residual.

\begin{figure}[t]
    \centering
    \includegraphics[width=\columnwidth]
    {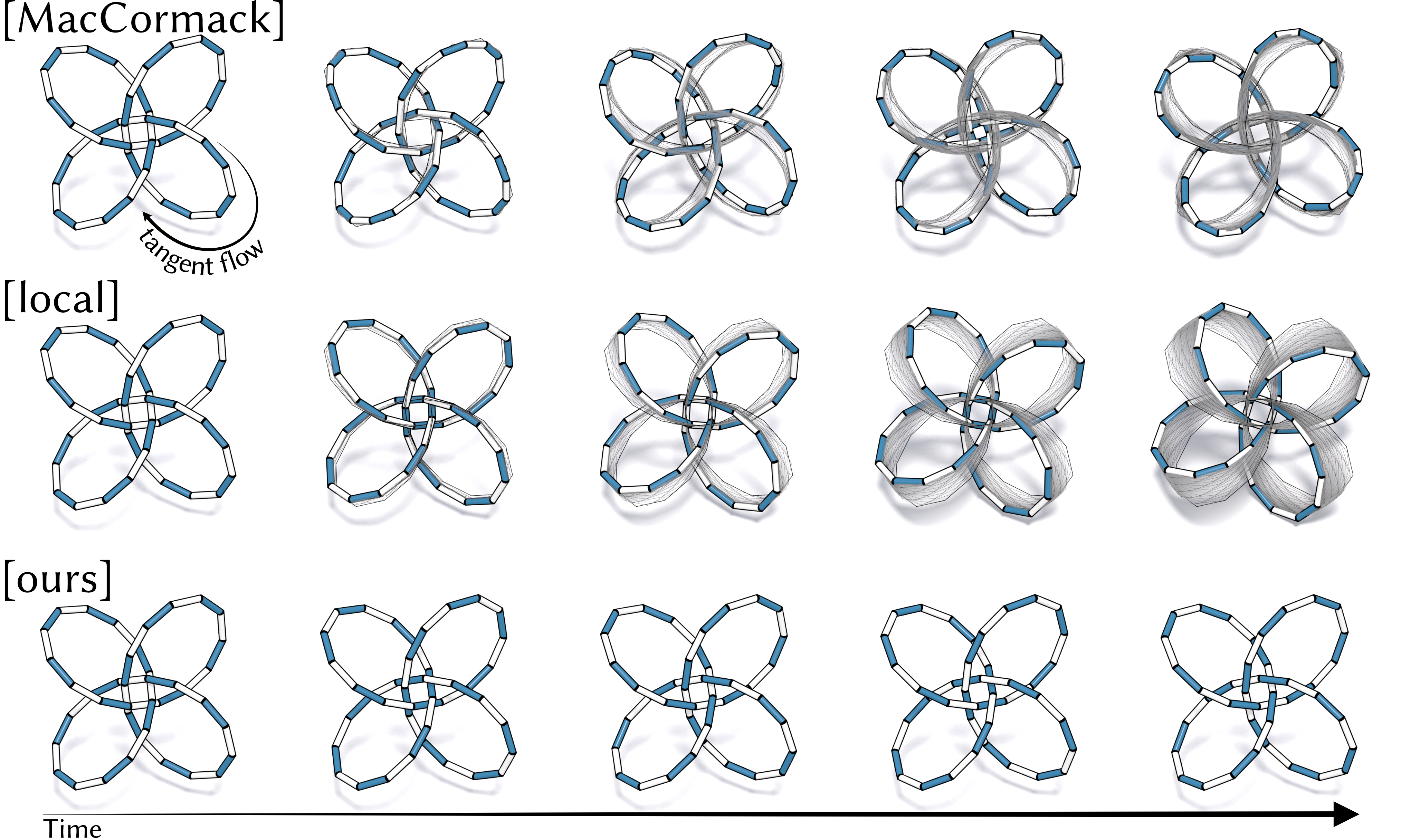}
    \caption{On a low-resolution fourfold-symmetric elastica, MacCormack advection visibly deforms the curve, while both it and the local M\"obius flow accumulate drift. The Hamiltonian tangent flow preserves the shape with substantially less long-time drift.}
    \label{fig:TangentFlow_4FoldSymmetry_Comparison}
\end{figure}
Equivalently, if \(Q\) spans \(\ker\cC_\gamma\) and \(X=Qb\), then
\(b\) is a generalized eigenvector corresponding to the smallest
eigenvalue \(\nu\):
\begin{equation*}
    Q^\top\Omega^\top\sfM_\gamma^{-1}\Omega Q\,b
    =
    \nu\,Q^\top\sfM_\gamma Q\,b .
\end{equation*}
We normalize \(X=Qb\) by \(\|X\|_{\sfM_\gamma}=1\) and fix its sign by
positive average alignment with the oriented polygon.

The construction requires no vertex tangents and is global, since
\(X_0\) depends on the full matrix \(\Omega\), the edge-length
constraints, and the chosen metric rather than on a local vertex stencil.

\subsubsection{Vortex-filament flow}
For the length Hamiltonian, \teqref{eq:DefHamitlonianFlowDiscrete}
becomes
\begin{equation}
    \label{eq:DefHamiltonianFlowLengthDiscrete}
    \OmegaMWD_\gamma X_{\sfL}=d\sfL_\gamma .
\end{equation}
The near-kernel \(\cR_\gamma\) makes direct inversion ill-conditioned,
so we solve on its spectral complement.

Let \(\Omega\) denote the matrix of \(\OmegaMWD_\gamma\), and let
\(\sfM_\gamma=K_\gamma K_\gamma^\top\) be a Cholesky factorization.
In mass coordinates, write
\begin{equation*}
    K_\gamma^{-1}\Omega K_\gamma^{-\top}
    =
    U\Sigma V^\top,
\end{equation*}
with singular values in increasing order. Discarding the first \(N\)
modes and denoting the retained blocks by
\(U^\perp\), \(\Sigma^\perp\), and \(V^\perp\) gives the quotient
vortex-filament field
\begin{equation*}
    X^\perp_{\sfL}
    =
    K_\gamma^{-\top}V^\perp
    (\Sigma^\perp)^{-1}
    (U^\perp)^\top
    K_\gamma^{-1}d\sfL_\gamma .
\end{equation*}
Here \((\cdot)^\perp\) denotes removal of \(\cR_\gamma\), not a
pointwise normal projection.

\begin{figure}[t]
    \centering
    \includegraphics[width=\columnwidth]
    {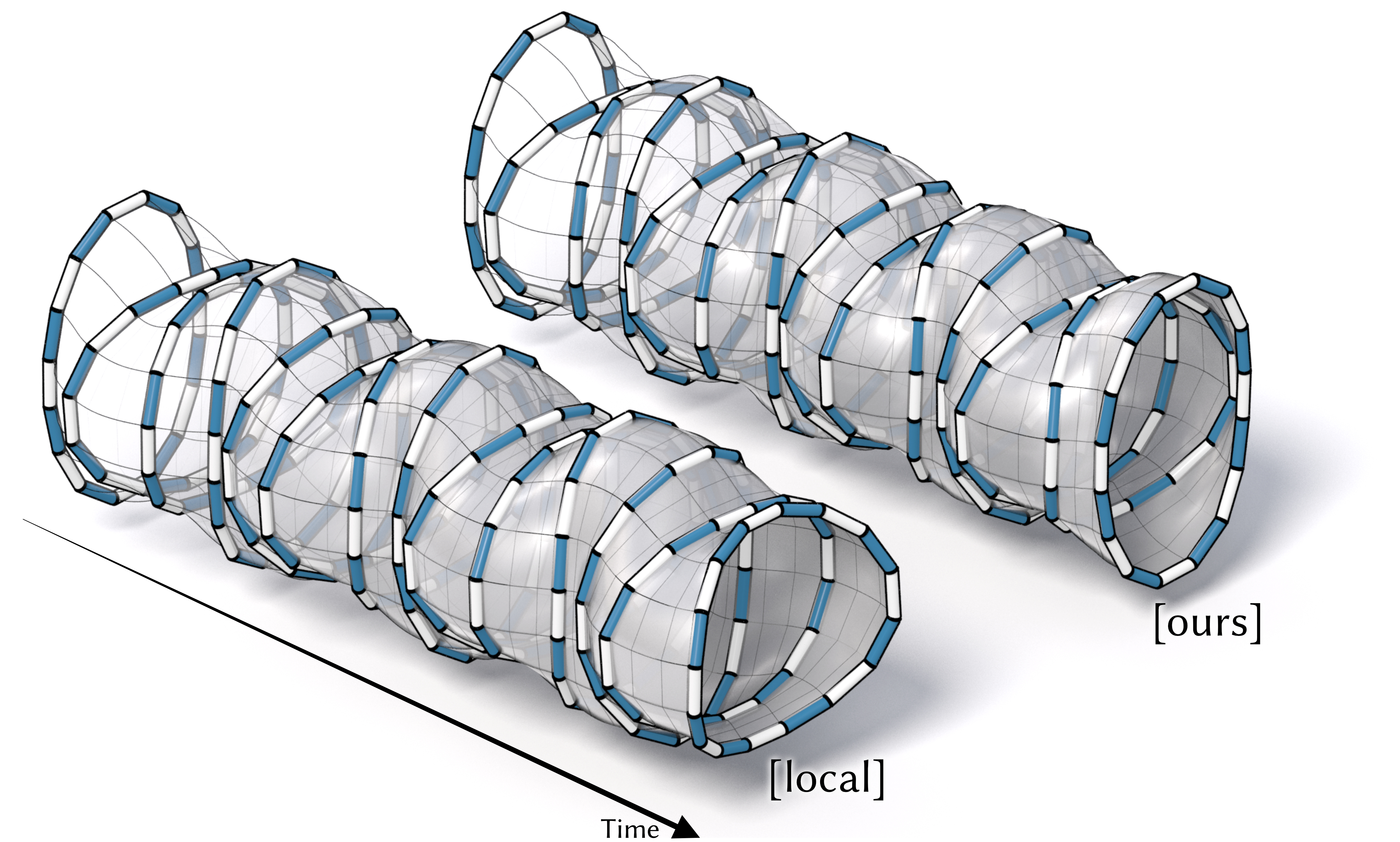}
    \caption{Evolution of an initially elliptical polygon under the vortex-filament flow. The motion of the Hamiltonian flow exhibits the expected binormal-curvature behavior and remains consistent with the local integrable model.}
    \label{fig:SmokeRingFlowEllipseComparison}
\end{figure}

This defines a discrete smoke-ring flow on polygonal shape space.
Removing \(\cR_\gamma\) suppresses reparametrization drift and isolates the shape-changing component up to the near-kernel residual. Like the
tangent flow, the construction is global: the velocity depends on the
full mass-coordinate matrix and its singular modes, rather than on a
local vertex stencil.

For visualization and time stepping, \(X^\perp_{\sfL}\) may be
normalized to unit \(\sfM_\gamma\)-norm. This changes only the time
parametrization, not the trajectory in shape space.

\subsubsection{mKdV flow}
\label{sec:mKdVFlow}
The holonomy \(\Phi\) is the last Hamiltonian in the hierarchy naturally
defined on polygons. Its smooth Hamiltonian field is the geometric mKdV
flow, the highest-order flow considered here. Following
\citet{Bergou:2008:DER}, we define polygonal holonomy by discrete parallel
transport. Assuming adjacent edges are not antiparallel, a normal vector
is transported across each vertex by the minimal rotation mapping the
incoming unit edge direction to the outgoing one. Equivalently, this is Levi-Civita parallel transport along the
shorter great-circle arc joining the two edge directions on \(S^2\),
making it the canonical discrete counterpart of Bishop transport. Composing these rotations
and comparing the transported normal with the terminal reference normal
defines
\(
    \Phi\in\RR/2\pi\ZZ.
\)
For closed polygons, the endpoint reference normals are identified. For
quasi-periodic polygons, the terminal reference is obtained from the initial
one by the rotational part of the monodromy. Changing the references shifts
\(\Phi\) only by a constant, so \(d\Phi_\gamma\) is
reference-independent~\cite{Pinkall:2024:DGF}.

We define the discrete mKdV field by the edge-length-preserving weighted
least-squares problem
\begin{equation*}
    X_\Phi
    \coloneqq
    \argmin_{X\in\ker\cC_\gamma}
    \left\|
        \OmegaMWD_\gamma X-d\Phi_\gamma
    \right\|_{\sfM_\gamma^{-1}}^2 .
\end{equation*}
Using a basis of \(\ker\cC_\gamma\), we eliminate the constraint and solve
the reduced weighted least-squares problem directly, without forming normal
equations. For integration, we normalize \(X_\Phi\) to unit
\(\sfM_\gamma\)-norm, fixing its speed without changing its direction.

\subsubsection{Discussion}
In their classical PDE form, the tangent, vortex-filament, and mKdV flows depend explicitly on derivatives of the curve, up to third order in the mKdV case. These derivatives have no canonical meaning on polygons, making the flows difficult to discretize and simulate directly. The Hamiltonian formulation bypasses this obstacle because the relevant Hamiltonians have natural polygonal counterparts. The same discrete Marsden--Weinstein structure that determines the isoperimetric constraints for discrete elastica in \secref{sec:Discretization} now converts the Hamiltonian differentials into equations for the corresponding polygonal vector fields.

\section{Results and Validation}
\label{sec:ResultsValidation}
We evaluate the discrete theory developed in
\secref{sec:Discretization} from three complementary perspectives.
First, we assess its qualitative fidelity by computing discrete elastica across a range of geometric settings and resolutions. Second, we quantitatively test the consistency of the variational formulation and the discrete Marsden--Weinstein structure and benchmark its
performance against a baseline formulation. Third, we explore its connections to dynamical systems through
Kirchhoff's analogy and Hamiltonian curve flows, as well as to surface
theory.

\subsection{Comparison of Optimization Approaches}
\label{sec:ResultsDiscreteElastica}
The isoperimetric formulation computes both closed and quasi-periodic discrete elastica without requiring a curvature discretization, material frame, rod energy, or underlying integrable system. Figures~\ref{fig:Teaser}, \ref{fig:RigidMotionToElastica},
\ref{fig:QuasiPeriodicFundamentalDomainsPureTranslation},
\ref{fig:QuasiPeriodicFundamentalDomains},
\ref{fig:ProjectedAreas}, \ref{fig:ResolutionAblationClosed},
\ref{fig:TangentFlow_4FoldSymmetry_Comparison},
\ref{fig:SmokeRingFlow_ResolutionAblation},
and~\ref{fig:WillmoreSpheres} span a range of resolutions and monodromy types. Since the monodromy enters only through the constraint map, the same augmented Lagrangian solver handles both closed and quasi-periodic curves, with the fundamental-domain vertex positions as optimization variables.

All experiments were implemented in~\textsc{Matlab}~\cite{Matlab:2025:OT}
and run on a MacBook Pro with an Apple M1 Max processor. 

\paragraph{Convergence} We observe convergence under refinement at two levels. First, our curvature-free discrete elastica converge quadratically to the smooth reference solution, as expected for piecewise-linear finite elements (\figref{fig:ConvergenceAndConditioning}). Second, for polygonal curve theory in general, the eigenvalues of the \(L^2\)-representative
\(\sfJ_\gamma\) of the discrete Marsden--Weinstein form separate into \(N\)
eigenvalues near \(0\) and two normal branches approaching \(\pm i\)
(\figref{fig:SpectrumJ}).

\begin{figure}[b]
    \centering
    \includegraphics[width=\linewidth]
    {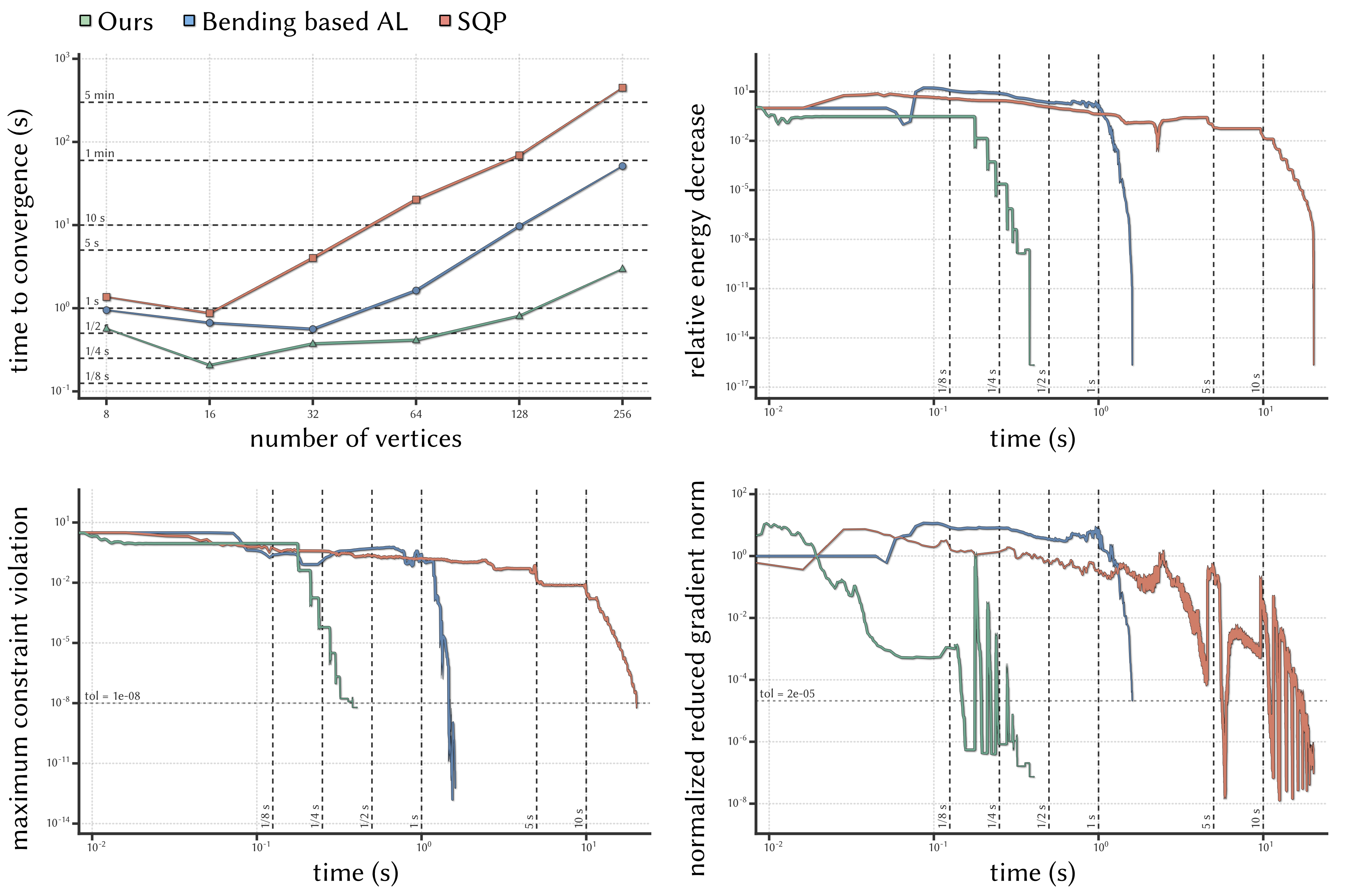}
    \caption{The isoperimetric formulation is solved with an augmented
    Lagrangian method and compared to bending-energy baselines. The plots
    report runtime under refinement, as well as energy decrease,
    constraint violation, and first-order optimality for an example with
    \(N=64\) vertices.}
    \label{fig:ConvergeceComparison}
\end{figure}
\paragraph{Performance.}
The computational advantage of the isoperimetric
formulation is already visible in its conditioning
(\figref{fig:ConvergenceAndConditioning}), which compares the reduced
Hessians of our formulation and the bending-energy baseline on the same
closed elastica under refinement. Here, reduced means the second variation
restricted to the constraint tangent space, with the remaining degenerate
directions removed. No solver parameters or tolerances enter, so the
comparison isolates the two variational formulations.

\figref{fig:ConvergeceComparison} reports the runtime under refinement of
\algref{alg:SolveDiscreteIsoperimetricProblem}, together with energy
decrease, constraint violation, and first-order optimality over the
runtime. We compare our augmented Lagrangian solver for the isoperimetric
formulation against two baselines based on the classical bending-energy
formulation. For each resolution \(N\), all methods are initialized
independently from the same circle with small random perturbations of the
vertex positions, use identical solver tolerances, and are supplied with
analytic gradients for the objective and all constraints. Thus, every
resolution is solved from scratch rather than warm-started from a
preceding solution. The isoperimetric formulation minimizes \(\sfL\) at
prescribed \(\sfA=\sfV=(0,0,\pi)\), with convergence histories shown for
\(N=64\) and \(g=\id\). For the classical formulation, we discretize
bending energy following the discrete elastic rods model of
\citet{Bergou:2008:DER} and impose matching constraints on length and
holonomy \(\Phi=\pi\), chosen so that the same configuration is
stationary. We solve this problem using either an analogous augmented
Lagrangian method or the SQP solver provided by Matlab's \texttt{fmincon}.

Even with a standard augmented Lagrangian
implementation~\cite{Nocedal:2006:NO}, the isoperimetric formulation
shows favorable convergence and scaling under refinement, consistent
with its lower-order variational structure and the absence of curvature
estimates. Under this uniform cold-start protocol, all but one of the
reported isoperimetric solves finish in under one second, amounting to
near-real-time performance for these problem sizes even without warm
starts. A faster implementation instead warm-starts each solve from the preceding solution, yielding substantially shorter runtimes in practice. These results are obtained without explicitly removing all
gauge freedom, leaving residual indeterminacies that can complicate the
numerical solve.

\begin{figure}[b]
    \centering
    \includegraphics[width=\columnwidth]{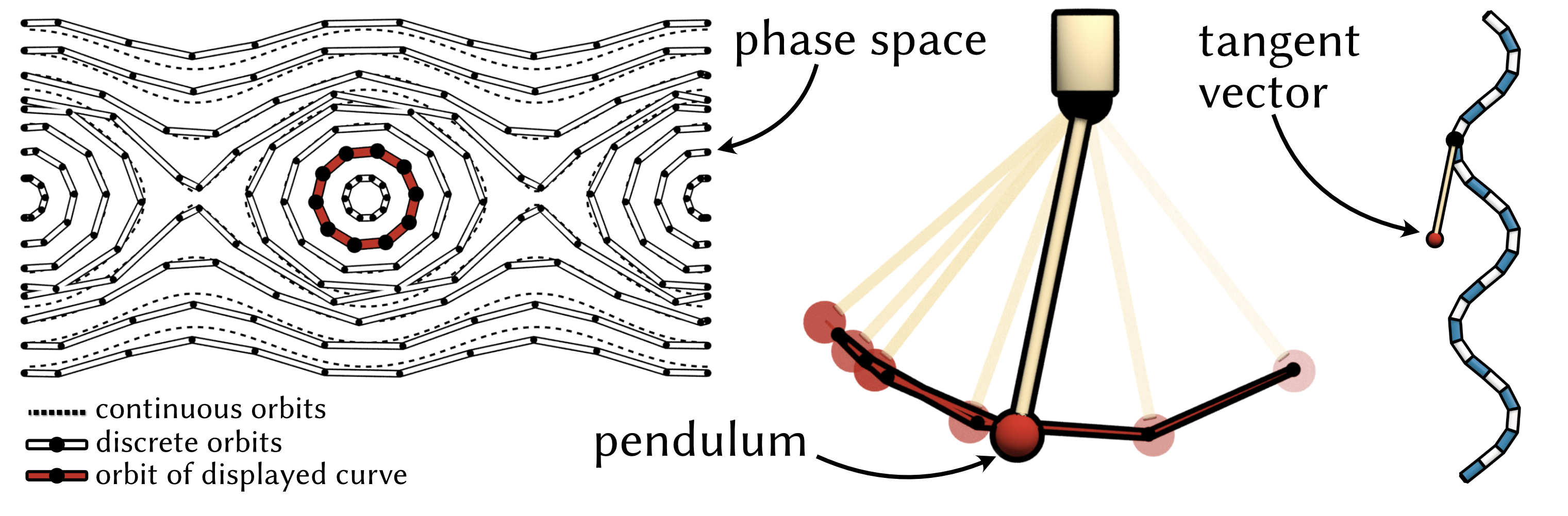}
    \caption{Solving the discrete isoperimetric problem with prescribed quasi-periodicity recovers the discrete Kirchhoff picture without integrating a pendulum equation. The edge tangents form a discrete orbit in the corresponding pendulum phase portrait.}
    \label{fig:PendulumPhaseSpace}
\end{figure}
\subsection{Connections to dynamical systems}
\label{sec:ResultsDynamicalSystems}
As throughout the paper, elastica are closely connected to several classical dynamical systems. Specifically, we use the Kirchhoff analogy to construct discrete pendulum dynamics from our discrete elastica and illustrate Hamiltonian curve flows generated by the discrete Marsden--Weinstein form.

\subsubsection{Kirchhoff analogy}
\label{sec:ResultsKirchhoffAnalogy}
The classical Kirchhoff analogy identifies elastic curves with spinning-top dynamics, whose reduced motion is described by a pendulum equation. Figs.~\ref{fig:ElasticFromPendulum}~and~\ref{fig:PendulumPhaseSpace} show that this picture remains visible
 in our discrete construction. The curves are obtained by solving the discrete isoperimetric problem with prescribed quasi-periodicity. Their edge tangents form a discrete sequence on the unit sphere whose phase portrait follows the corresponding pendulum orbit. The monodromy enforces exact closure or quasi-periodicity of this discrete orbit, so the phase portrait has no drift from time stepping. Thus the pendulum structure is recovered directly from the polygonal elastica solve, without a separate pendulum integrator.

\subsubsection{Hamiltonian flows}
\label{sec:ResultsHamiltonianFlows}
We compare the Hamiltonian flows induced by the discrete Marsden--Weinstein form to local and sampling-based flow discretizations. 
\begin{figure}[t]
    \centering
    \includegraphics[width=\columnwidth]{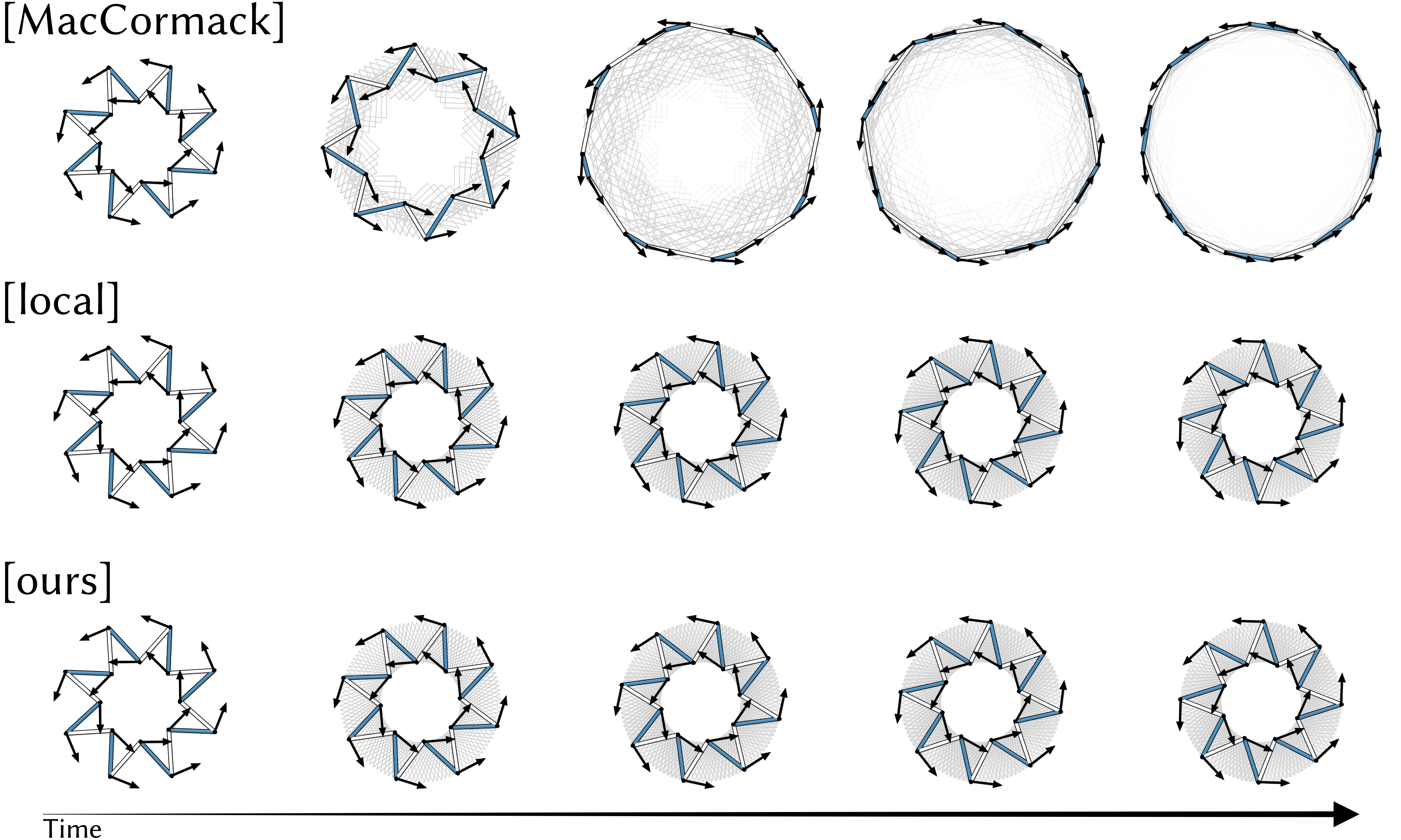}
    \caption{On a planar star-shaped curve, the Hamiltonian tangent flow agrees closely with the local M\"obius flow and preserves the shape, while MacCormack advection visibly deforms the curve over time.}
    \label{fig:TangentFlow_StarSHape_Comparison}
\end{figure}

\paragraph{Tangent flow.}
For tangential motion, a direct approach treats the polygon as sampled
data and solves the advection equation
\begin{equation*}
\partial_t\gamma-\partial_x\gamma=0
\end{equation*}
on the vertex grid, for example with semi-Lagrangian,
Lax--Friedrichs, Lax--Wendroff, or MacCormack advection. These schemes
are simple, but depend on the sampling and do not define an intrinsic
polygonal flow.

A geometric alternative comes from the integrable geometry of discrete
curves~\cite{Doliwa:1999:GDC,Bobenko:2000:DTL}. For an arc-length polygon
with unit edges
\(\vec e_i=\vec\gamma_{i+1}-\vec\gamma_i\), a M\"obius-geometric tangent
flow is
\begin{equation*}
\dot{\vec\gamma}_i
=
\frac{\vec e_{i-1}+\vec e_i}
{1+\langle\vec e_{i-1},\vec e_i\rangle}.
\end{equation*}
This flow is local, geometric, and edge-length preserving, but may
exhibit drift over long time horizons.

Our tangent flow is instead selected spectrally as the
edge-length-preserving direction of smallest Marsden--Weinstein
residual. Its reparametrization-like character is encoded in the
near-kernel structure of the discrete two-form. In the
fourfold-symmetric example of
\figref{fig:TangentFlow_4FoldSymmetry_Comparison}, our flow exhibits
substantially less drift over long time scales, even at low resolution.
For the planar star-shaped curve in
\figref{fig:TangentFlow_StarSHape_Comparison}, it agrees closely with
the local M\"obius flow and keeps the shape nearly unchanged, while
MacCormack advection visibly smooths the curve over time. On the generic
planar bunny in \figref{fig:TangentFlow_PlanarBunny_Comparison}, the
Hamiltonian tangent flow again preserves the shape more faithfully over
long integrations.

\begin{figure}[t]
    \centering
    \includegraphics[width=\columnwidth]{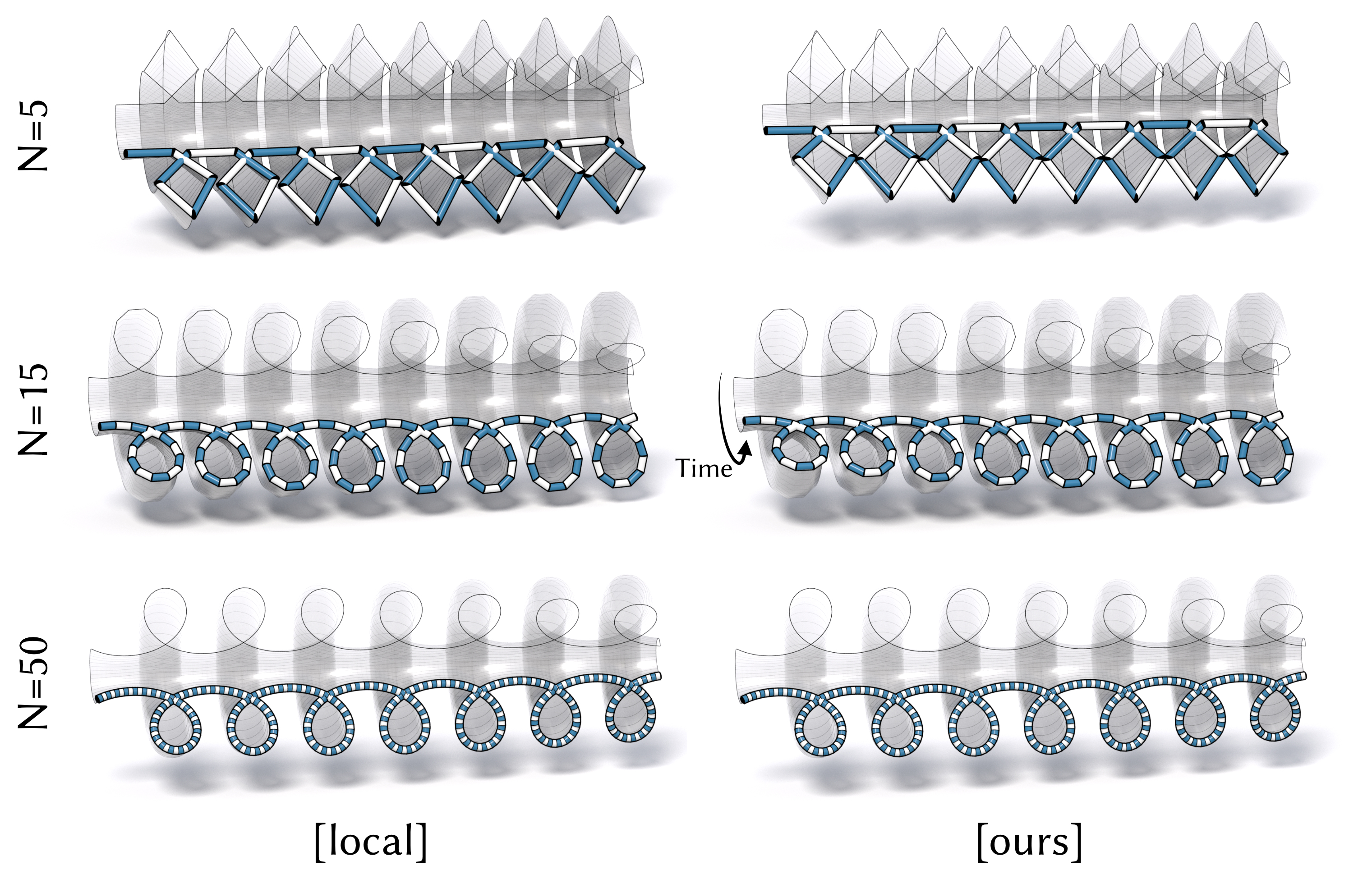}
    \caption{Since our theory is formulated for quasi-periodic polygonal curves, the Hamiltonian construction applies to both closed curves and open curves with prescribed monodromy. In particular, the evolution of open discrete elastica under the resulting vortex-filament flow is close to a rigid body motion. Since the construction depends only on the \(N\) fundamental-domain vertices, the flow can efficiently be computed on a single fundamental piece and extended indefinitely by the monodromy.}
    \label{fig:SmokeRingFlow_ResolutionAblation}
\end{figure}

\paragraph{Vortex-filament flow.}
For the vortex-filament flow, the local integrable discretization is the
discrete binormal motion
\begin{equation*}
\dot{\vec\gamma}_i
=
\frac{\vec e_i\times\vec e_{i-1}}
{1+\langle\vec e_i,\vec e_{i-1}\rangle},
\end{equation*}
again for polygons with unit edge lengths
\cite{Doliwa:1999:GDC,Bobenko:2000:DTL,Hoffmann:2008:DHS}. This flow is local and integrable, but requires arc-length normalization and becomes singular when adjacent edges are antiparallel. Doubly discrete vortex-filament models based on Darboux or B\"acklund transformations
also preserve integrability, but require a choice of transformation parameters together with appropriate closure conditions~\cite{Pinkall:2007:DSR,Hoffmann:2008:DHS}.

Our vortex-filament flow is generated by the discrete length Hamiltonian through the discrete Marsden--Weinstein form. For elastic curves, the smooth theory predicts rigid-body motion up to reparametrization. The evolution in \figref{fig:SmokeRingFlow_ResolutionAblation} stays close
to such a motion and agrees closely, at the shown scale, with the local integrable discretization. For more generic curves, the same construction exhibits the expected binormal-curvature behavior, as shown in \figref{fig:SmokeRingFlowEllipseComparison}.

\begin{figure}[b]
    \centering
    \includegraphics[width=\columnwidth]{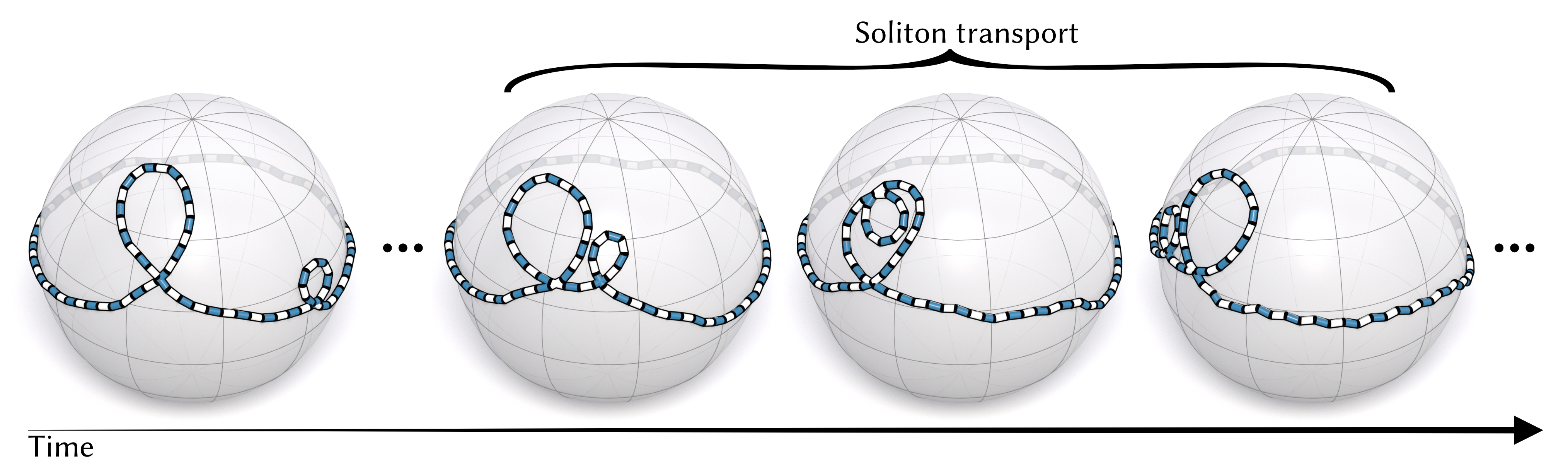}
    \caption{In the continuous theory, the mKdV flow of a spherical space curve is known to stay on the sphere and to 
    admit soliton solutions, in which a localized shape feature travels along the curve without changing form. 
    Our discrete Hamiltonian counterpart reproduces this behavior over long time horizons 
    without prescribing a separate integrable update.
    }
    \label{fig:mKdVSolitonOnSphere}
\end{figure}
\paragraph{mKdV flow}
Finally, we test two qualitative signatures of the smooth mKdV flow. The flow preserves spherical curves and admits soliton solutions that transport localized geometric features coherently~\cite{Langer:1999:RIC}. Both features are visible over long time horizons in the simulation shown in \figref{fig:mKdVSolitonOnSphere}. This is particularly striking because the smooth mKdV flow contains the third arclength derivative \(\gamma'''\), which has no classical counterpart on a polygonal curve. Nevertheless, the holonomy angle \(\Phi_\gamma\), the discrete Marsden--Weinstein form, and the edge-length constraint determine a polygonal vector field without requiring a separate discretization of
this third-order term.

These experiments highlight the tradeoff. Local formulations are explicit, efficient, and often integrable, but require flow-specific discretization choices. Our construction is global but uniform: the
Hamiltonian, the discrete Marsden--Weinstein form, and the relevant constraints determine the vector field. 

\subsubsection{Numerical consistency}
\label{sec:ResultsNumericalConsistency}
In the smooth hierarchy, the tangent and vortex-filament flows preserve \(\pzcL\), \(\pzcA\), and \(\pzcV\). For the tangent flow, this follows from repara\-metrization invariance. For the vortex-filament flow, it follows from the Hamiltonian structure and the associated Noether quantities.

\begin{wrapfigure}[13]{r}{0.435\columnwidth}
    \vspace{-.75em}
  {\hspace{-20pt}
    \centering
    \includegraphics[width=.5\columnwidth]
    {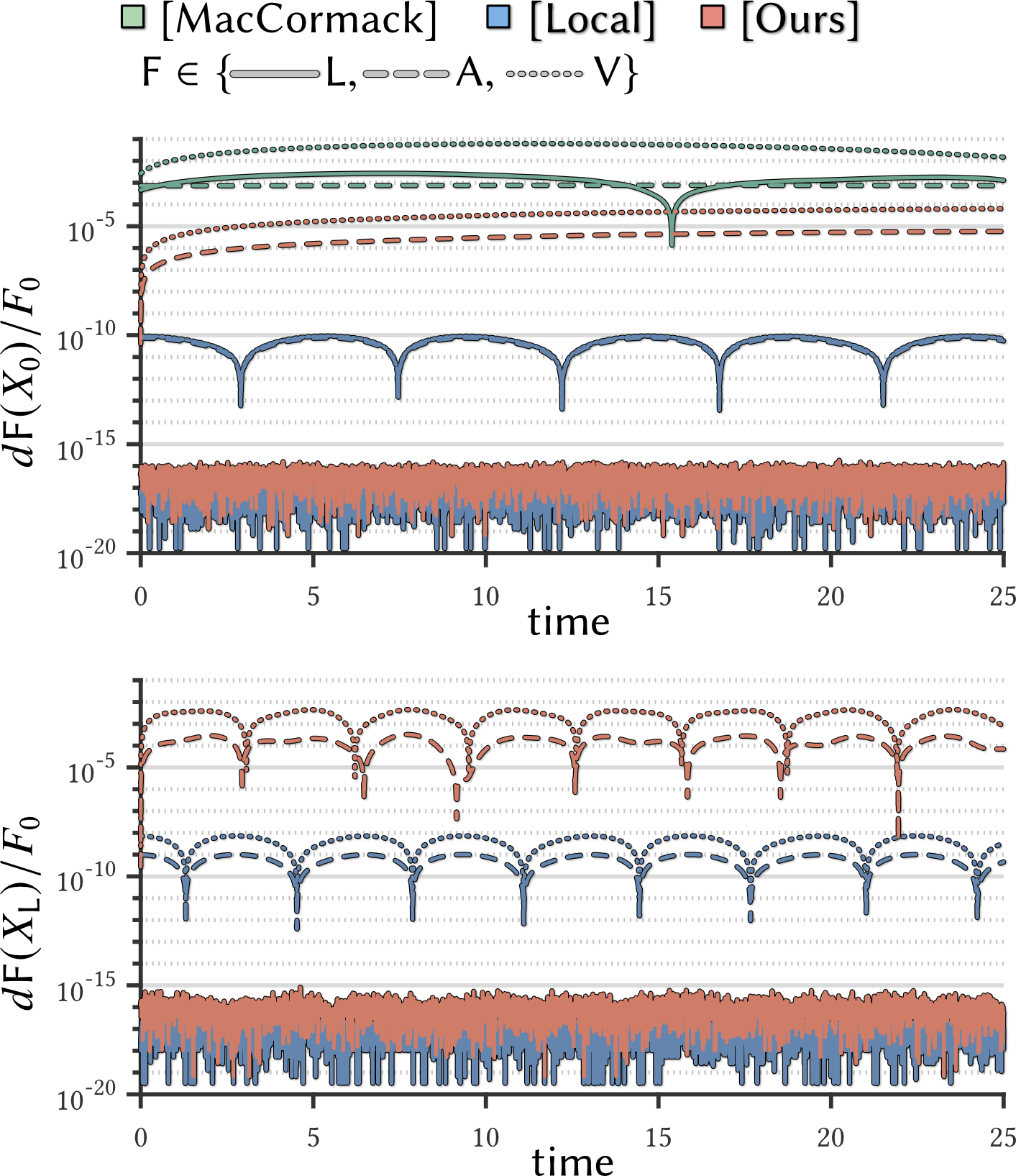}
  }
\end{wrapfigure}
Exact preservation need not survive discretization. The local M\"obius and binormal flows preserve the corresponding invariants built into their integrable constructions. For our Hamiltonian flows, the preservation defect reflects the failure of
the finite-element discretization to reproduce the exact smooth kernel.

The inset shows the normalized instantaneous rates of change of \(\sfL\), \(\sfA\), and \(\sfV\) along the tangent and vortex-filament flows. The local formulations preserve their corresponding invariants to numerical precision, while our flows exhibit small residual rates consistent with the near-kernel mechanism. For the tangent flow, MacCormack advection produces substantially larger rates of change.

\subsubsection{Limitations}
These experiments also illustrate a limitation of the Hamiltonian-flow
formulation. Unlike explicit local models for flows on discrete space
curves, the proposed vector fields require a global linear-algebra
problem involving the discrete Marsden--Weinstein form, the chosen
metric, and, where applicable, the edge-length constraints. This
provides a common framework for the different Hamiltonians considered
here, but is considerably more expensive than explicit local updates.

The resulting vector fields can also depend on how the near-kernel modes
are selected or removed, especially at low resolution or when the small
eigenvalues are tightly clustered. In our experiments, this sensitivity
decreases under refinement. We therefore view these flows primarily as
a structure-driven construction and a validation of the discrete
Marsden--Weinstein geometry. They are not intended to replace
specialized local integrable schemes when such schemes are available.\\

\begin{figure}[h]
    \centering
    \includegraphics[width = \columnwidth]{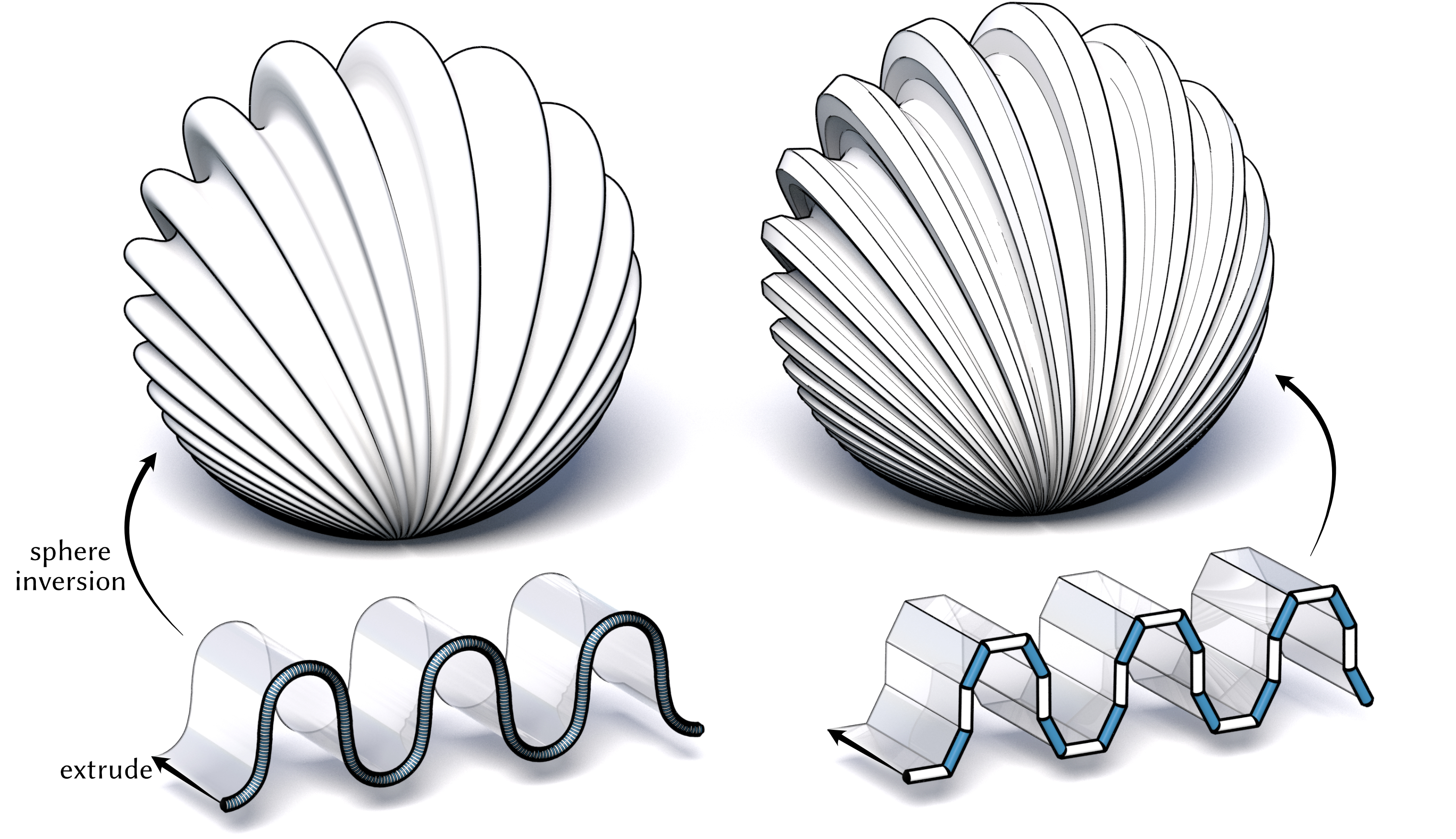}
    \caption{A cylinder over a planar free elastic curve is a Willmore surface, a property which is preserved under M\"obius transformations~\cite{Pinkall:2024:DGF}. Applying the same construction to a smooth planar free elastica and to its discrete counterpart gives surfaces with strong resemblance.}
    \label{fig:WillmoreSpheres}
\end{figure}
\subsection{Connections to surface theory}
\label{sec:ResultsSurfaceTheory}
Elastic curves are also connected to the theory of Willmore
surfaces~\cite{Pinkall:1985:HTS,Bohle:2008:CWS,Heller2014SpaceForms,
Pinkall:2024:DGF}. A cylinder over a planar free elastic curve is a Willmore surface~\cite{Pinkall:2024:DGF}.
Since the Willmore property is M\"obius invariant, a sphere inversion
maps such a cylinder to another Willmore surface.

To construct~\figref{fig:WillmoreSpheres}, we start from a smooth planar
free elastica and a discrete counterpart obtained from the
isoperimetric optimization. For each generator, we form a ruled cylinder
by translating the curve in a fixed normal direction and connecting
corresponding samples. We then apply the same sphere inversion pointwise
to the smooth and polygonal cylinders. The lower row of
\figref{fig:WillmoreSpheres} shows the two generators and their
extrusions, while the upper row shows the inverted surfaces.

Although this experiment neither defines nor evaluates a discrete
Willmore energy, the visual agreement suggests that the curve-level
discretization retains sufficient geometric structure to serve as input
for analogous (semi-)discrete surface constructions.

\section{Discussion \& Future Work}
We have presented a discrete theory of elastic curves based on the isoperimetric characterization. The definition uses only polygonal length and the monodromy-compatible Noether charges obtained from the area and volume vectors, all of which are well defined directly on polygonal curves. This yields a variational, curvature-free discretization without auxiliary choices such as material frames, rod energies, or underlying integrable systems.

A central advantage of the isoperimetric characterization is its lower differential order. Instead of minimizing a bending energy involving second derivatives and fourth-order stationarity equations, we optimize a first-order functional under low-order constraints. This produces simpler finite-dimensional constrained problems in the vertex positions and avoids estimating curvature from polygonal data. Even at very coarse resolutions, the polygonal solutions exhibit the desired qualitative
characteristics and closely resemble their smooth counterparts. Under refinement they show the expected quadratic convergence for piecewise-linear elements, and the spectrum of the discrete Marsden--Weinstein form splits into \(N\) eigenvalues near \(0\) and two branches approaching \(\pm i\).

The construction also retains several classical interpretations of elastica. The Kirchhoff analogy appears through the discrete tangent orbits, and the Marsden--Weinstein form generates Hamiltonian polygonal analogues of the tangent, vortex-filament, and mKdV flows. M\"obius images of cylinders over planar free elastica give visually consistent polygonal analogues of Willmore surface constructions. Thus the same discretization keeps the variational, mechanical, Hamiltonian, and surface-theoretic viewpoints connected.

Several questions remain. The splitting of the spectrum of \(\sfJ_\gamma\) into \(N\) eigenvalues near \(0\) and two branches approaching \(\pm i\)
awaits a rigorous proof, together with its dependence on element order, sampling, and mass metric. Such a result would in turn control the convergence of the discrete Hamiltonian vector fields built on this spectral gap. 
The discrete Marsden--Weinstein form may support Hamiltonian constructions beyond the low-order flows considered here. Finally, the surface examples suggest a discrete Willmore theory whose variational meaning remains to be made precise.

\bibliographystyle{ACM-Reference-Format}
\bibliography{Reference}


\appendix
 \section{Normal Form of the Monodromy}
 \label{app:NormalForm}

\begin{lemma}
\label{lem:NormalForm}
For any \(g=(R,\vec b)\in\SE(3)\) with \(R\neq\Id\), there exists
\(\vec c\in\RR^3\) such that, in the shifted coordinates
\(\tilde{\vec x}=\vec x-\vec c\), the monodromy takes the form
\(\tilde{\vec x}\mapsto R\tilde{\vec x}+\tilde{\vec b}\), where
    \begin{equation}
        \begin{cases}
            R\tilde{\vec b}=\tilde{\vec b}, & \text{in }3D,\\
            \tilde{\vec b}=0, & \text{in }2D.
        \end{cases}
    \end{equation}
\end{lemma}
 
\begin{proof}
Setting \(\tilde{\vec b}=R\vec c+\vec b-\vec c\) and requiring \((R-\Id)\tilde{\vec b}=0\) gives \((R-\Id)^2
\vec c=-(R-\Id)\vec b\). For \(R\in\SO(3)\) with \(R\neq \Id\), the rotation axis spans
\(\ker(R-\Id)\), so \(\im(R-\Id)^2=\im(R-\Id)\). Since \(-(R-\Id)\vec b\in\im(R-\Id)\), a solution
\(\vec c\) exists, unique up to translation along the axis. In 2D, \(R-\Id\) is invertible
and \(\vec c=-(R-\Id)^{-1}\vec b\) is the unique solution.
\end{proof}

\section{Closed Forms for Axial Components}
\label{sec:AxialClosedForms}
When \(g\neq\id\) we write \(\vec b=\ell\vec s_g\) with \(R\vec s_g=\vec s_g\), and \(\vec\gamma^\perp=\vec\gamma-\langle\vec s_g,\vec\gamma\rangle\vec s_g\). The components of the completed vectors along the screw axis admit simple closed forms, on which several proofs below rely.

\begin{lemma}
\label{lem:AxialClosedForms}
For \(\gamma\in\cM_g\) with \(g\neq\id\), evaluated on any fundamental domain \(I=[t_0,\tau(t_0)]\),
\begin{equation*}
    \langle\vec s_g,\pzcA(\gamma)\rangle
    =\tfrac12\int_I\langle\vec s_g,\vec\gamma\times\vec\gamma'\rangle\,ds,
    \quad
    \langle\vec s_g,\pzcV(\gamma)\rangle
    =-\tfrac12\int_I|\vec\gamma^\perp|^2\,\langle\vec s_g,\vec\gamma'\rangle\,ds .
\end{equation*}
\end{lemma}
\begin{proof}
For the area, the endpoint completion in \teqref{eq:AreaVectorCompleted} contributes \(-\tfrac12\langle\vec s_g,\vec b\times\vec\gamma(t_0)\rangle=0\) because \(\vec b\parallel\vec s_g\), which leaves the stated integral.

For the volume, \(\vec\gamma\times(\vec\gamma\times\vec\gamma')=\vec\gamma\langle\vec\gamma,\vec\gamma'\rangle-\vec\gamma'|\vec\gamma|^2\) and \(|\vec\gamma^\perp|^2=|\vec\gamma|^2-\langle\vec s_g,\vec\gamma\rangle^2\) give the pointwise identity
\begin{equation*}
    \tfrac13\langle\vec s_g,\vec\gamma\times(\vec\gamma\times\vec\gamma')\rangle
    =-\tfrac12|\vec\gamma^\perp|^2\,\langle\vec s_g,\vec\gamma'\rangle
    +\tfrac16\tfrac{d}{ds}\!\bigl(\langle\vec s_g,\vec\gamma\rangle\,|\vec\gamma^\perp|^2\bigr).
\end{equation*}
Integrating the total derivative over the fundamental domain \(I\), \(\tfrac16\langle\vec s_g,\vec\gamma\rangle\,|\vec\gamma^\perp|^2\big\vert_{t_0}^{\tau(t_0)}=\tfrac16\ell\,|\vec\gamma^\perp(t_0)|^2\), where we use \(\vec\gamma^\perp(\tau(t_0))=R\vec\gamma^\perp(t_0)\) and \(\langle\vec s_g,\vec\gamma(\tau(t_0))\rangle=\langle\vec s_g,\vec\gamma(t_0)\rangle+\ell\). The endpoint completion in \teqref{eq:VolumeVectorCompleted} cancels it exactly, since \(\langle\vec s_g,\vec\gamma(t_0)\times(\vec b\times\vec\gamma(t_0))\rangle=\ell\,|\vec\gamma^\perp(t_0)|^2\), leaving the stated integral.
\end{proof}

\section{Proof of \lemref{thm:AreaFundamentalDomainIndependent}}
\label{sec:AreaTransformation}
\begin{proof}
Since \(\tau\) is a translation, the integral of a \(\tau\)-periodic density over a fundamental domain \([t_0,\tau(t_0)]\) is independent of \(t_0\). We use this fact repeatedly, together with \(\vec\gamma'(\tau(t_0))=R\vec\gamma'(t_0)\).

If \(g=\id\), then \(\vec b=0\), there is no completion, and \(\pzcA(\gamma)=\tfrac12\oint\vec\gamma\times\vec\gamma'\) has a periodic integrand. Thus the full vector is independent of \(t_0\).

If \(R=\Id\) and \(\vec b\neq0\), differentiating the completed vector gives
\begin{equation*}
    \tfrac{d}{dt_0}\pzcA(\gamma)
    =\tfrac12\bigl(\vec\gamma(\tau(t_0))-\vec\gamma(t_0)\bigr)\times\vec\gamma'(t_0)
    -\tfrac12\vec b\times\vec\gamma'(t_0)
    =0,
\end{equation*}
where we use \(\vec\gamma(\tau(t_0))-\vec\gamma(t_0)=\vec b\). Thus the full vector is again independent of \(t_0\).

If \(R\neq \Id\), choose the origin on the screw axis. By \lemref{lem:AxialClosedForms}, \(\langle\vec s_g,\pzcA(\gamma)\rangle=\tfrac12\int_I\langle\vec s_g,\vec\gamma\times\vec\gamma'\rangle\,ds\). Its integrand is \(\tau\)-periodic, since \(\langle\vec s_g,\vec\gamma(\tau(t_0))\times\vec\gamma'(\tau(t_0))\rangle
=\langle\vec s_g,\vec\gamma\times\vec\gamma'\rangle\), where we use that \(R^\top\vec s_g=\vec s_g\) and \(\vec b\parallel\vec s_g\). Hence the axial component is independent of \(t_0\).
\end{proof}

\section{Proof of \lemref{thm:VolumeFundamentalDomainIndependent}}
\label{sec:VolumeTransformation}
\begin{proof}
As in \appref{sec:AreaTransformation}, the integral of a \(\tau\)-periodic density over a fundamental domain is independent of \(t_0\).

If \(g=\id\), then \(\vec b=0\), the completion vanishes, and \(\pzcV(\gamma)=\tfrac13\oint\vec\gamma\times(\vec\gamma\times\vec\gamma')\) has a periodic integrand. Thus the full vector is independent of \(t_0\).

If \(g\neq\id\), choose \(\vec s_g\) as in \lemref{lem:AxialClosedForms} (with the origin on the screw axis when \(R\neq \Id\)). By that lemma \(\langle\vec s_g,\pzcV(\gamma)\rangle=-\tfrac12\int_I|\vec\gamma^\perp|^2\,\langle\vec s_g,\vec\gamma'\rangle\,ds\). Its integrand is \(\tau\)-periodic, since \(|\vec\gamma^\perp(\tau(t_0))|^2=|R\vec\gamma^\perp(t_0)|^2=|\vec\gamma^\perp(t_0)|^2\) and \(\langle\vec s_g,\vec\gamma'(\tau(t_0))\rangle=\langle\vec s_g,R\vec\gamma'(t_0)\rangle=\langle\vec s_g,\vec\gamma'(t_0)\rangle\). Hence the axial component is independent of \(t_0\).
\end{proof}

\section{Proof of \thmref{thm:Isoperimetric}}
\label{app:IsoperimetricProof}

\begin{lemma}
\label{lem:AreaVolumeGradients}
Let \(\gamma\in\cM_g\), let \(I=[t_0,\tau(t_0)]\) be a fundamental domain, and let \(\mathring\gamma\) be a variation compactly supported in the interior of \(I\). Then, for every \(\vec\lambda_1,\vec\lambda_2\in\RR^3\),
\begin{align*}
    d\pzcL(\mathring\gamma)
    &=
    -\int_I
    \langle\vec\gamma'',\mathring{\vec\gamma}\rangle\,ds, \\
    d\langle\pzcV\,\vert\,\vec\lambda_1\rangle(\mathring\gamma)
    &=
    \int_I
    \langle(\vec\gamma\times\vec\lambda_1)\times\vec\gamma',
    \mathring{\vec\gamma}\rangle\,ds, \\
    d\langle\pzcA\,\vert\,\vec\lambda_2\rangle(\mathring\gamma)
    &=
    \int_I
    \langle\vec\gamma'\times\vec\lambda_2,\mathring{\vec\gamma}\rangle\,ds.
\end{align*}
\end{lemma}

\begin{proof}
Since \(\mathring\gamma\) is compactly supported in the interior of \(I\), the endpoint corrections in \teqref{eq:AreaVectorCompleted} and \teqref{eq:VolumeVectorCompleted} have zero first variation. It therefore suffices to differentiate the integral terms.

The length variation is
\begin{equation*}
    d\pzcL(\mathring\gamma)
    =
    \int_I\langle\vec\gamma',\mathring{\vec\gamma}'\rangle\,ds
    =
    -\int_I\langle\vec\gamma'',\mathring{\vec\gamma}\rangle\,ds,
\end{equation*}
where the boundary term vanishes because \(\mathring\gamma\) is compactly supported.

For the area component,
\begin{align*}
    d\langle\pzcA\,\vert\,\vec\lambda_2\rangle(\mathring\gamma)
    &=
    \tfrac12\int_I
    \langle\vec\lambda_2,
    \mathring{\vec\gamma}\times\vec\gamma'
    +
    \vec\gamma\times\mathring{\vec\gamma}'\rangle\,ds \\
    &=
    \tfrac12\int_I
    \langle\vec\lambda_2,
    \mathring{\vec\gamma}\times\vec\gamma'
    -
    \vec\gamma'\times\mathring{\vec\gamma}\rangle\,ds \\
    &=
    \int_I
    \langle\vec\gamma'\times\vec\lambda_2,
    \mathring{\vec\gamma}\rangle\,ds.
\end{align*}
For the volume component, writing
\(
    \langle\pzcV\,\vert\,\vec\lambda_1\rangle
    =
    \tfrac13\int_I
    \langle\vec\lambda_1,
    \vec\gamma\times(\vec\gamma\times\vec\gamma')\rangle\,ds,
\)
we compute
\begin{align*}
    d\langle\pzcV\,\vert\,\vec\lambda_1\rangle(\mathring\gamma)
    &=
    \tfrac13\int_I
    \langle\vec\lambda_1,
    \mathring{\vec\gamma}\times(\vec\gamma\times\vec\gamma')
    +
    \vec\gamma\times(\mathring{\vec\gamma}\times\vec\gamma')
    +
    \vec\gamma\times(\vec\gamma\times\mathring{\vec\gamma}')\rangle\,ds \\
    &=
    \tfrac13\int_I
    \langle
    \langle\vec\gamma,\vec\gamma'\rangle\vec\lambda_1
    +
    \langle\vec\lambda_1,\vec\gamma\rangle\vec\gamma'
    -
    2\langle\vec\lambda_1,\vec\gamma'\rangle\vec\gamma,
    \mathring{\vec\gamma}\rangle\,ds \\
    &\quad+
    \tfrac13\int_I
    \langle
    \langle\vec\lambda_1,\vec\gamma\rangle\vec\gamma
    -
    |\vec\gamma|^2\vec\lambda_1,
    \mathring{\vec\gamma}'\rangle\,ds.
\end{align*}
Integration by parts of the last term and the bac-cab identity then give
\begin{align*}
    d\langle\pzcV\,\vert\,\vec\lambda_1\rangle(\mathring\gamma)
    &=
    \int_I
    \langle
    \langle\vec\gamma,\vec\gamma'\rangle\vec\lambda_1
    -
    \langle\vec\lambda_1,\vec\gamma'\rangle\vec\gamma,
    \mathring{\vec\gamma}\rangle\,ds \\
    &=
    \int_I
    \langle
    (\vec\gamma\times\vec\lambda_1)\times\vec\gamma',
    \mathring{\vec\gamma}\rangle\,ds,
\end{align*}
which proves the claim.
\end{proof}

\begin{lemma}
\label{lem:QuasiPeriodicFirstVariation}
Let \(\gamma\in\cM_g\), let \(I=[t_0,\tau(t_0)]\), and let \(\vec\lambda_1,\vec\lambda_2\) satisfy the conditions in \thmref{thm:Isoperimetric}. Then, for every \(\mathring\gamma\in T_\gamma\cM_g\),
\begin{align}
&d\bigl(
    \pzcL
    +
    \langle\pzcV\,\vert\,\vec\lambda_1\rangle
    +
    \langle\pzcA\,\vert\,\vec\lambda_2\rangle
\bigr)(\mathring\gamma)
\notag\\
&\qquad=
\int_I
\langle
    -\vec\gamma''
    +
    (\vec\gamma\times\vec\lambda_1)\times\vec\gamma'
    +
    \vec\gamma'\times\vec\lambda_2,
    \mathring{\vec\gamma}
\rangle\,ds.
\label{eq:QuasiPeriodicFirstVariation}
\end{align}
\end{lemma}

\begin{proof}
The interior terms are those computed in \lemref{lem:AreaVolumeGradients}. It remains to show that the boundary terms vanish. Since \(\mathring\gamma\in T_\gamma\cM_g\), we have
\(
    \mathring{\vec\gamma}(\tau(t_0))
    =
    R\mathring{\vec\gamma}(t_0)
\)
and
\(
    \vec\gamma'(\tau(t_0))
    =
    R\vec\gamma'(t_0).
\)
Thus the boundary term in the length variation is
\begin{equation*}
    \langle\vec\gamma',\mathring{\vec\gamma}\rangle
    \big\vert_{t_0}^{\tau(t_0)}
    =
    \langle R\vec\gamma'(t_0),R\mathring{\vec\gamma}(t_0)\rangle
    -
    \langle\vec\gamma'(t_0),\mathring{\vec\gamma}(t_0)\rangle
    =
    0.
\end{equation*}

For the completed area vector, the boundary contribution is
\begin{equation}
	\label{eq:CompAreaVecBdyContribution}
    \tfrac12
    \langle\vec\lambda_2,
    \vec\gamma\times\mathring{\vec\gamma}\rangle
    \big\vert_{t_0}^{\tau(t_0)}
    -
    \tfrac12
    \langle\vec\lambda_2,
    \vec b\times\mathring{\vec\gamma}(t_0)\rangle,
\end{equation}
where the second term is the variation of the endpoint correction in \teqref{eq:AreaVectorCompleted}. If \(R=\Id\), then
\begin{equation*}
    \langle\vec\lambda_2,
    \vec\gamma\times\mathring{\vec\gamma}\rangle
    \big\vert_{t_0}^{\tau(t_0)}
    =
    \langle\vec\lambda_2,
    \vec b\times\mathring{\vec\gamma}(t_0)\rangle,
\end{equation*}
so the two terms cancel. If \(R\neq \Id\), then \(\vec\lambda_2\) and \(\vec b\) are parallel to \(\vec s_g\), and \(R^\top\vec\lambda_2=\vec\lambda_2\). Hence \teqref{eq:CompAreaVecBdyContribution} vanishes, since
\begin{align*}
    \langle\vec\lambda_2,
    \bigl(R\vec\gamma(t_0)+\vec b\bigr)
    \times R\mathring{\vec\gamma}(t_0)\rangle 
    =
    \langle\vec\lambda_2,\vec\gamma(t_0)\times\mathring{\vec\gamma}(t_0)\rangle,
\end{align*}
where we use that \(\langle\vec\lambda_2,\vec b\times \mathring{\vec\gamma}(t_0)\rangle=\langle\vec\lambda_2,\vec b\times R\mathring{\vec\gamma}(t_0)\rangle=0\). 

For the completed volume vector, the boundary contribution is
\begin{align*}
&\tfrac13
\langle
    \langle\vec\lambda_1,\vec\gamma\rangle\vec\gamma
    -
    |\vec\gamma|^2\vec\lambda_1,
    \mathring{\vec\gamma}
\rangle
\big\vert_{t_0}^{\tau(t_0)}
\\
&\qquad-
\tfrac16
\langle\vec\lambda_1,
    \mathring{\vec\gamma}(t_0)
    \times\bigl(\vec b\times\vec\gamma(t_0)\bigr)
    +
    \vec\gamma(t_0)
    \times\bigl(\vec b\times\mathring{\vec\gamma}(t_0)\bigr)
\rangle.
\end{align*}
If \(g=\id\), then \(\vec b=0\), and the endpoint values of both \(\vec\gamma\) and \(\mathring{\vec\gamma}\) agree, so this contribution vanishes.

Suppose that \(g\neq\id\). Then \(\vec\lambda_1\) and \(\vec b\) are parallel to \(\vec s_g\), while \(R^\top\vec\lambda_1=\vec\lambda_1\) and \(R^\top\vec b=\vec b\). Substituting the monodromy relations into the first term and expanding gives
\begin{align*}
\langle
    \langle\vec\lambda_1,\vec\gamma\rangle\vec\gamma
    -
    |\vec\gamma|^2\vec\lambda_1,
    \mathring{\vec\gamma}
\rangle
\big\vert_{t_0}^{\tau(t_0)}
=
\langle\vec\lambda_1,\vec b\rangle
\langle\vec\gamma(t_0),\mathring{\vec\gamma}(t_0)\rangle
-
\langle\vec b,\vec\gamma(t_0)\rangle
\langle\vec\lambda_1,\mathring{\vec\gamma}(t_0)\rangle.
\end{align*}
On the other hand, the bac-cab identity gives
\begin{align*}
&\langle\vec\lambda_1,
    \mathring{\vec\gamma}(t_0)
    \times\bigl(\vec b\times\vec\gamma(t_0)\bigr)
    +
    \vec\gamma(t_0)
    \times\bigl(\vec b\times\mathring{\vec\gamma}(t_0)\bigr)
\rangle
\\
&\qquad=
2\langle\vec\lambda_1,\vec b\rangle
\langle\vec\gamma(t_0),\mathring{\vec\gamma}(t_0)\rangle
-
2\langle\vec b,\vec\gamma(t_0)\rangle
\langle\vec\lambda_1,\mathring{\vec\gamma}(t_0)\rangle.
\end{align*}
After multiplication by \(1/3\) and \(-1/6\), respectively, the two contributions cancel. This proves \teqref{eq:QuasiPeriodicFirstVariation}.
\end{proof}

\begin{lemma}
\label{lem:IsoperimetricEulerLagrange}
Let \(\gamma\in\cM_g\). Then \(\gamma\) is elastic if and only if there exist \(\vec\lambda_1,\vec\lambda_2\) satisfying the conditions in \thmref{thm:Isoperimetric} such that
\begin{equation}
\label{eq:IsoperimetricEulerLagrange}
    -\vec\gamma''
    +
    (\vec\gamma\times\vec\lambda_1)\times\vec\gamma'
    +
    \vec\gamma'\times\vec\lambda_2
    =
    0.
\end{equation}
\end{lemma}

\begin{proof}
Suppose first that \teqref{eq:IsoperimetricEulerLagrange} holds. Equivalently,
\(
    \vec\gamma''
    =
    \vec\gamma'\times
    \bigl(\vec\lambda_2-\vec\gamma\times\vec\lambda_1\bigr).
\)
Setting
\(
    \vec m
    =
    \vec\lambda_2-\vec\gamma\times\vec\lambda_1,
\)
we find
\(
    \vec\gamma''=\vec\gamma'\times\vec m
\)
and
\(
    \vec m'=-\vec\gamma'\times\vec\lambda_1.
\)
It follows that
\begin{equation*}
    \tfrac{d}{ds}\langle\vec\gamma',\vec m\rangle
    =
    \langle\vec\gamma'\times\vec m,\vec m\rangle
    -
    \langle\vec\gamma',\vec\gamma'\times\vec\lambda_1\rangle
    =
    0.
\end{equation*}
Thus \(\varrho\coloneqq\langle\vec\gamma',\vec m\rangle\) is constant. Crossing \(\vec\gamma''=\vec\gamma'\times\vec m\) with \(\vec\gamma'\) gives
\begin{align*}
    \vec\gamma'\times\vec\gamma''
    &=
    \vec\gamma'\times(\vec\gamma'\times\vec m)
    =
    -\vec m+\varrho\vec\gamma' \\
    &=
    -\vec\lambda_2
    +
    \vec\gamma\times\vec\lambda_1
    +
    \varrho\vec\gamma' \\
    &=
    (-\vec\lambda_1)\times\vec\gamma
    -
    \vec\lambda_2
    +
    \varrho\vec\gamma'.
\end{align*}
After renaming the constant vectors, this is \teqref{eq:SRFtoElasticRelation}, and therefore \(\gamma\) is elastic.

Conversely, let \(\gamma\) be elastic. Then, 
\begin{equation}
\label{eq:IsoperimetricElasticRelation}
    \vec\gamma'\times\vec\gamma''
    =
    \vec c_1\times\vec\gamma
    +
    \vec c_2
    +
    \varrho\vec\gamma'.
\end{equation}
If \(g=\id\), there are no restrictions on \(\vec c_1\) or \(\vec c_2\).

Suppose that \(R=\Id\) and \(\vec b\neq0\). The left-hand side of \teqref{eq:IsoperimetricElasticRelation}, as well as \(\vec\gamma'\), is periodic. Evaluating the equation at \(\tau(t)\) and subtracting its value at \(t\) gives \(\vec c_1\times\vec b=0\). Thus \(\vec c_1\) is parallel to \(\vec b\), while \(\vec c_2\) is unrestricted.

Finally, suppose that \(R\neq \Id\). Evaluating \teqref{eq:IsoperimetricElasticRelation} at \(\tau(t)\) and comparing it with \(R\) applied to the equation at \(t\) gives
\begin{equation}
\label{eq:ElasticMultiplierCompatibility}
    (\vec c_1-R\vec c_1)\times R\vec\gamma(t)
    +
    \vec c_1\times\vec b
    +
    \vec c_2-R\vec c_2
    =
    0.
\end{equation}
Differentiating in \(s\) yields
\(
    (\vec c_1-R\vec c_1)\times R\vec\gamma'(t)=0.
\)
Unless \(\gamma\) is a straight line, its tangent assumes two nonparallel directions, and hence \(\vec c_1=R\vec c_1\). The fixed subspace of the nontrivial rotation \(R\) is the screw axis, so \(\vec c_1\) is parallel to \(\vec s_g\). Since \(\vec b\) is also parallel to \(\vec s_g\), \teqref{eq:ElasticMultiplierCompatibility} gives \(\vec c_2=R\vec c_2\), and hence \(\vec c_2\) is parallel to \(\vec s_g\). If \(\gamma\) is a straight line, \teqref{eq:IsoperimetricEulerLagrange} holds with \(\vec\lambda_1=\vec\lambda_2=0\).

In every case, we may therefore set
\(
    \vec\lambda_1=-\vec c_1
\)
and
\(
    \vec\lambda_2=-\vec c_2.
\)
These vectors satisfy the conditions in \thmref{thm:Isoperimetric}, and reversing the preceding calculation gives \teqref{eq:IsoperimetricEulerLagrange}.
\end{proof}

\begin{proof}[Proof of \thmref{thm:Isoperimetric}]
Suppose first that condition~(i) holds. Fix a fundamental domain \(I=[t_0,\tau(t_0)]\), and let \(\mathring\gamma\) be compactly supported in its interior. Extending it by
\(
    \mathring{\vec\gamma}(\tau(t))
    =
    R\mathring{\vec\gamma}(t)
\)
gives a smooth variation in \(T_\gamma\cM_g\). Condition~(i) and \lemref{lem:QuasiPeriodicFirstVariation} therefore imply
\begin{equation*}
    \int_I
    \langle
        -\vec\gamma''
        +
        (\vec\gamma\times\vec\lambda_1)\times\vec\gamma'
        +
        \vec\gamma'\times\vec\lambda_2,
        \mathring{\vec\gamma}
    \rangle\,ds
    =
    0.
\end{equation*}
By the fundamental lemma of the calculus of variations, \teqref{eq:IsoperimetricEulerLagrange} holds in the interior of \(I\). Since \(t_0\) was arbitrary, it holds on all of \(\RR\). 
A partition of unity subordinate to a finite cover of the support by interiors of fundamental domains, together with \lemref{lem:AreaVolumeGradients}, then proves condition~(ii).

Conversely, suppose that condition~(ii) holds. Testing against variations compactly supported in the interior of arbitrary fundamental domains and applying \lemref{lem:AreaVolumeGradients} and the fundamental lemma gives \teqref{eq:IsoperimetricEulerLagrange} on all of \(\RR\). Substitution into \teqref{eq:QuasiPeriodicFirstVariation} shows that the augmented functional is stationary with respect to every variation in \(T_\gamma\cM_g\), proving condition~(i).

For each fixed pair \((\vec\lambda_1,\vec\lambda_2)\) satisfying the
stated conditions, the preceding argument shows that conditions~(i)
and~(ii) are each equivalent to
\teqref{eq:IsoperimetricEulerLagrange}. By
\lemref{lem:IsoperimetricEulerLagrange}, \(\gamma\) is elastic if and
only if this equation holds for some such pair. Hence \(\gamma\) is
elastic if and only if either, and therefore both, of conditions~(i)
and~(ii) holds for some admissible pair.
\end{proof}

\section{Proof of \thmref{thm:AreaVolumeTransformation}}
\label{sec:AreaVolumeTransformationProof}
Throughout, \(h=(Q,\vec c)\in Z(g)\) and \(\gamma\in\cM_g\) has monodromy \(g=(R,\vec b)\), evaluated on a
fundamental domain \(I=[t_0,\tau(t_0)]\). When \(g\neq\id\) we write \(\vec b=\ell\vec s_g\) with
\(R\vec s_g=\vec s_g\) (so \(\ell=|\vec b|\) when \(R=\Id\)) and
\(\vec\gamma^\perp=\vec\gamma-\langle\vec s_g,\vec\gamma\rangle\vec s_g\).

By the description of \(Z(g)\) in~\secref{sec:MonodromyCases}, every \(h=(Q,\vec c)\in Z(g)\) satisfies
\begin{equation*}
    Q\vec b=\vec b,\qquad R\vec c=\vec c ,
\end{equation*}
in each of the three cases, so it suffices to treat the two factors of
\(h=(\Id,\vec c)\circ(Q,0)\) in turn.

For the rotational factor, \(Q\circ\gamma\in\cM_g\), and since
\(Q\) commutes with the cross product and fixes \(\vec b\), we have \(\pzcA(Q\circ\gamma)=Q\pzcA(\gamma)\) and
\(\pzcV(Q\circ\gamma)=Q\pzcV(\gamma)\). 

It therefore suffices to verify the transformation laws for compatible translations, \ie, those with \(R\vec c=\vec c\). We do so for the area (\lemref{lem:AreaCovariance}) and volume (\lemref{lem:VolumeCovariance}) separately. The proof of \thmref{thm:AreaVolumeTransformation} then follows from combining the matching cases.

\begin{lemma}
\label{lem:AreaCovariance}
For \(h=(Q,\vec c)\in Z(g)\):
\begin{enumerate}[\upshape(i)]
    \item if \(g=\id\), then \(\pzcA(h\circ\gamma)=Q\pzcA(\gamma)\);
    \item if \(R=\Id\), \(\vec b\neq0\), then \(\pzcA(h\circ\gamma)=Q\pzcA(\gamma)+\vec c\times\vec b\);
    \item if \(R\neq \Id\), then \(\langle\vec s_g,\pzcA(h\circ\gamma)\rangle=\langle\vec s_g,Q\pzcA(\gamma)\rangle\).
\end{enumerate}
\end{lemma}
\begin{proof}
By the reduction it suffices to treat \(\gamma\mapsto\gamma+\vec c\), \(R\vec c=\vec c\). 
As \((\gamma+\vec c)'=\gamma'\) and the monodromy translation is unchanged,
\begin{align*}
    \pzcA(\gamma+\vec c)
    &=\tfrac12\int_I(\vec\gamma+\vec c)\times\vec\gamma'\,ds-\tfrac12\vec b\times(\vec\gamma(t_0)+\vec c)\\
    &=\pzcA(\gamma)+\tfrac12\,\vec c\times\bigl(\vec\gamma(\tau(t_0))-\vec\gamma(t_0)\bigr)-\tfrac12\vec b\times\vec c .
\end{align*}
With \(\vec\gamma(\tau(t_0))-\vec\gamma(t_0)=(R-\Id)\vec\gamma(t_0)+\vec b\) and \(\tfrac12\vec c\times\vec b-\tfrac12\vec b\times\vec c=\vec c\times\vec b\),
\begin{equation*}
    \pzcA(\gamma+\vec c)=\pzcA(\gamma)+\tfrac12\,\vec c\times(R-\Id)\vec\gamma(t_0)+\vec c\times\vec b .
\end{equation*}
If \(g=\id\) both added terms vanish. If \(R=\Id\) the first vanishes, leaving \(\vec c\times\vec b\). If \(R\neq \Id\), pairing with \(\vec s_g\) annihilates both terms: \(\langle\vec s_g,\vec c\times(R-\Id)\vec\gamma(t_0)\rangle=\langle\vec s_g\times\vec c,(R-\Id)\vec\gamma(t_0)\rangle=0\) because \(\vec c\parallel\vec s_g\), and \(\langle\vec s_g,\vec c\times\vec b\rangle=0\) because \(\vec b\parallel\vec s_g\). 
\end{proof}

\begin{lemma}
\label{lem:VolumeCovariance}
For \(h=(Q,\vec c)\in Z(g)\):
\begin{enumerate}[\upshape(i)]
    \item if \(g=\id\), then \(\pzcV(h\circ\gamma)=Q\pzcV(\gamma)+\vec c\times Q\pzcA(\gamma)\);
    \item if \(R=\Id\), \(\vec b\neq0\), then \(\langle\vec s_g,\pzcV(h\circ\gamma)\rangle=\langle\vec s_g,\,Q\pzcV(\gamma)+\vec c\times Q\pzcA(\gamma)\rangle-\tfrac12|\vec b|\,|\vec c^\perp|^2\);
    \item if \(R=\Id\), \(\vec b\neq0\), and \(\vec c^\perp\coloneqq \vec c-\langle\vec s_g,\vec c\rangle\vec s_g\), then \(\langle\vec s_g,\pzcV(h\circ\gamma)\rangle=\langle\vec s_g,Q\pzcV(\gamma)\rangle\).
\end{enumerate}
\end{lemma}
\begin{proof}
Again we reduce to the case \(\gamma\mapsto\gamma+\vec c\), \(R\vec c=\vec c\).

\paragraph{Case (i).} Here \(\vec b=0\), the completion is absent, and \(\oint\vec\gamma'=0\). Expanding \((\vec\gamma+\vec c)\times((\vec\gamma+\vec c)\times\vec\gamma')\), the term quadratic in \(\vec c\) integrates to zero, and the linear term gives, after one integration by parts,
\begin{align*}
    3\pzcV(\gamma+\vec c)
    &=3\pzcV(\gamma)
      +\oint(\langle\vec\gamma,\vec\gamma'\rangle\vec c-\langle\vec c,\vec\gamma\rangle\vec\gamma')\,ds
      +2\,\vec c\times\pzcA(\gamma)\\
    &=3\pzcV(\gamma)
      +\oint\tfrac12\langle\vec\gamma,\vec\gamma\rangle'\vec c\,ds-\oint\langle\vec c,\vec\gamma\rangle\vec\gamma'\,ds
      +2\,\vec c\times\pzcA(\gamma)\\
    &=3\pzcV(\gamma)+\tfrac12\oint\vec c\times(\vec\gamma\times\vec\gamma')\,ds+2\,\vec c\times\pzcA(\gamma)\\
    &=3\pzcV(\gamma)+\vec c\times\pzcA(\gamma)+2\,\vec c\times\pzcA(\gamma)\\
    &=3\pzcV(\gamma)+3\,\vec c\times\pzcA(\gamma).
\end{align*}
Replacing \(\gamma\) by \(Q\circ\gamma\) proves (i).

\paragraph{Cases (ii), (iii).} By \lemref{lem:AxialClosedForms}, using \(\langle\vec s_g,(\vec\gamma+\vec c)'\rangle=\langle\vec s_g,\vec\gamma'\rangle\) and \((\vec\gamma+\vec c)^\perp=\vec\gamma^\perp+\vec c^\perp\),
\begin{align*}
    \langle\vec s_g,\pzcV(\gamma+\vec c)\rangle
    &=-\tfrac12\int_I|\vec\gamma^\perp+\vec c^\perp|^2\langle\vec s_g,\vec\gamma'\rangle\,ds\\
    &=\langle\vec s_g,\pzcV(\gamma)\rangle
    \!-\!\!\!\int_I\langle\vec\gamma,\vec c^\perp\rangle\langle\vec s_g,\vec\gamma'\rangle\,ds
    -\tfrac12|\vec c^\perp|^2\!\int_I \langle\vec s_g,\vec\gamma'\rangle\,ds,\\
    &=\langle\vec s_g,\pzcV(\gamma)\rangle
    \!-\!\!\!\int_I\langle\vec\gamma,\vec c^\perp\rangle\langle\vec s_g,\vec\gamma'\rangle\,ds
    -\tfrac12\ell|\vec c^\perp|^2,
\end{align*}
where we used that \(\langle\vec\gamma^\perp,\vec c^\perp\rangle=\langle\vec\gamma,\vec c^\perp\rangle\) since \(\vec c^\perp\perp\vec s_g\), and \(\int_I \langle\vec s_g,\vec\gamma'\rangle\,ds=\langle\vec s_g,\vec b\rangle=\ell\). 

Now let \(R=\Id\), and consider the integral to be evaluated alongside the one with \(\vec s_g\) and
\(\vec c^\perp\) exchanged. The product rule for
\((\langle\vec s_g,\vec\gamma\rangle\langle\vec c^\perp,\vec\gamma\rangle)'\) gives their
sum,
\begin{align*}
    &\quad \ \int_I\langle\vec s_g,\vec\gamma\rangle\langle\vec c^\perp,\vec\gamma'\rangle\,ds
    +\int_I\langle\vec s_g,\vec\gamma'\rangle\langle\vec c^\perp,\vec\gamma\rangle\,ds
    \\ &=\langle\vec s_g,\vec\gamma\rangle\langle\vec c^\perp,\vec\gamma\rangle\big\vert_{t_0}^{\tau(t_0)}
     =\ell\,\langle\vec c^\perp,\vec\gamma(t_0)\rangle,
 \end{align*}
 where we used that \(\vec\gamma(\tau(t_0))=\vec\gamma(t_0)+\vec b\), \(\langle\vec s_g,\vec b\rangle=\ell\), and \(\langle\vec c^\perp,\vec b\rangle=0\).
 On the other hand, the Lagrange identity
\(\langle\vec s_g\times\vec c^\perp,\vec\gamma\times\vec\gamma'\rangle
=\langle\vec s_g,\vec\gamma\rangle\langle\vec c^\perp,\vec\gamma'\rangle
-\langle\vec s_g,\vec\gamma'\rangle\langle\vec c^\perp,\vec\gamma\rangle\) gives their difference. That is,
 \begin{align*}
    &\quad \ \int_I\langle\vec s_g,\vec\gamma\rangle\langle\vec c^\perp,\vec\gamma'\rangle\,ds
    -\int_I\langle\vec s_g,\vec\gamma'\rangle\langle\vec c^\perp,\vec\gamma\rangle\,ds
    \\ &=2\,\langle\vec s_g\times\vec c^\perp,\pzcA(\gamma)\rangle
     +\ell\,\langle\vec c^\perp,\vec\gamma(t_0)\rangle .
\end{align*}
where we used that \(\int_I\vec\gamma\times\vec\gamma'\,ds=2\pzcA(\gamma)+\vec b\times\vec\gamma(t_0)\).
Subtracting the second line from the first cancels the first integral and both endpoint terms, so
\begin{equation*}
    -\int_I\langle\vec\gamma,\vec c^\perp\rangle\langle\vec s_g,\vec\gamma'\rangle\,ds
    =\langle\vec s_g\times\vec c^\perp,\pzcA(\gamma)\rangle
    =\langle\vec s_g,\vec c\times\pzcA(\gamma)\rangle ,
\end{equation*}
where the last step uses that \(\langle\vec s_g,\vec s_g\times\pzcA(\gamma)\rangle=0\).

In case (iii) we have \(\vec c\parallel\vec s_g\), so \(\vec c^\perp=0\) and \(\langle\vec s_g,\vec c\times\pzcA\rangle=0\), leaving \(\langle\vec s_g,\pzcV(h\circ\gamma)\rangle=\langle\vec s_g,Q\pzcV(\gamma)\rangle\).
\end{proof}

\section{Proof of \thmref{thm:MWSymplectic}}
\label{sec:MWLWellDefinedProof}
\begin{proof}[Proof of \thmref{thm:MWSymplectic}]
To show that \(\OmegaMW\) is indeed closed, we make use of the convenient fact that the Marsden--Weinstein form is in fact exact, with the \emph{Marsden--Weinstein--Liouville form} \(\ThetaMW\in\Omega^1(\cM_g)\) as potential. In what follows we show that 
\begin{equation*}
    \OmegaMW= d\ThetaMW \in\Omega^2(\cM_g).
\end{equation*}
Since the exterior derivative of an exact form always vanishes, this implies \(\OmegaMW\) is closed. 

Infinitesimal reparametrizations have the form
\(\dot{\vec\gamma}=f\vec\gamma'\), and for such vectors
\(\det(\dot{\vec\gamma},\cdot,\vec\gamma')\equiv 0\) by alternation of the
determinant. Thus they lie in the kernel. Conversely, if
\(
\det(\dot{\vec\gamma},\mathring{\vec\gamma},\vec\gamma')
=
\langle \mathring{\vec\gamma},\vec\gamma'\times\dot{\vec\gamma}\rangle
\)
vanishes for all \(\mathring{\vec\gamma}\in T_\gamma\cM_g\), then
\(\vec\gamma'\times\dot{\vec\gamma}=0\). Hence \(\dot{\vec\gamma}\) is pointwise
parallel to \(\vec\gamma'\), so \(\dot{\vec\gamma}=f\vec\gamma'\) for some
\(\tau\)-periodic function \(f(\tau(t))=f(t)\). Therefore the kernel consists exactly of the
infinitesimal reparametrizations.
\end{proof}

The Marsden--Weinstein--Liouville form is most conveniently constructed by first choosing a primitive of the Euclidean
volume form, \ie, some \(\beta\in\Omega^2(\RR^3)\) such that \(d\beta=\mathrm{vol}\). To be well-defined on \(\cM_g\), the choice must be compatible with the monodromy of \(\cM_g\), which is why we additionally require that \(g^*\beta=\beta\) (\lemref{lem:MWLWellDefined}). Given such a \(\beta\), we define the Marsden--Weinstein--Liouville form
\(\ThetaMW\in\Omega^1(\cM_g)\) by
\begin{equation*}
\label{eq:ThetaMW}
    \ThetaMW_\gamma(\mathring\gamma)
    =
    \int_\gamma \iota_{\mathring{\vec{\gamma}}}\beta
    =
    \int_I\beta(\mathring{\vec\gamma},d\vec\gamma).
\end{equation*}

Now a direct computation verifies that indeed \(d\ThetaMW=\OmegaMW\). Let
\(\dot\gamma\) and \(\mathring\gamma\) be local commuting variation fields on
\(\cM_g\). At \(\gamma\), their pointwise representatives are the sections
\(\dot{\vec\gamma},\mathring{\vec\gamma}\in\Gamma(\gamma^*T\mathbb{R}^3)\) in the tangent bundle \(T\mathbb{R}^3\) pulled back by \(\gamma\).
With \(\nabla\) denoting the flat connection on \(\mathbb{R}^3\), applied
pointwise to these pullback sections, the bracket term in \(d\ThetaMW\)
vanishes. Then,
\begin{align*}
d\ThetaMW_\gamma(\dot\gamma,\mathring\gamma)
&=
d\big(\ThetaMW(\mathring\gamma)\big)_\gamma(\dot\gamma)
-
d\big(\ThetaMW(\dot\gamma)\big)_\gamma(\mathring\gamma) \\
&=
\int_I
\Big[
(\nabla_{\dot{\vec\gamma}}\beta)(\mathring{\vec\gamma},\vec\gamma')
+
\beta(\nabla_{\dot{\vec\gamma}}\mathring{\vec\gamma},\vec\gamma')
+
\beta(\mathring{\vec\gamma},\dot{\vec\gamma}') \\
&\hspace{2.5em}
-
(\nabla_{\mathring{\vec\gamma}}\beta)(\dot{\vec\gamma},\vec\gamma')
-
\beta(\nabla_{\mathring{\vec\gamma}}\dot{\vec\gamma},\vec\gamma')
-
\beta(\dot{\vec\gamma},\mathring{\vec\gamma}')
\Big]\,dt .
\end{align*}
Since \(\dot\gamma\) and \(\mathring\gamma\) commute,
\(\nabla_{\dot{\vec\gamma}}\mathring{\vec\gamma}
=
\nabla_{\mathring{\vec\gamma}}\dot{\vec\gamma}\), so the middle terms cancel.
Integrating the remaining \(t\)-derivative terms by parts gives
\begin{equation*}
\int_I
\big[
\beta(\mathring{\vec\gamma},\dot{\vec\gamma}')
-
\beta(\dot{\vec\gamma},\mathring{\vec\gamma}')
\big]\,dt
=
\int_I
(\nabla_{\vec\gamma'}\beta)(\dot{\vec\gamma},\mathring{\vec\gamma})\,dt,
\end{equation*}
with no boundary term because the endpoints are related by \(g\) and
\(g^*\beta=\beta\). Hence, using that \(d\beta=\mathrm{vol}=\det\), we find
\begin{align*}
d\ThetaMW_\gamma(\dot\gamma,\mathring\gamma)
&=
\int_I
\Big[
(\nabla_{\dot{\vec\gamma}}\beta)(\mathring{\vec\gamma},\vec\gamma')
-
(\nabla_{\mathring{\vec\gamma}}\beta)(\dot{\vec\gamma},\vec\gamma')
+
(\nabla_{\vec\gamma'}\beta)(\dot{\vec\gamma},\mathring{\vec\gamma})
\Big]\,dt \\
&=
\int_I d\beta(\dot{\vec\gamma},\mathring{\vec\gamma},\vec\gamma')\,dt 
=
\int_I \det(\dot{\vec\gamma},\mathring{\vec\gamma},\vec\gamma')\,dt 
=
\OmegaMW_\gamma(\dot\gamma,\mathring\gamma).
\end{align*}

\begin{lemma}
\label{lem:MWLWellDefined}
    \(\ThetaMW\) is well-defined on \(\cM_g\) if and only if \(g^*\beta=\beta\).
\end{lemma}

\begin{proof}
    Let \(\gamma\in\cM_g\) and \(\mathring\gamma\in T_\gamma\cM_g\). Well-definedness requires that \(\ThetaMW_\gamma(\mathring\gamma)\) not depend on which fundamental domain one integrates over. For \(t\in\RR\), write the potential evaluated on \([t,\tau(t)]\) as
    \begin{equation*}
        \Theta^{[t]}
        \coloneqq
        \int_{t}^{\tau(t)}
        \beta_{\vec\gamma(s)}\bigl(\mathring{\vec\gamma}(s),\vec\gamma'(s)\bigr)\,ds ,
    \end{equation*}
    so that independence amounts to \(\Theta^{[t_0]}=\Theta^{[t_1]}\) for all \(t_0,t_1\in\RR\). Since \(\tau\) is a translation, \(\tau'\equiv1\), and the Leibniz rule gives
    \begin{equation*}
        \tfrac{d}{dt}\Theta^{[t]}
        =
        \beta_{\vec\gamma(\tau(t))}\bigl(\mathring{\vec\gamma}(\tau(t)),\vec\gamma'(\tau(t))\bigr)
        -
        \beta_{\vec\gamma(t)}\bigl(\mathring{\vec\gamma}(t),\vec\gamma'(t)\bigr).
    \end{equation*}
    The quasi-periodicity \(\vec\gamma(\tau(t))=g(\vec\gamma(t))\) gives, upon differentiation in \(t\), \(\vec\gamma'(\tau(t))=dg\,\vec\gamma'(t)\) and \(\mathring{\vec\gamma}(\tau(t))=dg\,\mathring{\vec\gamma}(t)\). Hence the first term equals
    \begin{equation*}
        \beta_{g(\vec\gamma(t))}\bigl(dg\,\mathring{\vec\gamma}(t),\,dg\,\vec\gamma'(t)\bigr)
        =
        (g^*\beta)_{\vec\gamma(t)}\bigl(\mathring{\vec\gamma}(t),\vec\gamma'(t)\bigr),
    \end{equation*}
    so that
    \begin{equation*}
        \tfrac{d}{dt}\Theta^{[t]}
        =
        (g^*\beta-\beta)_{\vec\gamma(t)}\bigl(\mathring{\vec\gamma}(t),\vec\gamma'(t)\bigr).
    \end{equation*}
    This vanishes for all \(\gamma\in\cM_g\), all \(\mathring\gamma\in T_\gamma\cM_g\), and all \(t\) if and only if \(g^*\beta=\beta\). Therefore \(\Theta^{[t_0]}=\Theta^{[t_1]}\) for all \(t_0,t_1\). That is, \(\ThetaMW\) is independent of the chosen fundamental domain if and only if \(g^*\beta=\beta\).

The Marsden--Weinstein--Liouville form \(\ThetaMW\) may depend on the chosen admissible
primitive \(\beta\). Its exterior derivative does not: for every admissible
\(\beta\), the calculation above gives
\begin{equation*}
d\ThetaMW=\OmegaMW.
\end{equation*}
Consequently, the Marsden--Weinstein--Liouville forms associated with any two admissible
primitives differ by a closed one-form on \(\cM_g\), and the resulting
Marsden--Weinstein form is independent of this choice.
\end{proof}

\begin{example}
An admissible primitive can be chosen explicitly. For \(g=\id\), let
\begin{equation*}
\beta_0=\tfrac13\star\vec p^{\flat},
\qquad
\vec p(\vec x)=\vec x.
\end{equation*}
For translational or screw monodromy with axis
\(\RR\vec s_g\), define
\begin{equation*}
    \vec p_{\vec s_g}(\vec x)
    \coloneqq
    \vec x-\langle\vec x,\vec s_g\rangle\vec s_g
\end{equation*}
and set
\begin{equation*}
    \beta_{\vec s_g}
    =
    \tfrac12\star\vec p_{\vec s_g}^{\flat}.
\end{equation*}
Then \(d\beta_{\vec s_g}=\mathrm{vol}\) and
\(g^*\beta_{\vec s_g}=\beta_{\vec s_g}\).
In either case, \(d\beta=\mathrm{vol}\) and \(g^*\beta=\beta\).
\end{example}

\section{Proof of \lemref{thm:CentralizerSymplectomorphism}}
\label{sec:CentralizerSymplectomorphismProof}
\begin{proof}
    Let \(h=(Q,\vec c)\in Z(g)\). Since \(h\) commutes with \(g\),
    \begin{equation*}
        h(\gamma)(\tau(t))
        =
        h(g(\gamma(t)))
        =
        g(h(\gamma(t))),
    \end{equation*}
    so \(h\) maps \(\cM_g\) to itself. Its differential is
    \(dh(\mathring\gamma)=Q\mathring\gamma\). Hence, using
    \teqref{eq:MWForm},
    \begin{align*}
        (h^*\OmegaMW)_\gamma(\dot\gamma,\mathring\gamma)
        &=
        \OmegaMW_{h(\gamma)}
        (Q\dot\gamma,Q\mathring\gamma) \\
        &=
        \int_I
        \det\bigl(Q\dot{\vec\gamma},Q\mathring{\vec\gamma},Qd\vec\gamma\bigr)
        =
        \OmegaMW_\gamma(\dot\gamma,\mathring\gamma),
    \end{align*}
    since \(Q\in\SO(3)\). Thus \(Z(g)\) preserves \(\OmegaMW\). Moreover,
    rigid motions commute with reparametrizations, so the action descends to
    the quotient. Since \(\OmegaMW\) descends there to a symplectic form by
    \thmref{thm:MWSymplectic}, the induced \(Z(g)\)-action is symplectic.
\end{proof}

\section{Proof of \thmref{thm:MomentumMap}}
\label{sec:MomentumMapProof}
\begin{proof}
Let \((\vec\omega,\vec v)\in\mathfrak z(g)\). Its infinitesimal vector field is
\begin{equation*}
    X_{(\vec \omega,\vec v)}(\gamma)
    =
    \vec \omega\times\vec\gamma+\vec v .
\end{equation*}
By the definition of the Marsden--Weinstein form,
\begin{align*}
    \iota_{X_{(\vec \omega,\vec v)}}\OmegaMW_\gamma(\mathring\gamma)
    &=
    \int_I
    \det
    (
        \vec \omega\times\vec\gamma,
        \mathring{\vec\gamma},
        \vec\gamma'
    )
    ds
    +
    \int_I
    \det
    (
        \vec v,
        \mathring{\vec\gamma},
        \vec\gamma'
    )
    ds .
\end{align*}
Differentiating \(\pzcA\) or \(\pzcV\) against a general variation \(\mathring\gamma\in T_\gamma\cM_g\) produces boundary terms at the seam at \(t_0\) and \(\tau(t_0)\). These are canceled by the first variations of the endpoint completions in \teqref{eq:AreaVectorCompleted} and \teqref{eq:VolumeVectorCompleted}, as shown in the proof of \lemref{lem:QuasiPeriodicFirstVariation} (whose compatibility conditions on \((\vec\lambda_1,\vec\lambda_2)\) are precisely those defining \((\vec\omega,\vec v)\in\fz(g)\)). Hence, for every \(\mathring\gamma\in T_\gamma\cM_g\),

\begin{align*}
    \int_I
    \det
    (
       \vec  v,
        \mathring{\vec\gamma},
        \vec\gamma'
    )
    ds
    &=
    \int_I
    \langle
        \vec\gamma'\times \vec v,
        \mathring{\vec\gamma}
    \rangle
    ds
    =
    d\langle \vec v,\pzcA\rangle(\mathring\gamma), \\
    \int_I
    \det
    (
        \vec \omega\times\vec\gamma,
        \mathring{\vec\gamma},
        \vec\gamma'
    )
    ds
    &=
    \int_I
    \langle
        (\vec\gamma\times\vec \omega)\times\vec\gamma',
        \mathring{\vec\gamma}
    \rangle
    ds
    =
    d\langle \vec \omega,\pzcV\rangle(\mathring\gamma).
\end{align*}
Hence
\begin{equation*}
    \iota_{X_{(\vec \omega,\vec v)}}\OmegaMW_\gamma(\mathring\gamma)
    =
    d
    \langle
        ( \pzcV(\gamma),
            \pzcA(\gamma)
        )
        \vert
        (\vec \omega,\vec v)
    \rangle
    (\mathring\gamma), 
\end{equation*}
and restriction from \(\mathfrak{se}(3)\) to \(\mathfrak z(g)\) yields the claim.
\end{proof}

\section{Proof of \thmref{thm:DiscreteMomentumMap}}
\label{sec:DiscreteMomentumMapProof}
\begin{proof}
Let \(\cI\) be the piecewise-linear interpolation map. By construction,
\begin{equation*}
    \OmegaMWD=\cI^*\OmegaMW,
    \qquad
    \sfA=\cI^*\pzcA,
    \qquad
    \sfV=\cI^*\pzcV .
\end{equation*}
Let \(\xi=(\vec \omega,\vec v)\in\fz(g)\). The smooth infinitesimal rigid-motion field is 
\(
    X_\xi(\gamma)=\vec \omega\times\vec\gamma+\vec v .
\) 
Its discrete counterpart \(X_\xi^{\mathrm{disc}}(\gamma)\) is obtained by applying the same formula at the
vertices. Since this formula is affine in \(\vec\gamma\), interpolation of
the vertexwise field agrees with the smooth field along the interpolated
curve:
\begin{equation*}
    d\cI_\gamma(X_\xi^{\mathrm{disc}}(\gamma))=X_\xi(\cI(\gamma)).
\end{equation*}
Consequently, using~\thmref{thm:MomentumMap},
\begin{equation*}
    \iota_{X_\xi^{\mathrm{disc}}}\OmegaMWD
    =
    \cI^*
    \left(
        \iota_{X_\xi}\OmegaMW
    \right) 
    =
    \cI^*
    d
    \left\langle
        (\pzcV,\pzcA)\,\vert\,\xi
    \right\rangle 
    =
    d
    \left\langle
        (\sfV,\sfA)\,\vert\,\xi
    \right\rangle,
\end{equation*}
which yields the claim.
\end{proof}

\end{document}